\documentclass[a4paper]{article} 

\newcommand{\dom}{\mathrm{dom}}
\usepackage{lineno}
\usepackage{grffile}

\usepackage[a4paper]{geometry}

\usepackage[utf8]{inputenc} % allow utf-8 input
\usepackage[T1]{fontenc}    % use 8-bit T1 fonts
\usepackage{hyperref}       % hyperlinks
\usepackage{url}            % simple URL typesetting
\usepackage{booktabs}       % professional-quality tables
\usepackage{amsfonts}       % blackboard math symbols
\usepackage{nicefrac}       % compact symbols for 1/2, etc.
\usepackage{microtype}      % microtypography

\usepackage{amsmath}
\usepackage{amsthm} % PBM
\usepackage{amsfonts, amssymb, amscd}

\usepackage{mystyle}

\newcommand{\commentout}[1]{}

\usepackage{mathtools,booktabs,tabu,graphicx,hyperref,epstopdf,url,rotating}
\usepackage{algorithm}
\usepackage{cite}
\usepackage{algorithmic}
\begin{document}

\title{Smooth over-parameterized solvers\\for non-smooth structured optimization}

\author{Clarice Poon\\
	Department of Mathematical Sciences, \\ University of Bath\\
	\url{cmhsp20@bath.ac.uk}	
	\and
        Gabriel Peyr\'e \\
        CNRS and DMA, ENS, \\ PSL Universit\'e.\\
        \url{gabriel.peyre@ens.fr}
}

\date{}

\maketitle

\newcommand{\keywords}[1]{}
\newcommand{\subclass}[1]{}

% !TEX root = ../SINUM-VarPro.tex

\begin{abstract}
	Non-smooth optimization is a core ingredient of many imaging or machine learning pipelines. 
	Non-smoothness encodes structural constraints on the solutions, such as sparsity, group sparsity, low-rank and sharp edges.
	It is also the basis for the definition of robust loss functions and scale-free functionals such as square-root Lasso.  
	Standard approaches to deal with non-smoothness leverage either proximal splitting or coordinate descent. These approaches are effective but usually require parameter tuning, preconditioning or some sort of support pruning.  
	In this work, we advocate and study a different route, which operates a non-convex but smooth over-parametrization of the underlying non-smooth optimization problems. This generalizes quadratic variational forms that are at the heart of the popular Iterative Reweighted Least Squares (IRLS).
	Our main theoretical contribution connects gradient descent on this reformulation to a mirror descent flow with a varying Hessian metric. This analysis is crucial to derive convergence bounds that are dimension-free. This explains the efficiency of the method when using small grid sizes in imaging.
	Our main algorithmic contribution is to apply the Variable Projection (VarPro) method which defines a new formulation by explicitly minimizing over part of the variables. This leads to a better conditioning of the minimized functional and improves the convergence of simple but very efficient gradient-based methods, for instance quasi-Newton solvers. 
	We exemplify the use of this new solver for the resolution of regularized regression problems for inverse problems and supervised learning, including  total variation prior and non-convex regularizers. 
\subclass{68Q25, 68R10, 68U05}
\end{abstract}

% !TEX root = ../SINUM-VarPro.tex

\section{Introduction}

This paper introduces and studies a new class of solvers for a general set of sparsity-regularized problems. It leverages two key ideas: a smooth over-parameterization of the initial non-smooth problem and a bi-level variable projection to enhance its conditioning and cope with analysis-type priors. We first present these two points before relating them to previous works.

%%%%%%%%%%%%%%%%%%%%%%%%%%%%%%%%%%%%%%%%%%%%%%%%
%%%%%%%%%%%%%%%%%%%%%%%%%%%%%%%%%%%%%%%%%%%%%%%%
\subsection{Non-convex parameterizations}

%%%%%%%%%%%%%%%%%%%%%%%%%%%%%%%%%%%%%%%%%%%%%%%%
\paragraph{Structured non-smooth optimization problems}

Let $A: \RR^n \to \RR^m$ and  $L:\RR^n\to\RR^p$ be  linear operators. We consider the following non-smooth optimization problem
\begin{equation}\label{eq:gen-tv}
	\min_{x\in\RR^n}  \Phi(x) \eqdef \norm{L x}_{1,2} +  F_0(A x).
\end{equation}
Here $A \in \RR^{m \times n}$ plays the role of the imaging operator in inverse problems or  the design matrix for supervised learning, while $F_0:\RR^m \to [0,\infty]$ is a proper, lower semi-continuous convex loss function.
A guiding example is the $\ell_2$ loss $F_0(z) = \frac{1}{2\lambda}\norm{z - y}^2$ where $y\in \RR^m$ represents the given data and  $\lambda>0$ is a regularisation parameter. 
The regularization is induced by a group sparsity norm  
\begin{equation}\label{eq:l1l2}
	\norm{z}_{1,2} \eqdef \sum_{g\in\Gg} \norm{z_g}_2
	= \sum_{g\in\Gg} ( \sum_{k\in g} z_k^2 )^{\frac{1}{2}}, 
\end{equation}
where $\Gg$ is partition of $\ens{1,\ldots, p}$. 
The simplest setup is $L=\Id_{n \times n}$, so that \eqref{eq:gen-tv} is a group-lasso problem inducing direct group-sparsity of $x$~\cite{yuan2006model}. The sub-case where the group have size 1 is the classical Lasso~\cite{tibshirani1996regression}, which is useful to perform feature selection in learning, regularized inverse problems in imaging~\cite{starck2010sparse} and for compressed sensing~\cite{CRT:CS}. 
Using more general $L$ operators leads to more complex regularization priors. A popular case is when $L$ is a finite difference discretization of the gradient operator, so that $\norm{L x}_{1,2}$ is the total variation semi-norm, favoring piecewise constant signals in 1-D~\cite{mammen1997locally} and cartoon images in 2-D~\cite{rudin1992nonlinear}.
Another example is when $L x = ( x_h )_{h \in \Hh}$ extract (possibly overlapping) blocks $h$ (so that $\Hh$ is in general not a partition) to favor possibly complex block patterns~\cite{bach2011optimization}. 

%%%%%%%%%%%%%%%%%%%%%%%%%%%%%%%%%%%%%%%%%%%%%%%%
\paragraph{Hadamard over-parameterization}

The goal of this paper is to study the application on \eqref{eq:gen-tv} of the \textit{Hadamard} parametrization of $\norm{\cdot}_{1,2}$, which reads
\begin{equation}\label{eq:overparam}
	\norm{z}_{1,2} = \min_{u\odot v = z} \frac12 \norm{u}^2_2 + \frac12 \norm{v}_2^2,
\end{equation}
where the minimisation is over vectors $u \in\RR^p$, $v\in\RR^{\abs{\Gg}}$ and $u\odot v \eqdef (u_g v_g)_{g\in\Gg}$. 
Thanks to this ``over-parameterization'', problem \eqref{eq:gen-tv} can be equivalently written as
\begin{align}\label{eq:hadamard-generalized}
&\min_{x\in\RR^n}  \Phi(x) = \min_{v\in\RR^{\abs{\Gg}}}\min_{u\in\RR^p}   G(u,v)\\
&\qwhereq G(u,v) \eqdef \min_x \enscond{ \frac{1}{2} \norm{u}^2_2 + \frac{1}{2} \norm{v}^2_2 + F_0(A x) }{ Lx = u\odot v}. 
\end{align}
In the case where $L=\Id$, this problem can be written as
\begin{equation}\label{eq:hadam-id}
	\min_{v\in\RR^{\abs{\Gg}}}\min_{u\in\RR^p}   \frac{1}{2}\norm{u}_2^2 + \frac12 \norm{v}_2^2 + F_0(A(u\odot v))
\end{equation}
and this is a smooth albeit nonconvex optimisation problem provided that $F_0$ is smooth. 
This idea has been previously studied in \cite{hoff2017lasso}. Moreover, in this case, the nonconvexity is harmless in the sense that all saddle points are strict and one can guarantee global convergence with certain gradient-based algorithms \cite{2021-Poon-noncvxpro}. 
In Section \ref{sec:finegrid-mirror}, we provide connections of gradient descent on this reparametrized form to mirror descent and show how such a parametrization leads to dimension-independent convergence rates. 

%%%%%%%%%%%%%%%%%%%%%%%%%%%%%%%%%%%%%%%%%%%%%%%%
\paragraph{Variable projection (VarPro) reduction}

In the case $L=\Id$, it is tempting to directly use smooth optimization methods to solve~\eqref{eq:hadam-id}, but as exposed in our previous work~\cite{2021-Poon-noncvxpro}, it makes sense to improve its conditioning by the so-called ``variable projection'' (VarPro) technique. 
In the more complicated case where $L\neq \Id$ and is not invertible, the Hadamard parameterization looks at first sight unhelpful as we have simply added in~\eqref{eq:hadam-id} the difficulty of  non-convexity without alleviating the non-smoothness issue. 
We thus propose to replace~\eqref{eq:hadam-id} by the following bilevel program
\begin{equation}\label{eq:varpro0}
\min_{v\in\RR^{\abs{\Gg}}} f(v) \qwhereq f(v) = \min_{u\in\RR^p} G(u,v).
\end{equation}
This idea of marginalizing on one variable is called variable projection (VarPro) and is a well-known technique~\cite{golub1973differentiation,golub2003separable}. One of its  advantages is that splitting into a bilevel problem leads to better problem conditioning. In particular, in certain situations when $G$ is smooth, minimizing $f$ instead of $G$ is preferable because the condition number of the Hessian of $f$ can be shown to be no worse (and often substantially better) than that of $G$ \cite{ruhe1980algorithms}. 
Moreover, as we see below, while $G$ is not differentiable, the function $f$ is differentiable. The motivations for the VarPro formulation are thus two-fold: first, it is essential to obtain a smooth optimisation problem when $L$ is not invertible; second, even in the case where $L=\Id$, while both optimisation of $G$ and $f$ improve over standard algorithms for handling \eqref{eq:gen-tv},  the improvement in conditioning in the VarPro formulation can further lead to substantial numerical gains over directly optimizing $G$. 

\begin{rem}[More general settings]
For the sake of clarity, we mostly focus in this article on this $\ell^1-\ell^2$ functional~\eqref{eq:l1l2}. The methods and algorithm that we introduce can be extended to any regulariser that admits a quadratic variational form, including nuclear norm and $\ell_q$ regularisation with $q<2$; and also handle non-smooth convex losses including the $\ell_1$-loss and the constrained setting. Such extensions are discussed in Section~\ref{sec:nonsmooth-robust}.
\end{rem}

%%%%%%%%%%%%%%%%%%%%%%%%%%%%%%%%%%%%%%%%%%%%%%%%
%%%%%%%%%%%%%%%%%%%%%%%%%%%%%%%%%%%%%%%%%%%%%%%%
\subsection{Previous Works}

\paragraph{Lasso solvers} The case $L=\Id$ is arguably simpler, and can be tackled using a flurry of non-smooth optimization solvers.
The simplest one is the forward-backward algorithm \cite{lions1979splitting}, which is known as the iterative soft thresholding (ISTA) algorithm~\cite{daubechies2004iterative}. 
It convergence is relatively slow, and assuming the dimension is a fixed constant, it enjoys a $1/t$ rate in worse case. This rate is improved to $1/t^2$ using Nesterov acceleration~\cite{nesterov-momentum} and leads to the  FISTA algorithm~\cite{beck2009fast}. 
In practice, the speed of the algorithm is improved using adaptive stepsizes and restarting strategies \cite{o2015adaptive} as well as quasi-Newton and variable metric methods \cite{combettes2014variable,becker2019quasi}. 
To better cope with fine grid settings in imaging, and obtain better dimension-free analysis, it is possible to replace Euclidean metrics by mirror descent methods~\cite{nemirovskij1983problem}. We give some details about this line of ideas in Section~\ref{sec:finegrid-mirror} since this is closely related to the Hadamard parameterization. 
For problems with very sparse solutions, algorithms leveraging coordinate descent strategies are often more efficient~\cite{friedman2010regularization}. These schemes are typically combined with support pruning schemes \cite{ghaoui2010safe,ndiaye2017gap,massias2018celer}.

\paragraph{Analysis-type priors and non-smooth loss}

Problems where $L \neq \Id$ cannot be treated directly using primal descent methods, and require some form of primal-dual reformulation. Proximal splitting schemes can be applied, such as Alternating Direction Method of Multipliers (ADMM) \cite{boyd2011distributed}, Douglas-Rachford algorithms \cite{douglas1956numerical} and primal-dual algorithms \cite{chambolle2011first}
These schemes are  popular due to their relatively low per iteration complexity. They usually exhibit slow sublinear convergence rates in general, with linear convergence under strong convexity and sharpness assumptions~\cite{liang2018local}. 
The performance of these methods are improved by using adaptive step size selection and preconditioning~\cite{bredies2017proximal}.
These solvers can also be used for non-smooth loss functions (as detailed in Section~\ref{sec:nonsmooth-robust}) such as the square root lasso problem~\cite{belloni2011square}, TV-$\ell^1$~\cite{nikolova2004variational} and matrix-regularizers such as nuclear norm~\cite{recht2010guaranteed}.
Similarly to the forward-backward algorithm, mirror geometry can be introduced in these primal-dual solvers to better cope with sparsity and positivity constraints~\cite{silveti2020generalized}.

\paragraph{The quadratic variational formulation and IRLS}

As explained in~\cite{2021-Poon-noncvxpro}, over-pa\-ra\-me\-te\-ri\-za\-tion formula of the form~\eqref{eq:overparam} are equivalent (up to a change of variable) to so-called quadratic variational formulations. 
For the case of the $\ell^1-\ell^2$ norm, writing $v_g=\sqrt{\eta_g}$ and $u_g=z_g/\sqrt{\eta_g}$ for $\eta_g \in \RR_+$, the non-convex smooth formulation~\eqref{eq:overparam} is re-written as the convex but non smooth over-parameterization 
\begin{equation}\label{eq:irls-param}
	\norm{z}_{1,2} = \frac{1}{2} \umin{\eta\in \RR_+^{\abs{\Gg}}} \pa{\frac{\norm{z_g}_2^2}{\eta_g} + \eta_g}. 
\end{equation}
These formulations can be traced back to early computer vision works such as \cite{geiger1991common,geman1992constrained}. 
A detailed account for these variational formulations can be found in \cite{micchelli2013regularizers}, and further studied in the monograph \cite{bach2011optimization} under the name of subquadratic norms. 
Such quadratic variational formulations are useful to derive, in the case $L=\Id$, the celebrated iterative reweighted least squares (IRLS) algorithm, which alternatively minimize on $z$ and $\eta$. In this basic form, IRLS fails to converge in general because of the non-smoothness of~\eqref{eq:irls-param}.
One popular approach is to add a regularization penalty $\frac{\epsilon}{2}\sum_g \eta_g^{-1}$ to the formulation~\eqref{eq:irls-param} as detailed in~\cite{daubechies2010iteratively}.
A nuclear norm version of IRLS has been used in \cite{argyriou2008convex} where an alternating minimisation algorithm was introduced. 
Instead of this IRLS convex optimization strategy, another route is to use alternating minimization directly on the non-convex $(u,v)$ parameterization~\eqref{eq:hadam-id}, see for instance \cite{rennie2005fast,hastie2015matrix,mardani2015estimating,hoff2017lasso} for the case of the $\ell_1$ and nuclear norms.

\paragraph{Variable projection}

These alternating minimization method, either on $(u,v)$ or on $(x,\eta)$ are quite slow in practice because of the poor conditioning of the resulting over-parameterized optimization problem. 
As explained in~\cite{ruhe1980algorithms,golub2003separable} the variable projection reformulation~\eqref{eq:varpro0} provably improves the conditioning of the Hessian of the functionals involved, we refer also to \cite{hong2017revisiting,zach2018descending} for more recent studies. 
This approach is classical (see for instance~\cite[Chap.10]{rockafellar2009variational} for some general theoretical results on reduced gradients), and was introduced initially for solving nonlinear least squares problems.

% Nonsmooth variable projection is studied in \cite{van2016non}, although the present work is in the classical setting of variable projection due to our smooth reparametrization. 
% 
% Reduced gradients have also been associated with the quadratic variational formulation  in several works \cite{bach2011optimization,pong2010trace,rakotomamonjy2008simplemkl}. The idea is to apply descent methods over $g(\eta) = \min_x R_0(\eta,x)+\frac12 \norm{A x - y}_2^2$. Although the function over $\eta$ and $x$ is discontinuous, the function $g$ over $\eta$ is smooth  and   one can apply first order methods, such as proximal gradient descent to minimise $g$ under positivity constraints. While quasi-Newton methods can be applied in th is setting with bound constraints, we show in Section \ref{sec:numerics_lasso} that this approach is typically less effective than our  nonconvex bilevel approach.  

% \todo{I did not recall to which problem the following sentence is related}  In the setting of the trace norm, the optimisation problem is constrained on the set of positive semidefinite matrices, so one is  restricted to using first order methods \cite{pong2010trace}.
 
\paragraph{Gradient flow and fine grid analysis}

The convergence speed of first order non-smooth methods in general degrades as the dimension increases. This is the case in particular in imaging problems (such as deconvolution or super-resolution problems) as the grid size goes to zero, which corresponds to  a setting  where the object of interest is a stream of Dirac masses and  one seeks to estimate  their precise positions ~\cite{duval2017sparse}. 
While dedicated solvers have been developed to alleviate this issue and can even cope with ``off-the-grid'' formulations (without explicit discretization)~\cite{bredies2013inverse,candes2014towards}, ISTA and related forward-backward solvers are still the most popular. 
Chizat proposed in~\cite{chizat2021convergence} an analysis of rate of convergence of forward-backward when the grid size is arbitrary small, in which case $O(1/t)$ rate does not holds, and one obtains slower $O( 1/t^{\frac{4}{4+d}} )$ rates (where $d$ is the ambient dimension, e.g. $d=2$ for images). These rates can be improved to $O( \log(t)/t )$ by replacing Euclidean proximal operators by more general mirror operators, as detailed in Section~\ref{sec:finegrid-mirror}. In our work, we relate the Hadamard over-parametrization to this mirror flow, which partly explains its efficiently, and is useful to derive convergence bounds.

%%%%%%%%%%%%%%%%%%%%%%%%%%%%%%%%%%%%%%%%%%%%%%%%
%%%%%%%%%%%%%%%%%%%%%%%%%%%%%%%%%%%%%%%%%%%%%%%%
\subsection{Contributions}

Our first set of contributions is the derivation and the analysis in Section~\ref{sec:nonconvexparam} of a VarPro reduced method in the general case of an analysis sparsity prior. The main contribution is the proof in Theorem~\ref{thm:diff} that the resulting functional is differentiable, and an explicit formula for the gradient. 
Our second set of contributions is the proof, in Section~\ref{sec:finegrid-mirror}, that the gradient descent on the Hadamard formulation~\eqref{eq:overparam} is equivalent to a mirror-flow with a time-varying entropy function. This shows that while the descent is computed in an over-parameterized domain $(u,v)$, it is still equivalent to a classical flow on the initial variable $x$, and that this flow should be understood for a non-Euclidean, time-varying, Hessian-type metric.
This analysis is leveraged to derive dimension-free (i.e. insensitive to the grid step size) convergence bound for the gradient descent on the Hadamard formulation. Most notably, we show in Proposition~\ref{prop:lip_g} a $1/\sqrt{t}$ convergence bound on the gradient of the minimized energy, and Proposition~\ref{prop:rate} shows that the convergence in function's value is controlled by the convergence of the gradients.   
Lastly, Section~\ref{sec:nonsmooth-robust} focusses on more practical considerations, by explaining how to extend our approach to non-smooth loss functions and non-convex regularizers. These extensions are exemplified with numerical simulations on imaging problems.

\paragraph{Connection with previous works}

This work builds on our initial work~\cite{2021-Poon-noncvxpro}, which derived and studied the VarPro method in the case $L=\Id$. The case of an arbitrary $L$ is more involved because of the lack of smoothness of the Hadamard parameterization, which fortunately is to a large extent absorbed by the VarPro reduction. Beside this extension to analysis-type prior, this work also propose a novel mirror-type analysis of the Hadamard formulation. 
% !TEX root = ../SINUM-VarPro.tex

\section{Hadamard and VarPro parameterizations} \label{sec:nonconvexparam}

We now give a detailed analysis of the Hadamard formulation~\eqref{eq:hadamard-generalized} and its associated VarPro marginalization~\eqref{eq:varpro0}, which is crucial to ensure differentiability of the function to be minimized. 

\textit{Remark on notation:} We already introduced above the Hadamard product for $u\in \RR^p$ and $v\in \RR^{\abs{\Gg}}$ denoted as $u\odot v \eqdef (u_g v_g)_{g\in\Gg}$.  When $u,v\in \RR^p$ are of the same length (i.e. trivial group structure), we  write $u\odot v \eqdef (u_i v_i)_i$ to denote pointwise multiplication and  $u/v \eqdef (u_i/v_i)$ to denote pointwise division. We will also use  $u^2 \eqdef u\odot u$.

%%%%%%%%%%%%%%%%%%%%%%%%%%%%%%%%%%%%%%%%%%%%%%%%
\subsection{Dual formulation}

Since for a generic $L$, the Hadamard formulation~\eqref{eq:hadamard-generalized} involves the resolution of a constrained problem, analyzing the differentiability of $f$ requires to study a dual formulation. 
The VarPro formulation~\eqref{eq:hadamard-generalized} has the form of a bi-level program
\begin{equation*}
	f(v) \eqdef \min_{u\in \RR^p} G(u,v)\qandq u(v) = \uargmin{u} G(u,v).
\end{equation*}
For such problem, it tempting to compute the gradient of $f$ by applying the chain rule:
\begin{equation*}
	\nabla f(v) = \partial_v G(u(v), v) + \partial_u G(u(v), v) \nabla u(v) = \partial_v G(u(v), v)
\end{equation*}
where we used the fact $\partial_u G(u(v), v) = 0$ due to optimality of $u(v)$. 
This is of course only a formal argument, and in particular, this requires $\partial_v G(u(v), v)$ to be well-defined.
For the VarPro problem \eqref{eq:varpro0}, it is not  immediately clear that $\partial_v G(u(v), v)$ is well-defined since the variable $v$ appears inside a linear constraint. However, the following proposition shows how the inner optimisation problem in \eqref{eq:varpro0} can be written as a concave maximisation problem involving a dual function $\phi(v,\alpha,\xi)$ that is differentiable with respect to $v$. Note also that due to the $\norm{v\odot \alpha}^2$ term, given a maximiser $(\alpha,\xi)$ to the inner  problem, $v\odot \alpha$ is uniquely defined and $\partial_v \phi(v,\alpha,\xi) = -v\odot \alpha^2$ is thus well-defined. Precise regularity properties of the function $f$ is studied in the following section, but from this proposition, one can at least formally expect $f$ to be differentiable.

\begin{prop}[Dual formulation of the inner problem]\label{prop:gentv}
The function $f$ defined in \eqref{eq:varpro0} can be written as 
\begin{equation}\label{eq:def_f_dual}
f(v) = \min_{u\in\RR^p} G(u,v)=   \max_{\xi\in\RR^m,\alpha\in\RR^p} \frac12\norm{v}^2 +\phi(v,\alpha,\xi) 
\end{equation}
$$
\qwhereq \phi(v,\alpha,\xi) \eqdef - \frac12 \norm{v\odot \alpha}^2 -F_0^*(\xi)  -\iota_{K}(\xi,\al),
$$
where $\iota_K$ denotes the indicator function on the set $K \eqdef \enscond{(\xi,\al)}{ L^\top \alpha = -A^\top \xi }$. Moreover,  the optimal $\al,\xi$ satisfy $Ax\in  \partial F_0^*(\xi)$ and $Lx = \alpha\odot v^2$ for some $x\in\RR^n$.
\end{prop}

\begin{rem}[Example of quadratic loss]
When $F_0(z) = \frac{1}{2\lambda}\norm{z-y}_2^2$, $F_0^*(\xi)  = \frac{\lambda}{2} \norm{\xi}_2^2 + \dotp{\xi}{  y }$, so
$$
\phi(v,\alpha,\xi) = -\frac12 \norm{v\odot \al}^2 - \frac{\lambda}{2}\norm{\xi}^2 - \dotp{\xi}{y}.
$$
 The maximisation problem~\eqref{eq:def_f_dual} is therefore a quadratic problem  and given optimal solutions $(\alpha,\xi)$, there exists $x$ such that 
$$\lambda \xi = A x - y\qandq Lx  = v^2\odot \alpha \qandq L^\top \alpha - A^\top \xi = 0.$$
This can be written as the linear system
\begin{equation}\label{eq:linear_system_general}
M_v
\begin{pmatrix}
\xi\\ \alpha\\ x
\end{pmatrix} = \begin{pmatrix}
-y\\0 \\0
\end{pmatrix}, \qwhereq M_v\eqdef \begin{pmatrix}
\lambda \Id & 0 &-A\\
0 &-\diag(\bar v^2)&D\\
-A^\top &L^\top &0
\end{pmatrix}.
\end{equation}
where $\bar v\in\RR^p$ is the \textit{extension} of $v$ such that $\bar v \odot \alpha = v\odot \alpha$ for all $\alpha\in\RR^p$.
One can therefore handle the inner problem $\max_{\alpha,\xi}\phi(v,\alpha,\xi)$ by solving a linear system.
\end{rem}

\begin{proof}
We first write
$$
f(v) =  \min_{\substack{u\in\RR^p, \\ x\in\RR^n, w\in\RR^m}} \enscond{ \frac12 \norm{v}^2 + \frac12 \norm{u}^2 + F_0(w)}{ Lx= u\odot v, \; Ax = w}.
$$
Note that this is a convex optimisation problem, and by considering its dual formulation, we have
\begin{align*}
&f(v) = \min_{\substack{u\in\RR^p, \\ x\in\RR^n, w\in\RR^m}} 
 \max_{\substack{\alpha \in\RR^p,\\ \xi\in\RR^m}} \frac12 \norm{v}^2 + \frac12 \norm{u}^2  + F_0(w) +\dotp{\alpha}{ Lx - u\odot v} + \dotp{\xi}{ Ax  - w}\\
 &= 
 \max_{\alpha \in\RR^p, \xi\in\RR^m} \frac12 \norm{v}^2 - \frac12 \norm{v\odot \alpha}^2 - F_0^*(\xi) \qwhereq L^\top \alpha = -A^\top \xi, 
\end{align*}  
with the optimal $\al,\xi$ satisfying $Ax \in  \partial F_0^*(\xi)$ and $Lx =  v^2\odot \alpha$.
\end{proof}

%%%%%%%%%%%%%%%%%%%%%%%%%%%%%%%%%%%%%%%%%%%%%%%%
\subsection{Differentiability}

In this section, we consider the regularity properties of $f$. We recall that 
$F_0:\RR^m \to [0,\infty]$ is a proper, lower semi-continuous convex loss function with Lipschitz gradient. 

\newcommand{\lipF}{M_F}
\begin{prop}[Well-posedness]
Assume that $F_0 \in \Cc^{1,1}(\RR^m; \RR)$ and recall the definition of $f$ from  \eqref{eq:def_f_dual} of Prop \ref{prop:gentv}. Then, $\dom(f) = \RR^p$, the set of maximisers in \eqref{eq:def_f_dual}  is non-empty.
\end{prop}
\textit{Notation:} We denote the range of a matrix $L$ by $\mathcal{R}(L)$. Given  $\alpha\in \RR^n$ and $S\subset \Gg$,  we write $\alpha_S$ to denote the restriction of $\alpha$ to entries whose indices are in the groups defined by $S$, that is,  the vector taking values $\alpha_j$ whenever there is a group $g$ such that $j\in g\in S$ and taking value 0 otherwise.  $L_S$ denotes the matrix $L$ with columns restricted to those indexed by $S$

\begin{proof}
First note that since $F_0\in C^{1,1}$, $F_0^*$ is strongly convex, and so, it is bounded from above. Also,
 there exist $\alpha$ and $ \xi$ satisfying $L^\top \al + A^\top \xi = 0$ (take $\alpha = 0$ and $\xi=0$).  So, for each $v$, $\max_{\alpha,\xi} \phi(v,\xi,\alpha)$ exists and hence, $f(v)<\infty$ for all $v$.
 
 To show that the set of maximisers is nonempty, first note that existence of the maximum along with strong concavity of $\phi(v,\alpha, \cdot)$ in the variable $\xi$ imply that we can consider a maximising sequence $(\al_n,\xi_n)$ with $L^\top \alpha_n + A^\top \xi_n = 0$, and $\norm{\xi_n} \leq C$ for some $C>0$. It follows that there is a convergent subsequence $\xi_{n_k} \to \xi_*$ for some $\xi_*\in\RR^m$.  Moreover, denoting the range of $L^\top$ by $\Rr(L^\top)$, since $(A^\top \xi_{n_k}) \subset \Rr(L^\top)$ is a convergent  sequence and $\Rr(L^\top)$ is closed, $A^\top \xi_*\in \Rr(L^\top)$. We also have $v\odot \alpha_n$ is uniformly bounded, so, denoting $S\eqdef \Supp(v)$,  $(\alpha_n)_S$ has a convergent subsequence. Since
$$
L^\top \alpha_n = L^\top (\alpha_n)_S + L^\top (\alpha_n)_{S^c}
$$
and both $A^\top \xi_n = L^\top \alpha_n$ and $L^\top (\alpha_n)_{S}$ converge up to a subsequence, $ L^\top (\alpha_n)_{S^c}$ is also convergent upto a subsequence. It follows that there exists $\alpha_*$ and $\xi_*$ such that $\xi_{n_k}$ converges to $\xi_*$, $(\alpha_{n_k})_S$ converges to $(\alpha_*)_S$ and $A^\top \xi_* = L^\top \alpha_*$. We then apply the fact that $\phi(v,\cdot,\cdot)$ is upper semi-continuous to deduce that $(\xi_*,\alpha_*)$ is a maximiser.
\end{proof}

The case where $F_0$ is the quadratic loss and $L=\Id$ was investigated in \cite{poon2021smooth} and it is straightforward in this case to see that $f_0$ is smooth. For more general $F_0$ and $L\neq \Id$, we have the following regularity result, which implies the differentiability of $f_0$. 

\begin{thm}[Differentiability]\label{thm:diff}
Assume that $F_0 \in \Cc^{1,1}(\RR^m; \RR)$.  Then,
 $f$ differentiable for all $v\in\RR^n$ with $\nabla f(v) =v- v \odot \alpha_v^2$ where $(\alpha_v, \xi_v)\in \argmin_{\alpha,\xi}\phi(v,\al,\xi)$. Moreover, if $v_i\neq 0$ for all $i$, then $f$ is strictly differentiable .
\end{thm}

Note that even through $\phi(v,\cdot,\cdot)$ does not necessarily have unique maximisers, $v\odot \alpha_v$ is uniquely defined due to the quadratic term $\norm{v\odot \alpha_v}^2$ in $\phi$, so $\alpha_v$ is unique on the support of $v$ and hence, the formula given for $\nabla f$ in the above theorem is clearly well-defined.

From the above result, we see that the computation of $\nabla f(v) = v- v\odot \alpha_v^2$ simply requires solving the inner problem  $\max_{\al,\xi} \phi(v,\al,\xi)$ to obtain $\alpha_v$. Before proving this theorem in Section \ref{proofs}, we first make some remarks on  the computation of the $\alpha_v$ and provide numerical examples.

%%%%%%%%%%%%%%%%%%%%%%%%%%%%%%%%%%%%%%%%%%%%%%%%
\subsection{Squared Euclidean loss}

The inner problem  $\max_{\al,\xi} \phi(v,\al,\xi)$ is a convex optimisation problem, which might require a dedicated inner solver. In the remain part of this section, we focus on the setting where 
$$
F_0(z) = \frac{1}{2\lambda} \norm{z-y}_2^2, \qwhereq y\in \RR^m.
$$
In this case, the inner maximisation problem is a least squares problem.
In the case where $v_i\neq 0$ for all $i$, this can be further simplified, by first rewriting~\ref{eq:linear_system_general} as 
\begin{equation}\label{eq:linsys}
\begin{pmatrix}
A^\top A &  L^\top\\
L &-\diag({\bar v}^2)
\end{pmatrix} \begin{pmatrix}
x\\
\alpha
\end{pmatrix} = \begin{pmatrix}
A^\top y\\
0
\end{pmatrix}.
\end{equation}
Equivalently, we have \begin{equation}\label{eq:toinvert}
A^\top Ax + \lambda L^\top  \diag(1/\bar v^2) L x  = A^\top y
\end{equation}
and let $\bar v^2 \cdot \alpha =  Lx$ and $\xi = \frac{1}{\lambda}(Ax- y)$.

\subsubsection{Group-Lasso setting}

In the case where $L=\Id$, which was studied in\cite{poon2021smooth}, the inner problem can be written as
$$
f_0(v) = \max_{\xi\in\RR^m}  \frac12 \norm{v}^2 - \frac12 \norm{v\odot (A^\top \xi)}^2 - \frac{\lambda}{2} \norm{\xi}_2^2 - \dotp{\xi}{  y }.
$$
The maximiser is the solution to the linear system
$$
(A \diag(v^2) A^\top + \lambda\Id) g= -y.
$$

\subsubsection{Proximal operators}

When $X= \Id$, one has
$$
f_0(v) = \max_{\alpha\in\RR^p} \frac12 \norm{v}^2 - \frac12 \norm{v\odot \alpha}^2 - \frac{\lambda}{2} \norm{L^\top \alpha}_2^2 - \dotp{L^\top \alpha}{  y } 
$$
The maximiser is the solution to the linear system
$$
\pa{\lambda LL^\top + \diag(v^2)} \alpha  =L y
\qandq
\xi  =  y  -  \lambda L^\top \alpha .
$$

\subsubsection{The overlapping group Lasso}

Complex block-sparse patterns can be favored in the solution using an operator $L$ which extracts blocks of the vector $x$~\cite{bach2011optimization}.
We set $L: x \mapsto (\sqrt{n_g}x_{I_g})_{g\in\Gg}$ where $I_g\subseteq \ens{1,\ldots, n}$, $n_g = \abs{I_g}$. This induces the regularizer 
$\norm{L x}_{2,1} = \sum_{g\in\Gg} \norm{x_{I_g}}$.
If the groups span the entire index set, that is $\bigcup_{g\in \Gg} I_g, = \ens{1,\ldots, n}$, then $W\eqdef L^\top \diag(1/\bar v^2) L$ is a diagonal matrix with $W_{i,i} =\sum_{g\in \Gg_i}  v_g^{-2} {n_g}$ where $\Gg_i = \enscond{g\in\Gg}{i\in g}$.
 We can therefore conveniently rewrite \eqref{eq:toinvert}  leveraging
\begin{align*}
 \pa{A^\top A  + \lambda L^\top  \diag(1/\bar v^2) L }^{-1}
= \frac{1}{\lambda} W^{-1} - \frac{1}{\la} W^{-1}  A^\top\pa{\lambda \Id_m +A W^{-1} A^\top}^{-1} A W^{-1}.
\end{align*}
This formulation is advantageous in the under-determined setting where $m\ll n$.

The performance of VarPro is illustrated in  figure \ref{fig:overlap_group}, where we apply L-BFGS quasi-Newton to minimise $f$. 
The top row  of figure \ref{fig:overlap_group} shows the results when $A\in \RR^{m\times n}$ is a random Gaussian matrix with $m=300$ and $n=3000$. The groups are chosen such that they have an overlap of $s$ and the size of each group  is chosen at random from 1 to 20. The results  shown in the figure are for different regularisation parameters. In the bottom row of Figure \ref{fig:overlap_group}, we show the results on the breast cancer dataset \cite{van2002gene} commonly used for benchmarking the group Lasso \cite{obozinski2011group} \footnote{We used the data matrix downloaded from \url{https://github.com/samdavanloo/ProxLOG}}. We compare our method against ADMM with different parameters (see Section \ref{sec:admm} in the appendix).
\begin{figure}

\begin{tabular}{c@{\hspace{1pt}}c@{\hspace{1pt}}c@{\hspace{1pt}}c}
$\frac{\lambda_{\max}}{2}$&$\frac{\lambda_{\max}}{10}$&$\frac{\lambda_{\max}}{20}$\\
\includegraphics[width=0.32\linewidth]{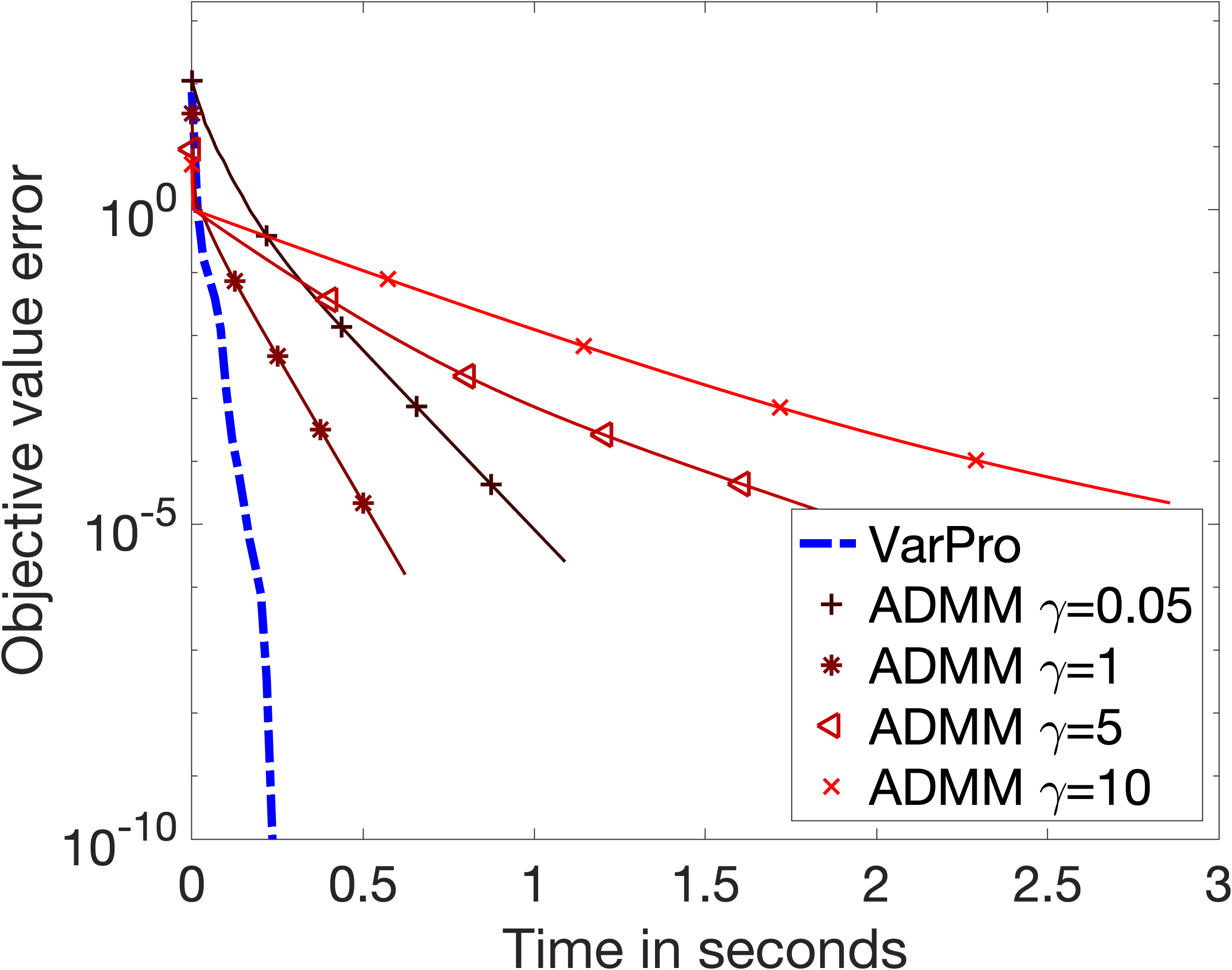}&
\includegraphics[width=0.32\linewidth]{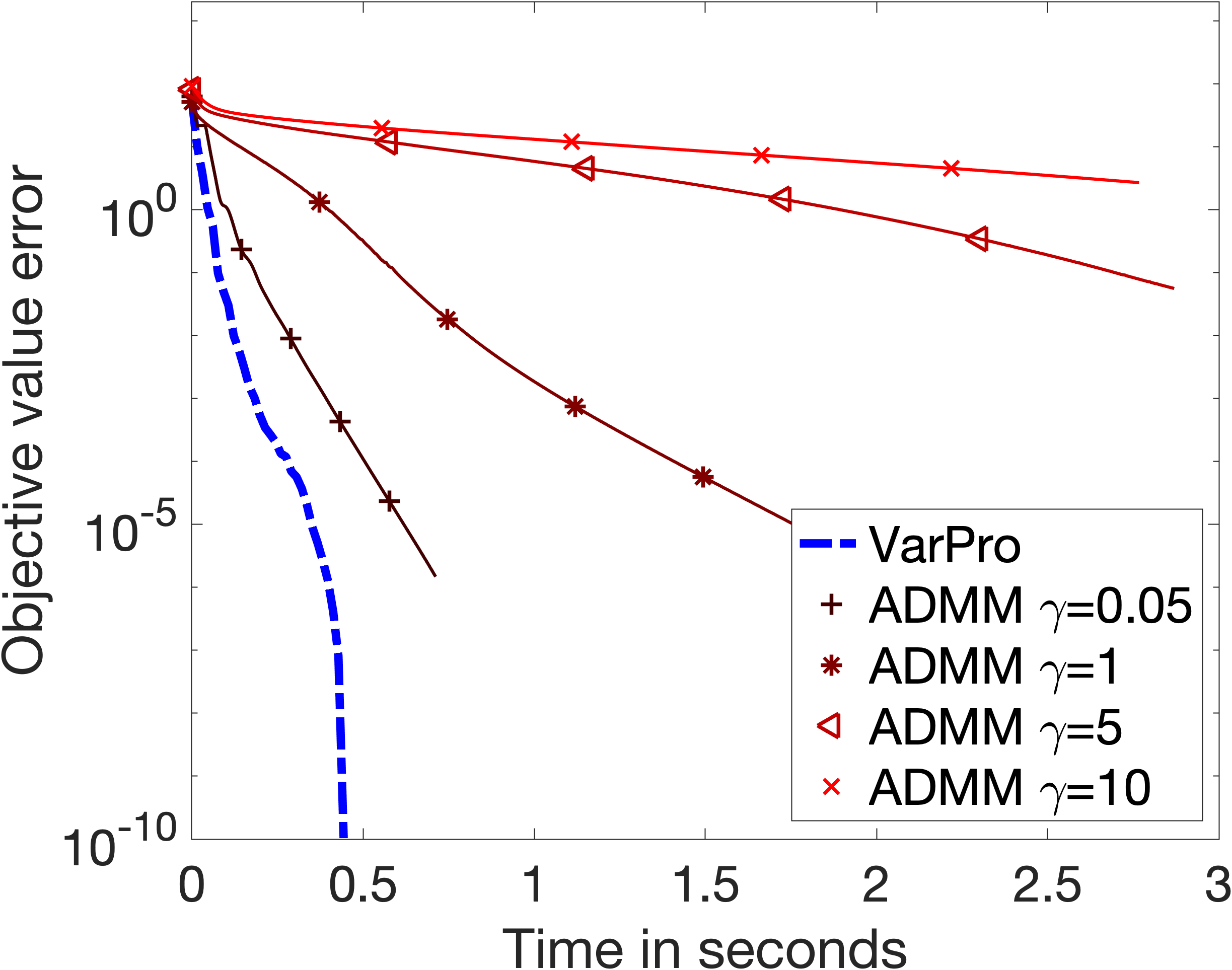}
&
\includegraphics[width=0.32\linewidth]{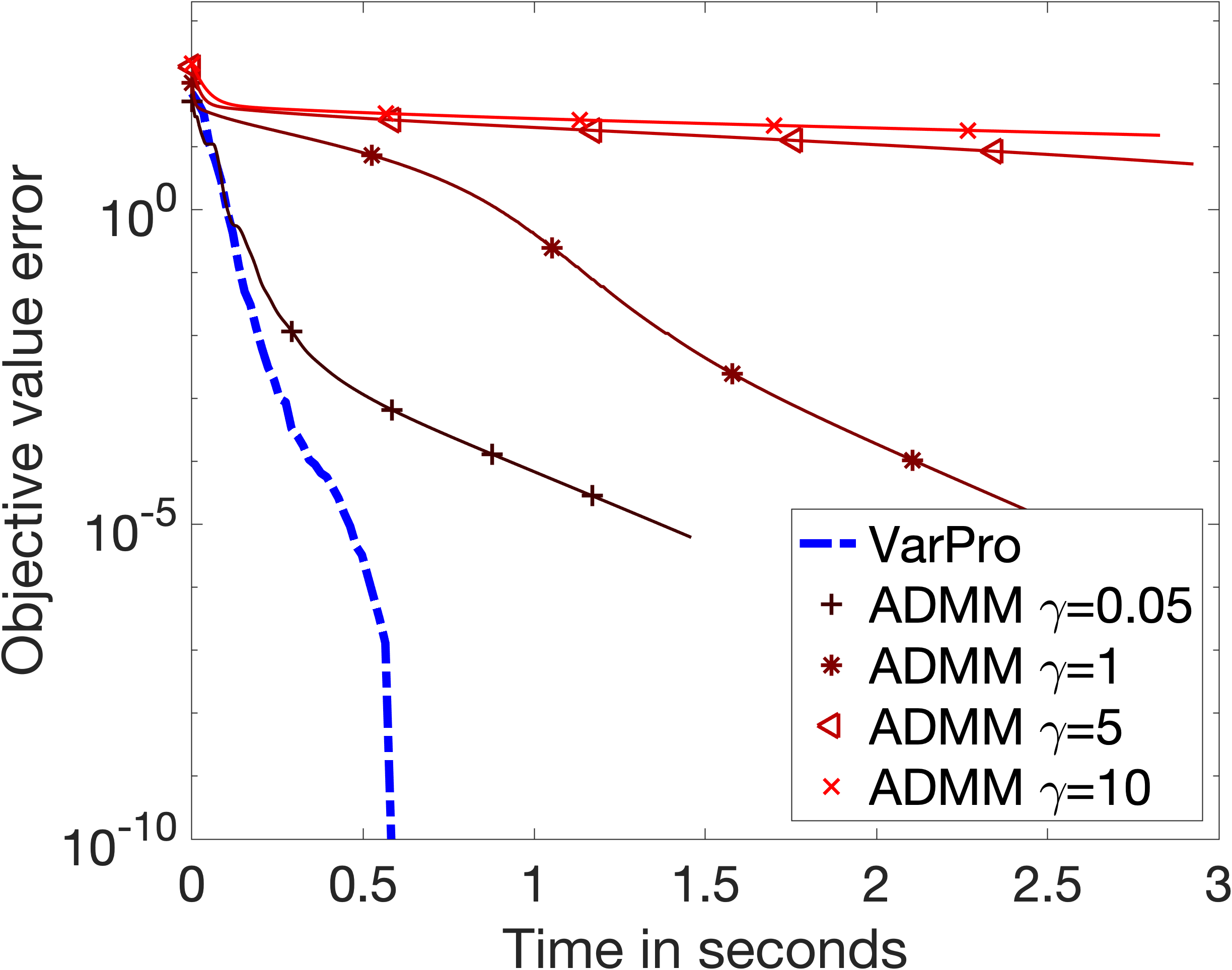}\\
\includegraphics[width=0.32\linewidth]{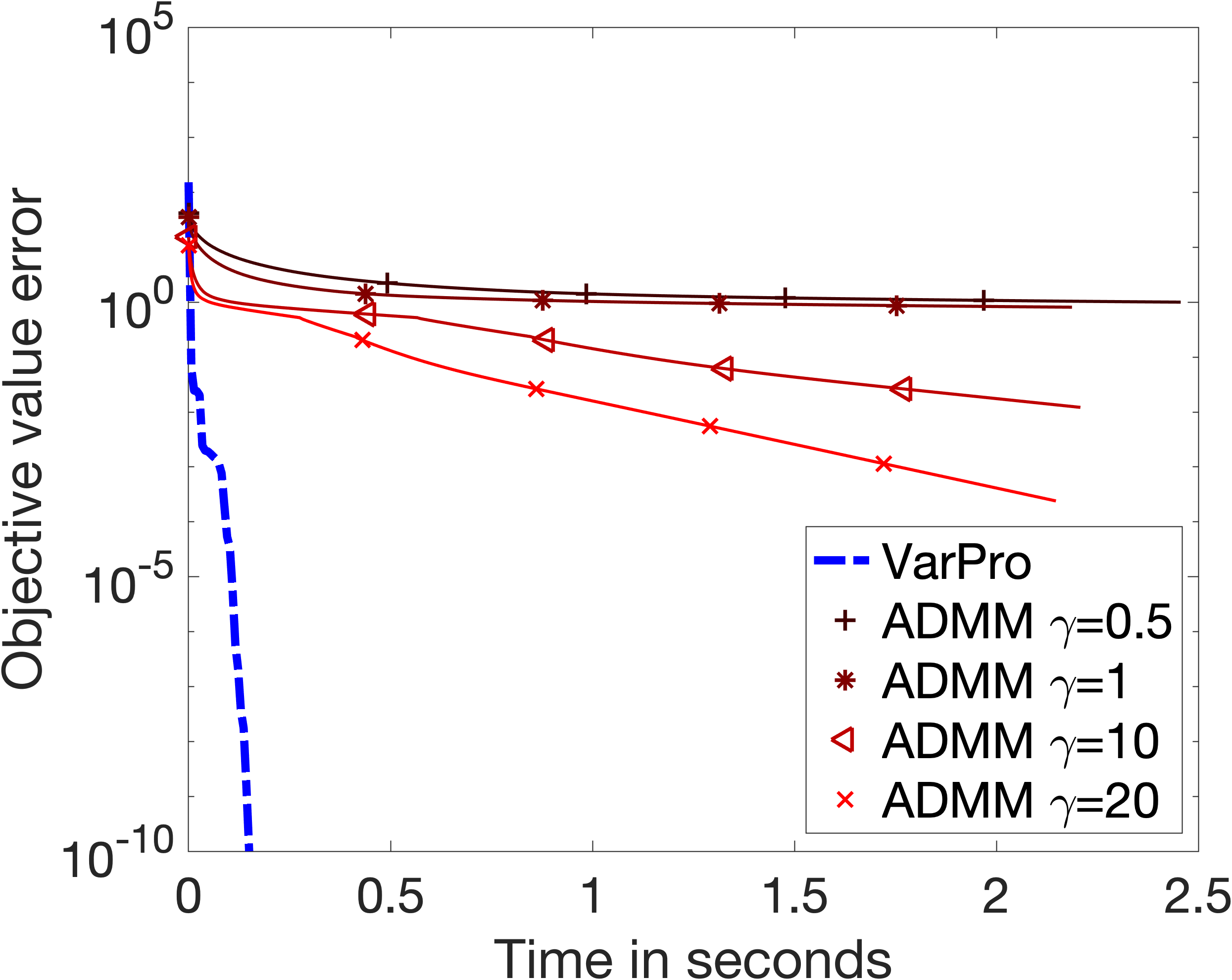}&
\includegraphics[width=0.32\linewidth]{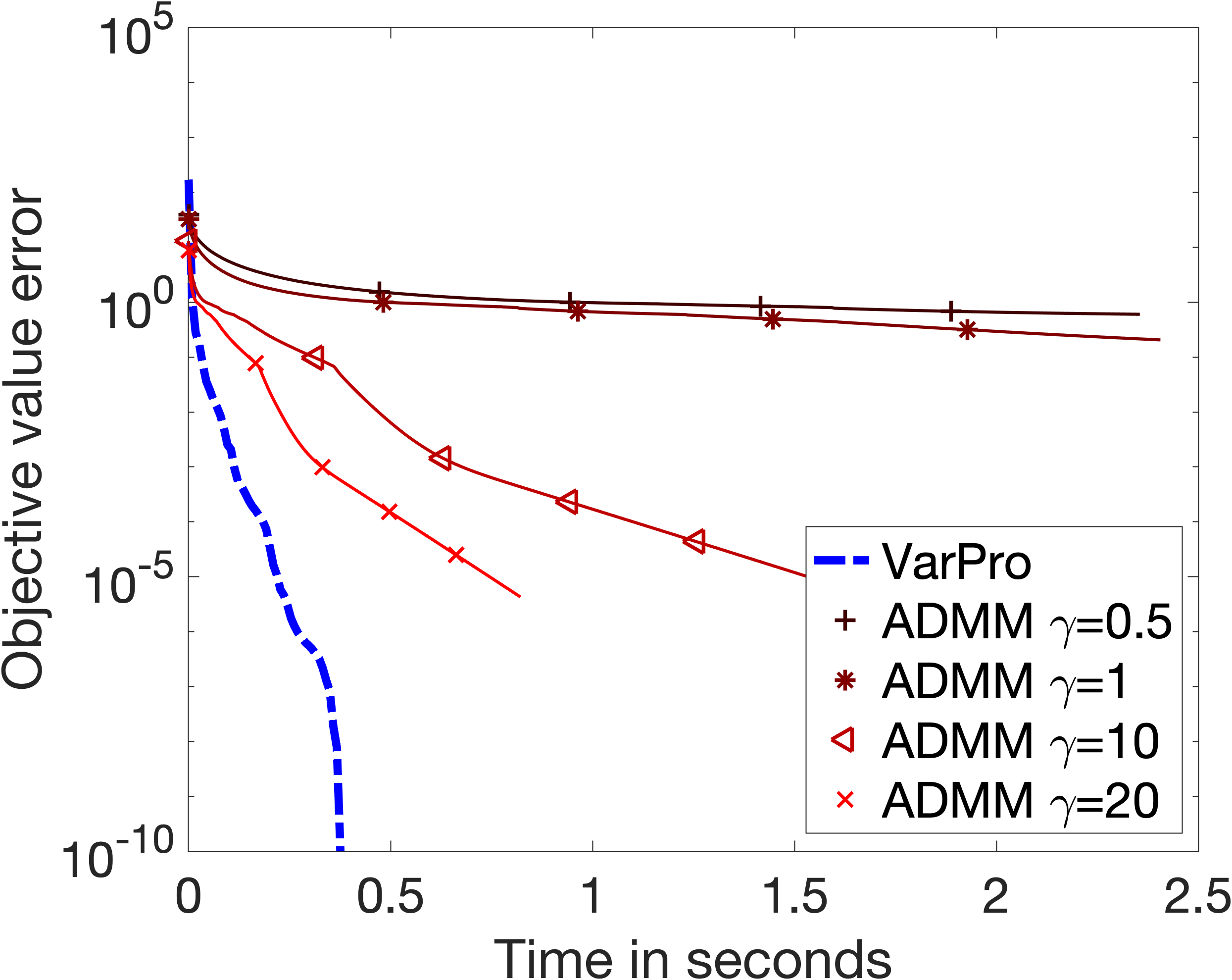}
&
\includegraphics[width=0.32\linewidth]{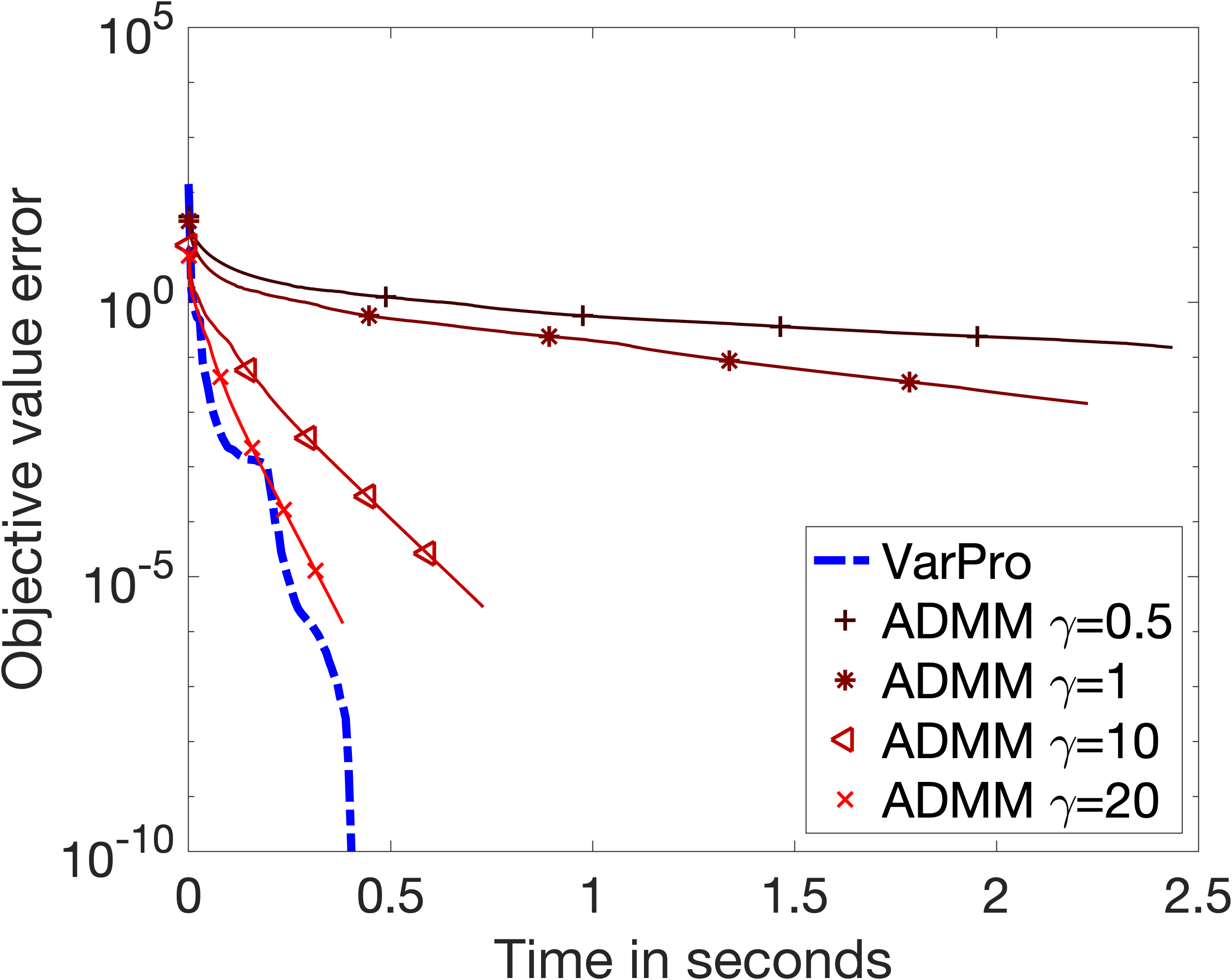}
\end{tabular}
\caption{Overlapping group lasso. Top row:  $A\in\RR^{300\times 3000}$ is a   random Gaussian matrix and the sought after vector has group overlaps of 5. Bottom row: $A$ is the  the breast cancer data matrix of \cite{van2002gene}. \label{fig:overlap_group}}
\end{figure}

\subsubsection{Total variation regularisation for image processing}\label{sec:TV}

We identify images $x \in \RR^n$ with 2-D arrays $x\in \RR^{d\times d}$ so that $n=d^2$. 
Horizontal and vertical derivative operators are defined as
$$
D^h x \eqdef (x_{i,j} - x_{i+1,j})_{i,j}, \qandq  D^v x \eqdef (x_{i,j} - x_{i,j+1})_{i,j}.
$$
and the 2 dimensional gradient operator acting on images $x\in\RR^{d\times d}$ is 
$$
Lx \eqdef (D^h x, D^v x).
$$
Given $T$ images $x= \pa{x^t}_{t=1}^T\in \RR^{d\times d \times T}$, we can consider the multi-channel total variation regularisation function by defining
$$
L:\RR^{d\times d\times T} \to \RR^{d\times d\times 2T}, \quad Lx = \pa{(D^h x^t, D^v x^t)}_{t=1}^T
$$
$$
\qandq R(x) = \norm{Lx}_{2,1} = \sum_{i,j=1}^d   \sqrt{ \sum_{t=1}^T \pa{D^h x^t}_{i,j}^2 + \pa{D^v x^t}_{i,j}^2 }
$$
Suppose also that $A:\RR^{d\times d\times T}\to\RR^{\sum_{t=1}^T m_t}$ is of the form $A(x) = \pa{A_t(x_t)}_{t=1}^T$ where $A_t:\RR^{d\times d} \to \RR^{m_t}$ is a linear operator and $y = (y_t)_{t=1}^T$ with $y_t\in\RR^{m_t}$ are the observations.  In this case, the function $f$ is defined over $v\in\RR^{d\times d}$ and the linear system in \eqref{eq:linsys} can be written as $T$ linear systems,
\begin{align*}
\begin{pmatrix}
A_t^\top A_t &  L^\top\\
L &- \begin{pmatrix}
\diag(v_t^2)& 0\\
0&\diag(v_t)^2
\end{pmatrix} 
\end{pmatrix} \begin{pmatrix}
x_t\\
\alpha_t
\end{pmatrix} = \begin{pmatrix}
A_t^\top y_t\\
0
\end{pmatrix}.
\end{align*}
These linear systems can be placed in the form \eqref{eq:toinvert} and solved simultaneously. Two popular settings for image processing with total variation is  denoising  where $A=\Id$ and  inpainting where $A$ is a masking operator.

We compare the VarPro formulation (using a BFGS solver) against Primal-Dual and ADMM for the following two problems:
\begin{itemize}
\item[(i)] Total variation inpainting of colour images, where we let $T=3$ and for $x = (x^1, x^2, x^3)\in \RR^{d\times d \times 3}$, each $x^i$ correspond to one of  three colour channels. The operator $A x= x_J$ is a subsampling operator, where $J$ is an index set selecting 30\% of the pixels at random.
Figures \ref{fig:hyper} and~\ref{fig:tv-display} (rows 1 and 2), show convergence curves and examples of reconstructions.
\item[(ii)] Hyperspectral imaging. We consider total variation denoising on the Indian pines dataset \footnote{Dataset downloaded from \url{http://www.ehu.eus/ccwintco/index.php/Hyperspectral_Remote_Sensing_Scenes}}. Here, $A=\Id$ and  the input data is normalized to take values between 0 and 1. This dataset is of size $145\times 145$ with $T=224$ spectral reflectance bands. 
Figures \ref{fig:hyper} and~\ref{fig:tv-display} (row 3), show convergence curves and examples of reconstructions.
\end{itemize}
One of the advantages of our method is that one can simply plug our gradient formula for $f$ into any gradient based method such as quasi-Newton BFGS, without the need to tune extra parameters. Thus, although ADMM and Primal-Dual do show favourable performance for certain parameter choices, we found that the VarPro is more straightforward to apply.

\begin{figure}
\begin{tabular}{c@{\hspace{1pt}}c@{\hspace{1pt}}c}
$\lambda = 0.1$&$\lambda = 0.5$&$\lambda = 1$\\
\includegraphics[width=0.32\linewidth]{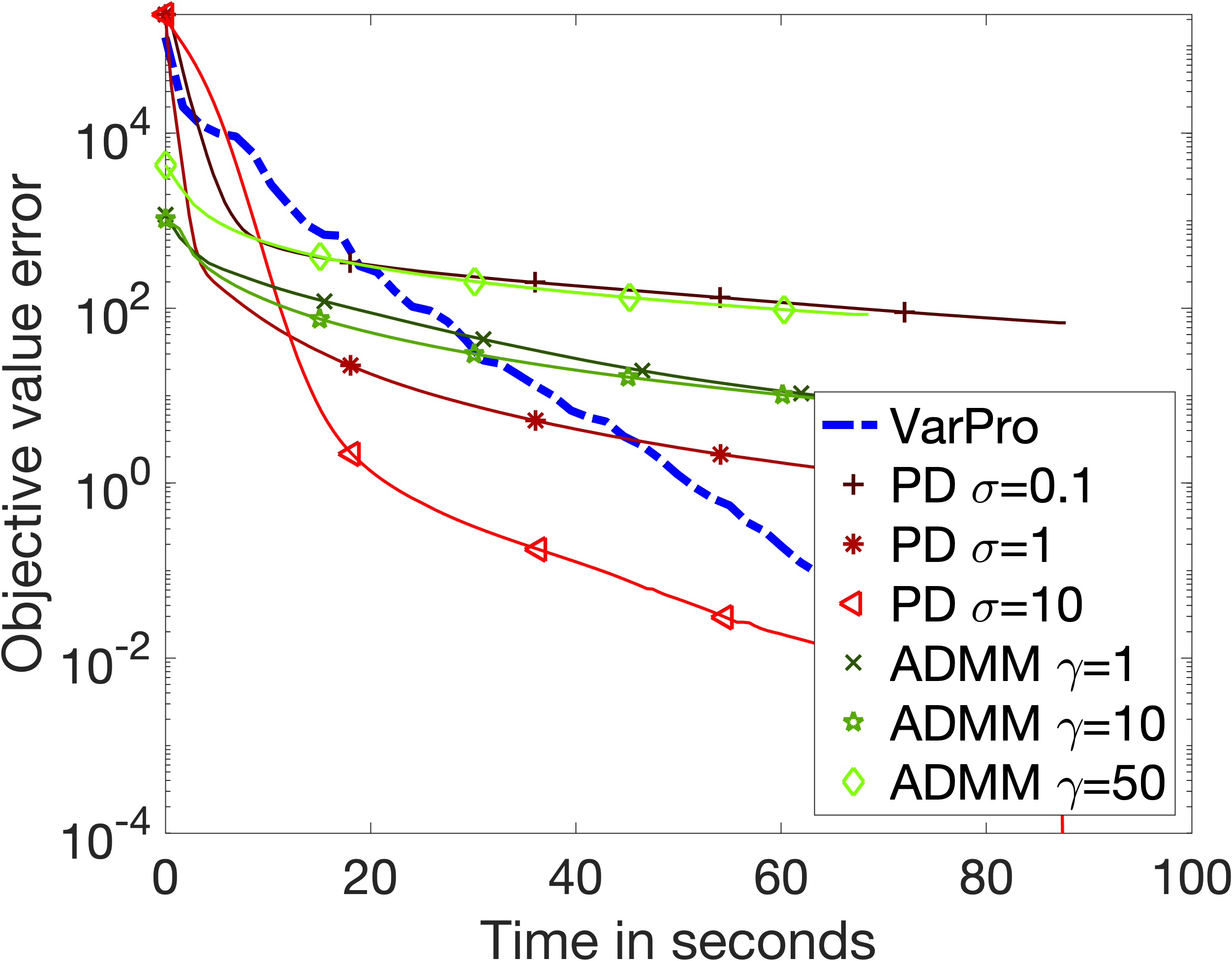}
&\includegraphics[width=0.32\linewidth]{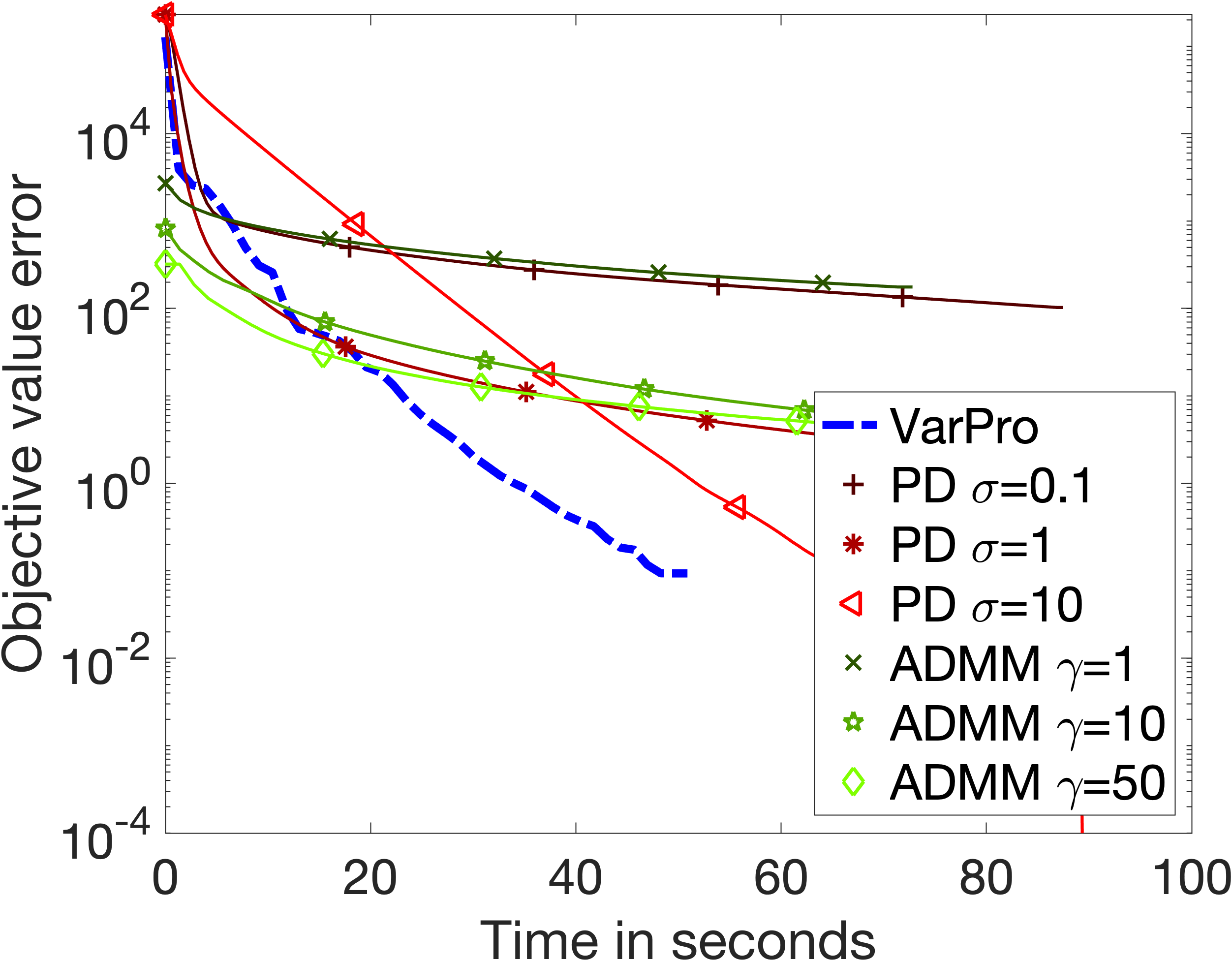}
&\includegraphics[width=0.32\linewidth]{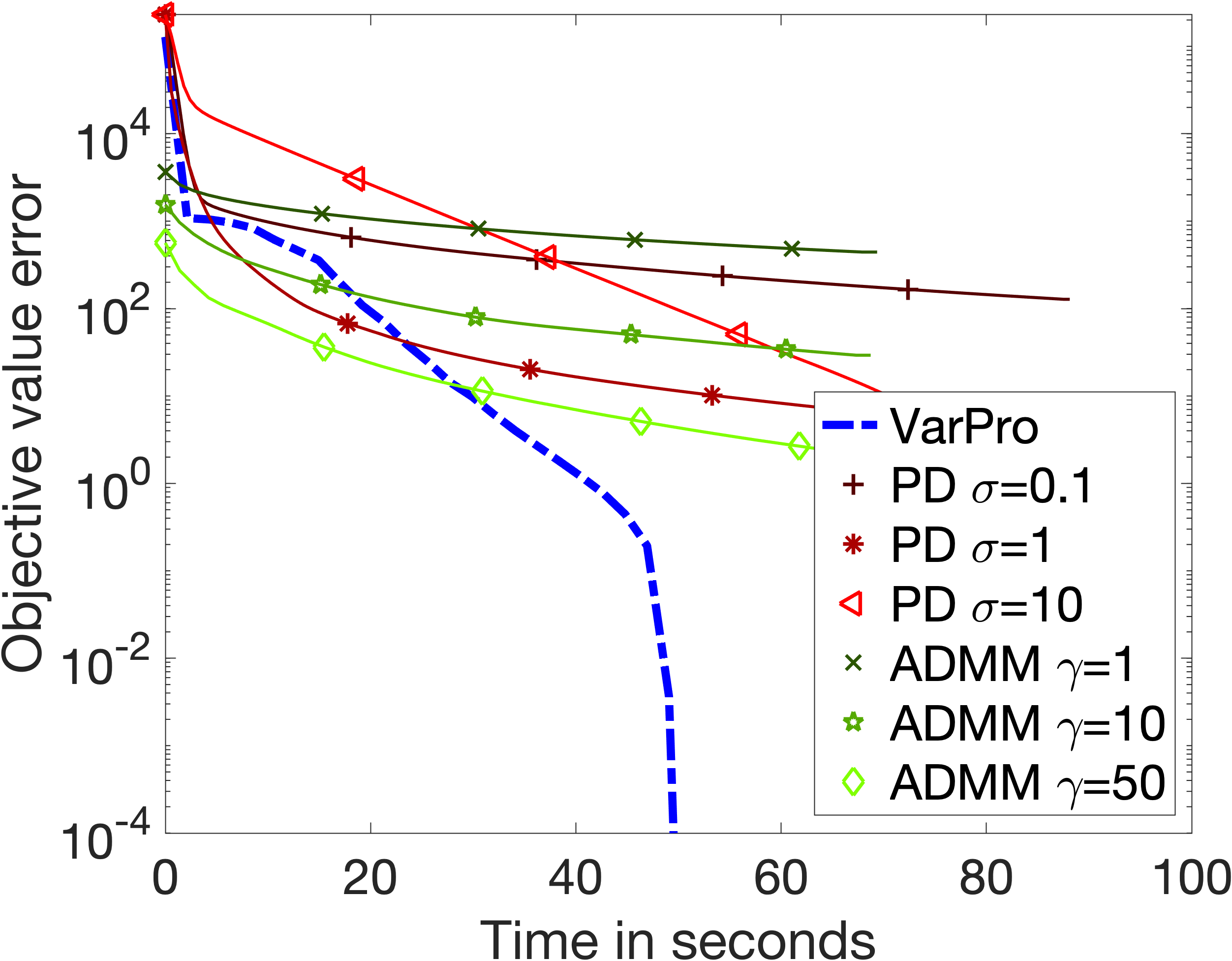}
\\
\includegraphics[width=0.32\linewidth]{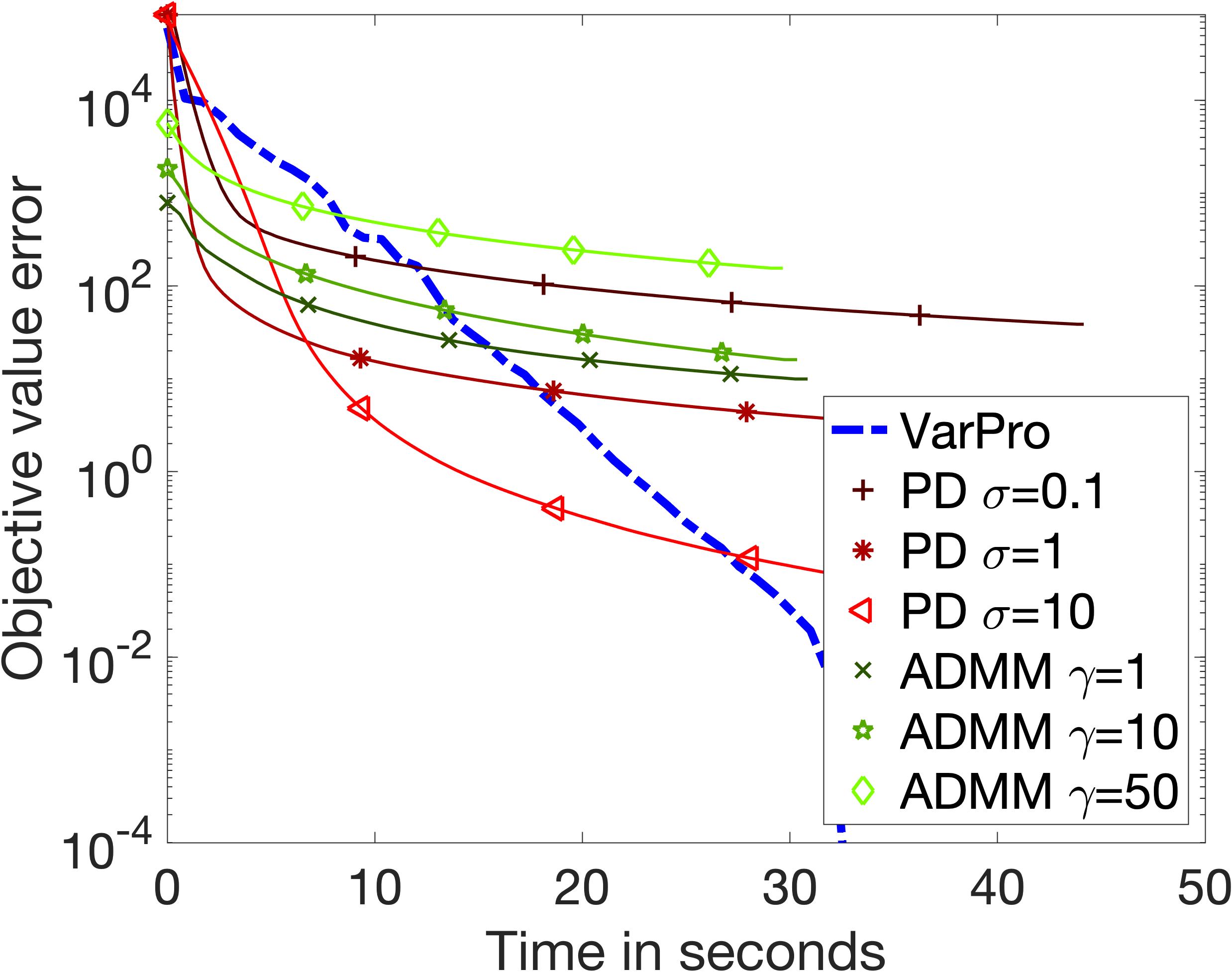}
&\includegraphics[width=0.32\linewidth]{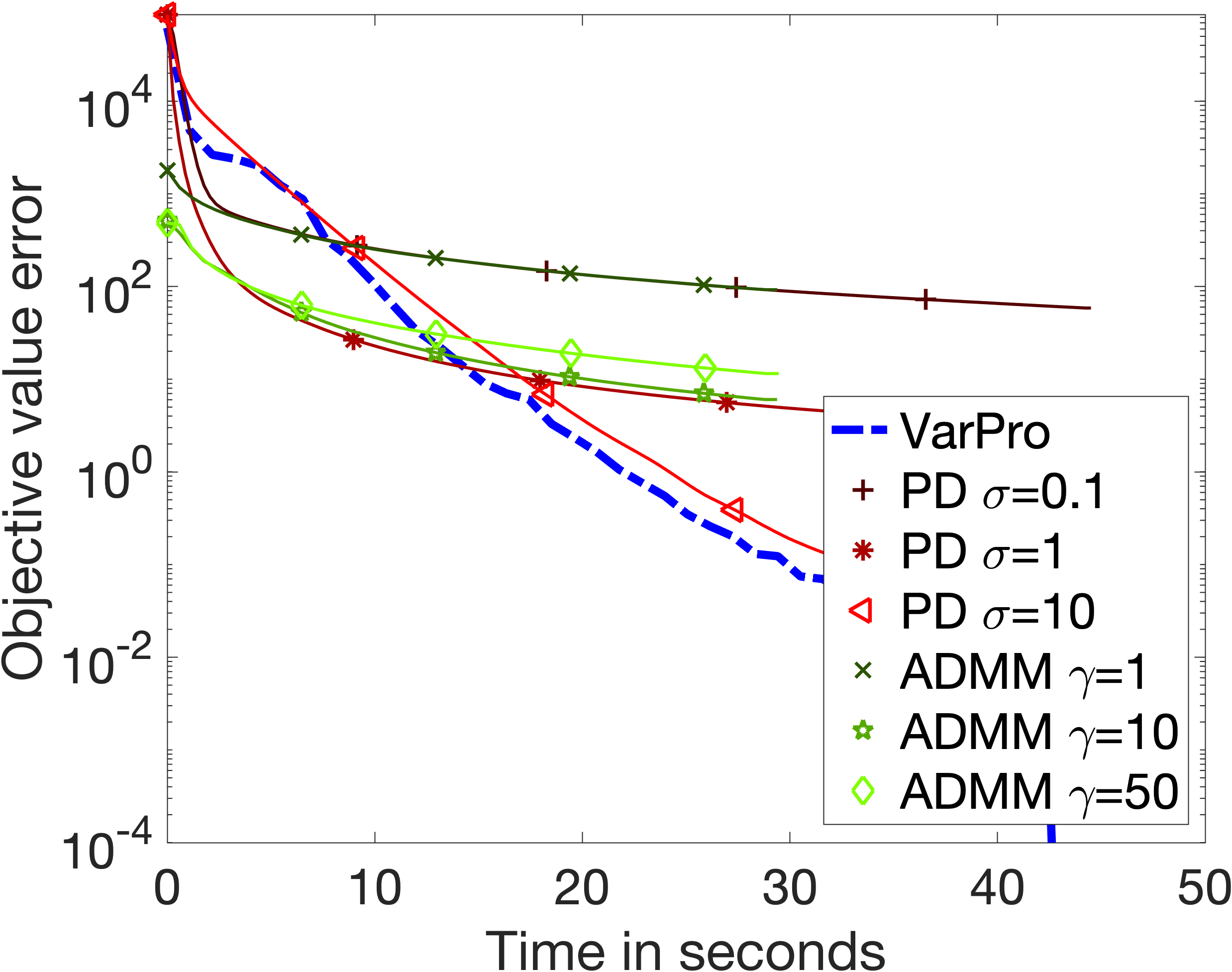}
&\includegraphics[width=0.32\linewidth]{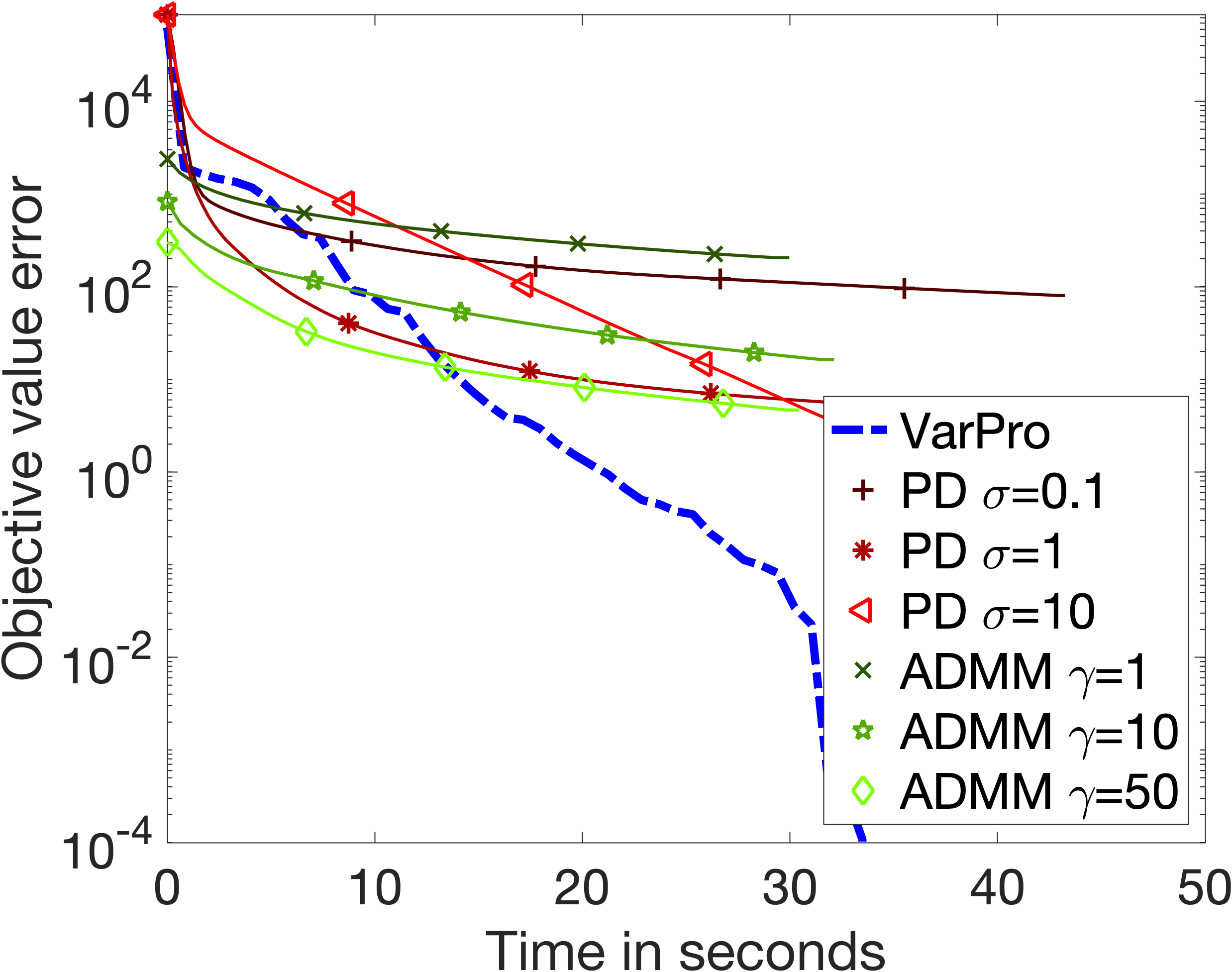}
\\
\includegraphics[width=0.32\linewidth]{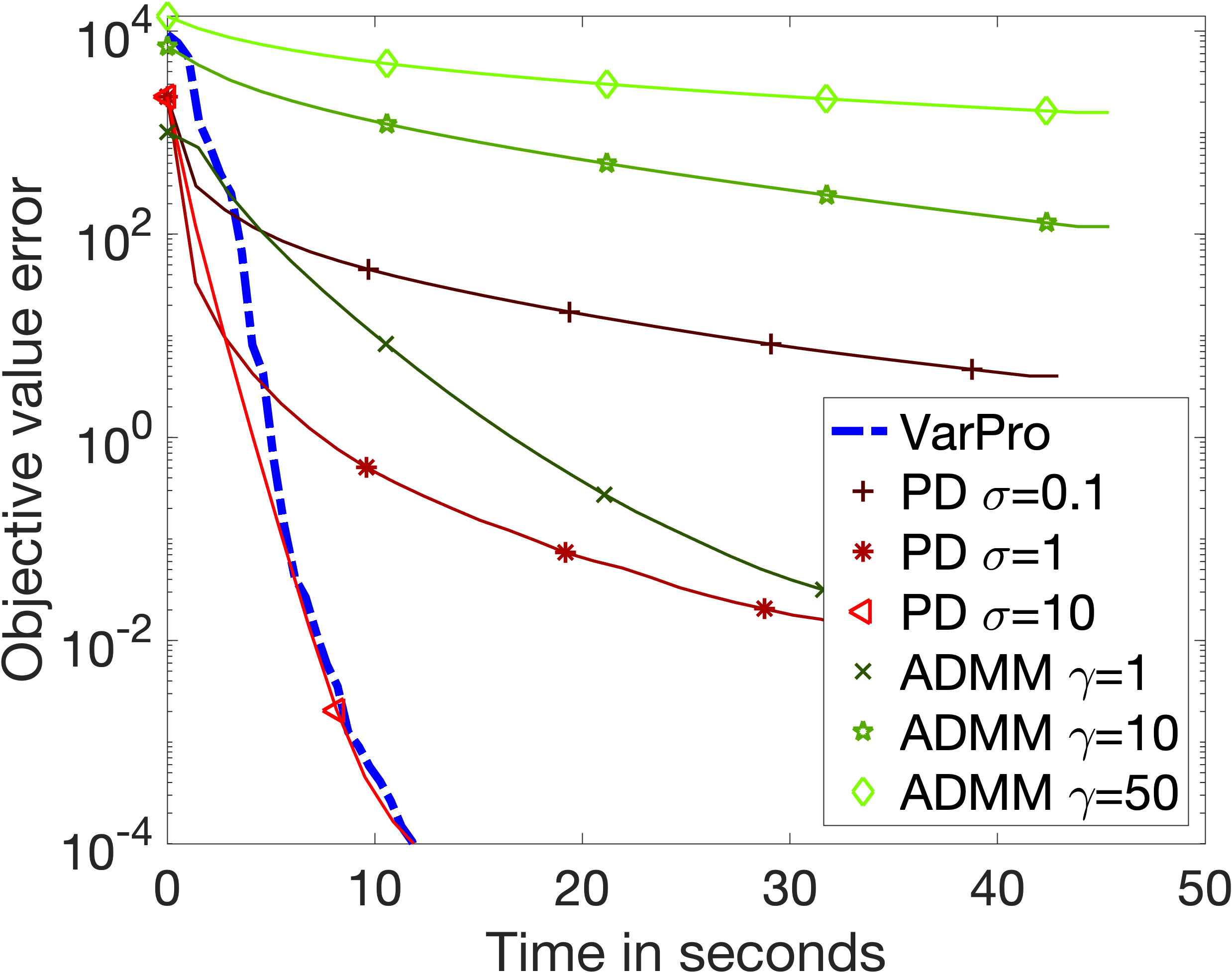}&
\includegraphics[width=0.32\linewidth]{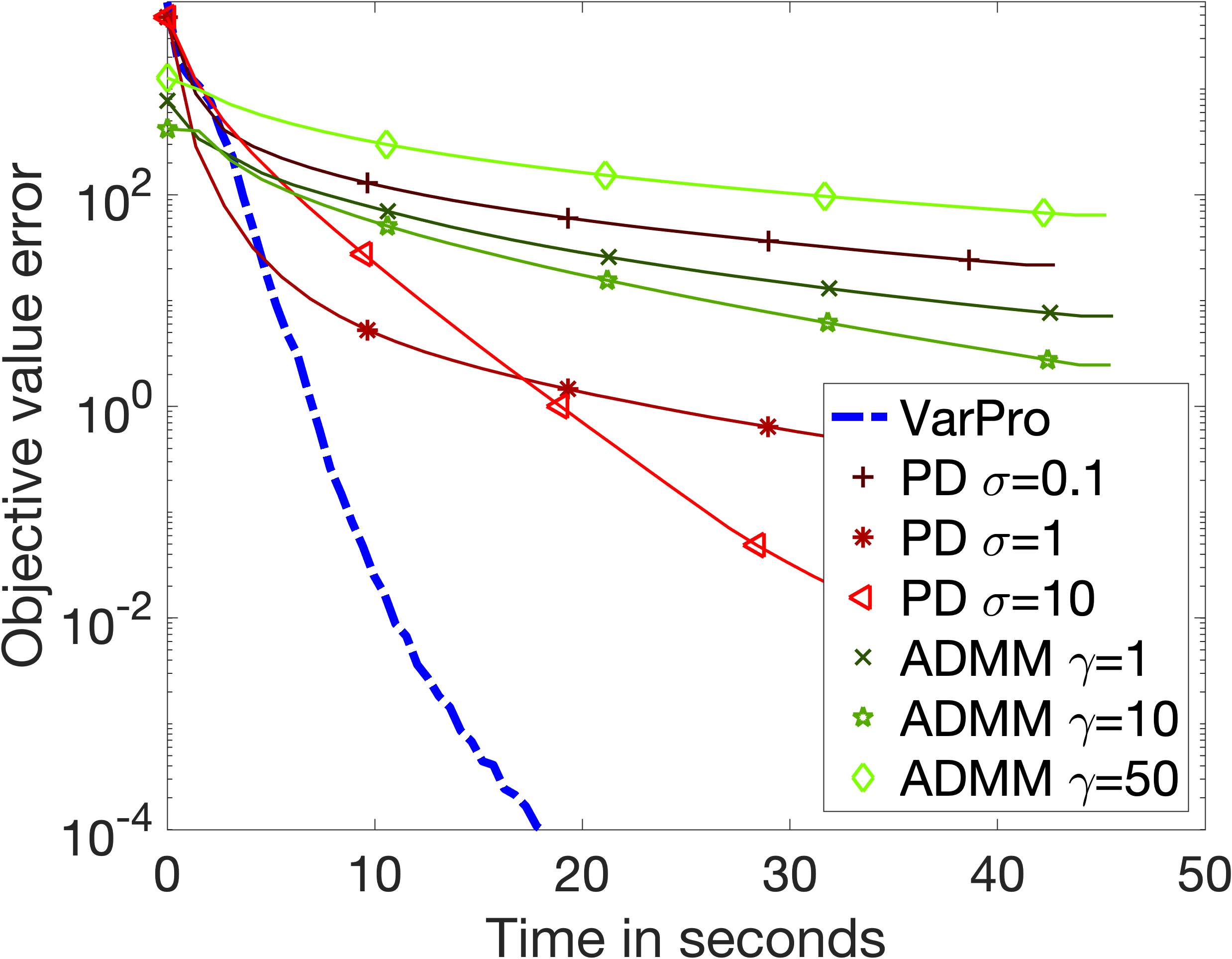}&
\includegraphics[width=0.32\linewidth]{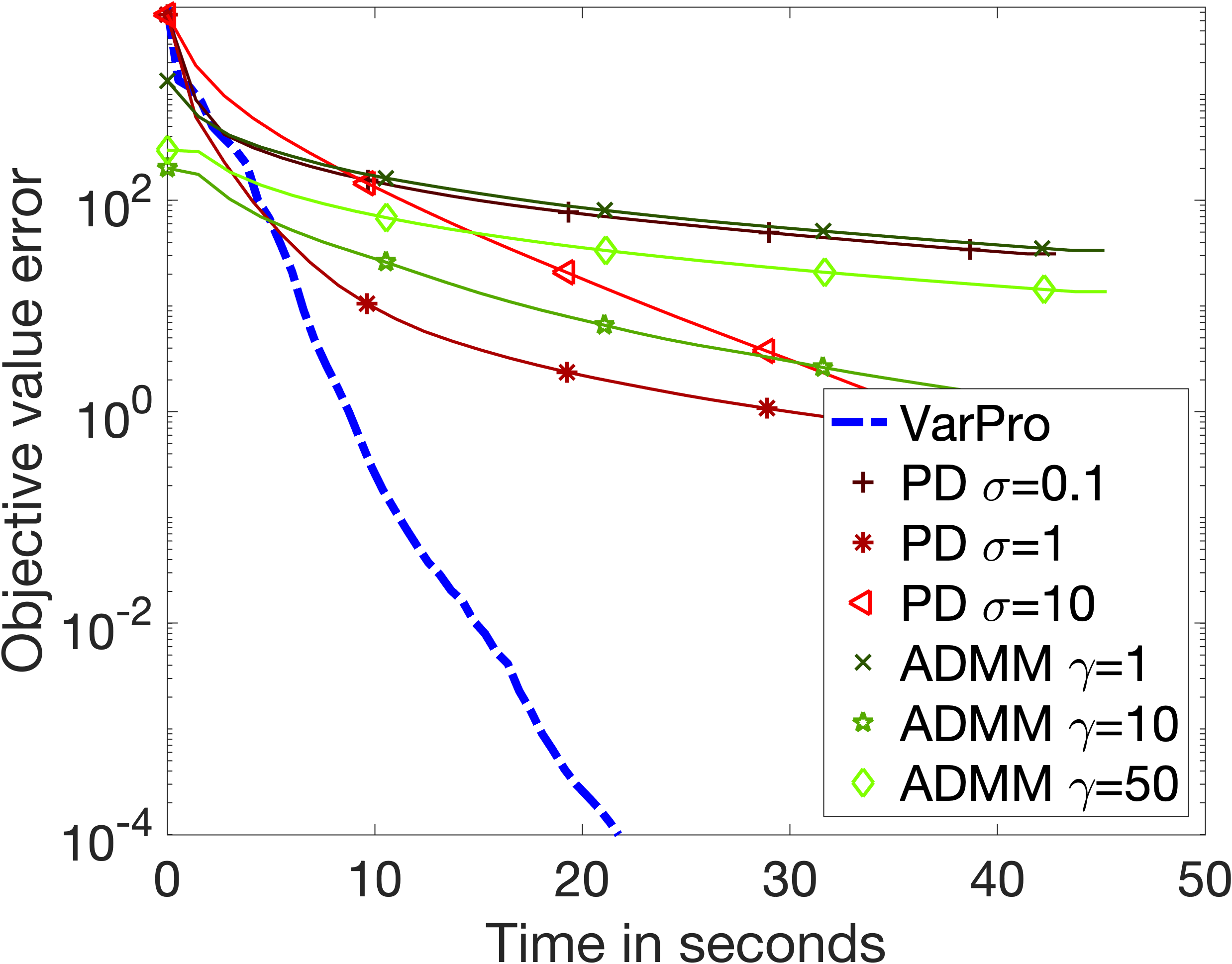}
\end{tabular}

\caption{
The top two rows show plots of the objective error against computational time in seconds for
2D vectorial TV colour inpainting with Matlab stock images  ``pears''  of    $486\times 732$ pixels (top row) and ``Peppers'' of $384\times 512$ pixels (middle row).  The regularisation of vectorial TV of $u = (u_1,u_2,u_3)$ is $J(u) = \sum_{i,j} \norm{\hat u_{i,j}}_2$ where $\hat u= (D_v(u_k),D_h(u_k))_{k=1,2,3}$ and so, $\hat u_{ij}$ is a vector of length 6, corresponding to the vertical and horizontal gradient of pixel $(i,j)$ of each of the 3 colour channels. Comparisons are against ADMM and Primal-Dual (PD). For Primal-Dual, the matrix inversion is trivial.  For ADMM, we need to invert $M=A^\top A + \gamma L^\top D$, for this inversion,  one  cholesky factorizaton is computed at the start of the iterations and is  re-used throughout.
The bottom row is a
Hyperspectral imaging example. Here, we perform total variation denoising with $A=\Id$ on the Indian Pines dataset. This consists of images of size 145 by 145, with 224 spectral reflectance bands. \label{fig:hyper}  }
\end{figure}

\begin{figure}
\begin{tabular}{c@{\hspace{1pt}}c@{\hspace{1pt}}c@{\hspace{1pt}}c}
$\lambda = 0.1$ & $\lambda = 0.5$&
$\lambda = 1.0$ & Input\\
\includegraphics[width=0.22\linewidth]{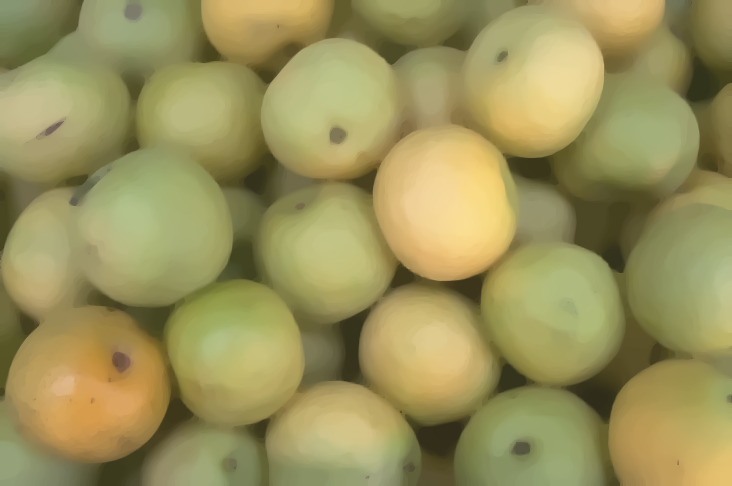}&
\includegraphics[width=0.22\linewidth]{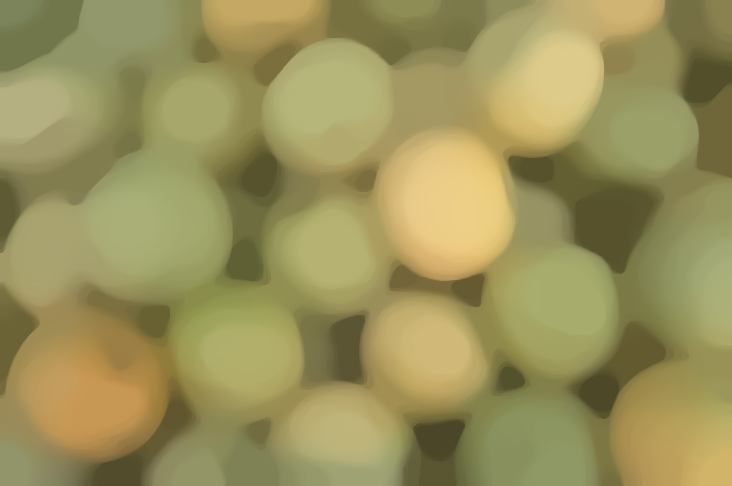}&
\includegraphics[width=0.22\linewidth]{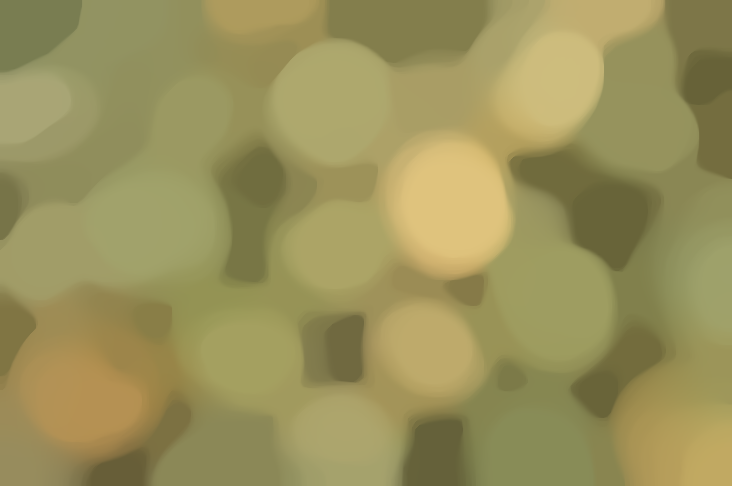}&
\includegraphics[width=0.22\linewidth]{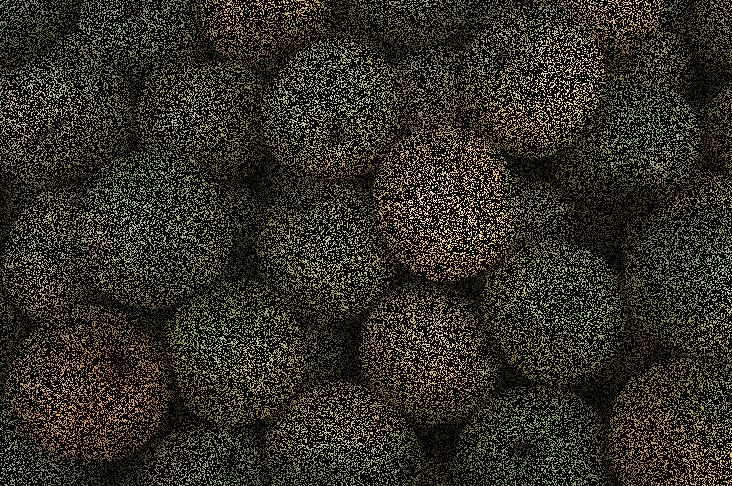}
\\
\includegraphics[width=0.22\linewidth]{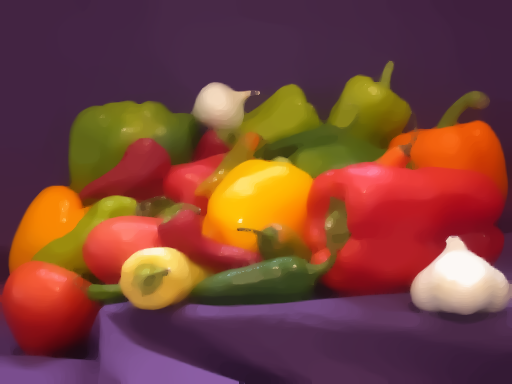}&
\includegraphics[width=0.22\linewidth]{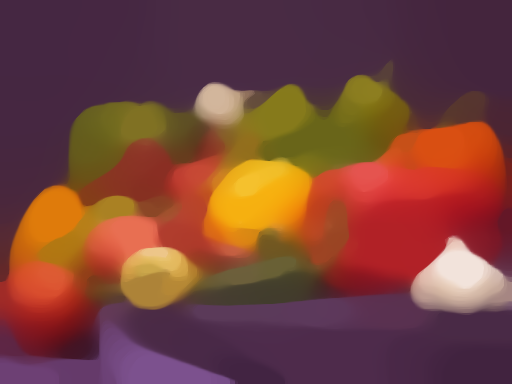}&
\includegraphics[width=0.22\linewidth]{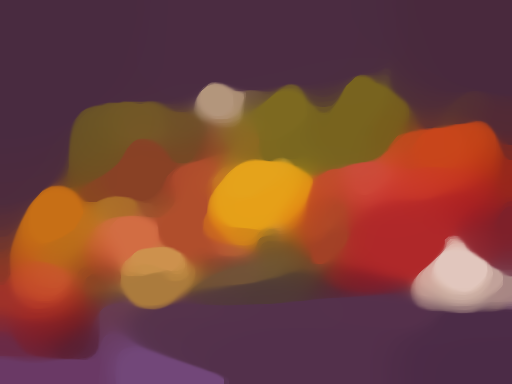}&
\includegraphics[width=0.22\linewidth]{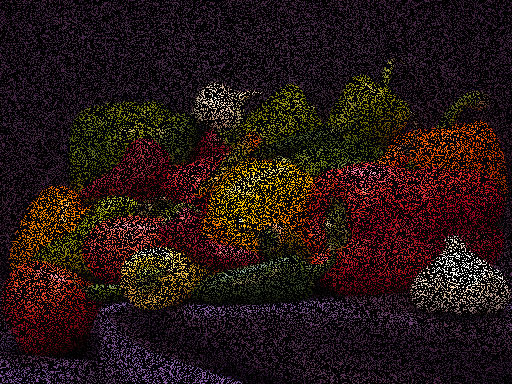}\\
\includegraphics[width=0.22\linewidth]{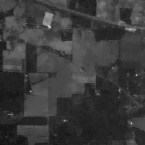}&
\includegraphics[width=0.22\linewidth]{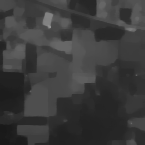}&
\includegraphics[width=0.22\linewidth]{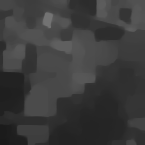}&
\includegraphics[width=0.22\linewidth]{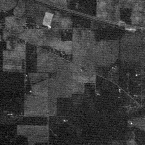}
\end{tabular}
\caption{  \label{fig:tv-display}
Top and middle row: color-TV reconstructions of ``pears''and ``peppers''. Bottom row: multi-channel TV reconstructions of the 3rd channel of the Indian Pines dataset.}
\end{figure}

%%%%%%%%%%%%%%%%%%%%%%%%%%%%%%%%%%%%%%%%%%%%%%%%%%%%%%%%%%%%%%%%%%%%%%%%%%%%%
\subsection{Proof of differentiability} \label{proofs}
In this section, we prove Theorem \ref{thm:diff}. 
In order to establish that $f$ is differentiable, we need to show that
 $$f_0(v) = \max_{\alpha,\xi} \phi(v,\xi,\alpha)$$
 is differentiable, where $\phi$ is defined in Proposition \ref{prop:gentv}. 
 We  first show that $f_0$ is strictly differentiable at $v$ if $v_i\neq 0$ for all $i$.

\textit{Notation:}
Throughout the proofs, given an index set $S$ and a vector $v$, let $v_S$ be the restriction of $v$ to the index set, such that $(v_S)_i = v_i$ for all $i\in S$ and $(v_S)_i = 0$ otherwise. Also, given a matrix $A$, $A_S$ denotes the restriction of the columns of $A$ by the indexes $S$.

\begin{prop}\label{prop:diff1}
 Let $v\in \RR^p$ and assume that $F_0$ is Lipschitz smooth. If  $v_i\neq 0$ for all $i$, then $f_0$ is strictly differentiable at $v$. 
\end{prop}

The proof of this result is a direct consequence of the following lemma.
\begin{lem}\label{lem:diff}
Let $B>0$ and let $$T \eqdef \enscond{(\xi,\alpha)\in \RR^m\times\RR^p}{ \norm{\alpha}\leq B, \norm{\xi}\leq B, \; L^\top \alpha = -A^\top \xi }.$$
The function
$$f_0(v) = \max_{\alpha,\xi\in T} \phi(v,\xi,\alpha)$$ is strictly differentiable with $\nabla f_0(v) = - v \alpha_v^2$ where
$$
(\alpha_v, \xi_v) \in \argmax_{\alpha,\xi\in T} \phi(v,\xi,\alpha)
$$
\end{lem}
\begin{proof}

Note that
$\phi$ and  $\partial_v \phi(v,\xi,\al)=  - v\odot  \alpha^2$ are continuous on $\RR^{\abs{\Gg}} \times T$. Moreover, due to the strongly concave $-\norm{v\odot\alpha}^2$ term inside $\phi$,  $v\odot\alpha_v$ is unique if $(\xi_v,\alpha_v)\in \argmax_{(\xi,\alpha)\in T} \phi(v,\xi,\alpha)$, and so, 
$$
\enscond{\partial_v \phi(v,\xi_v,\alpha_v)}{(\xi_v,\alpha_v)\in \argmax_{(\xi,\alpha)\in T} \phi(v,\xi,\alpha)}
$$
is single-valued. It follows from Theorem \ref{thm:rockafellar} that $f_0$ is differentiable with $\nabla f_0(v) = - v \odot\alpha_v^2$.
\end{proof}

\begin{proof}[Proof of Prop. \ref{prop:diff1}]
By Lemma \ref{lem:diff},
to show that $f_0(v) = \max_{\alpha\in\RR^p, \xi\in\RR^m} \phi(v,\xi,\alpha)$ is differentiable, it is sufficient to show that for each $v$, there exists a neighbourhood of $v$ and a bounded set $T$ such that for all $v\in V$,  
$f_0(v) = \max_{(\alpha, g) \in T} \phi(v,\xi,\alpha)$.

Suppose that $v_i \neq 0$ for all $i$ and there is a neighbourhood $V$ around $v$ on which $\abs{v_i'}> c>0$ for all $v'\in V$. Then, for all $v'\in V$, the optimal solution $(\xi,\alpha)$ to $\max_{\alpha,\xi} \phi(v',\xi,\alpha)$ satisfies
\begin{equation}\label{eq:bound_g}
-\frac12 \norm{v'\odot \alpha}_2^2 -  F_0^*(\xi) \geq \phi(v',0,0) - F_0^*(0) = - F_0^*(0) .
\end{equation}
Moreover, since $F_0^*$ is strongly convex, there exists $\gamma>0$  such that for any $z\in \partial F_0^*(0)$, 
$$
F_0^*(\xi) - F_0^*(0) - \dotp{z}{\xi} \geq \gamma \norm{\xi}^2
$$
and combining with \eqref{eq:bound_g}, we obtain
\begin{equation}\label{eq:unif_bound_g}
\dotp{z}{\xi}  \geq \frac12 \norm{v'\odot \alpha}_2^2   +  \gamma \norm{\xi}^2.
\end{equation}

Hence, $$\norm{\xi}_2\leq \frac{1}{\gamma} \norm{z} \qandq c^2 \norm{\al}^2 < \min_i\abs{v_i'}^2\norm{\alpha}^2 \leq  2\norm{z}\norm{\xi} \leq \frac{2}{\gamma}\norm{z}^2.$$ 
We can therefore take $B =  \max\pa{\frac{1}{\gamma}\norm{z} ,  \frac{\sqrt{2}}{c\sqrt{\gamma}}  }  $.
\end{proof}

We now show that $f_0$ is  differentiable for any $v\in \RR^p$.

\begin{prop}\label{prop:diff_f0}
	Let $F_0\in \Cc^{1,1}$  and $v\in\RR^p$. Then, given $$(\alpha_v,\xi_v) \in \argmax_{\al,\xi}\phi(v,\al,\xi),$$
  $v\odot\alpha_v^2$ is uniquely defined, and $f_0$ is differentiable with $\nabla f_0(v) = -v\odot\alpha_v^2$.
\end{prop}

\begin{proof}
We break the proof into several steps.
\paragraph{Step 1: $f_0$ is strictly continuous when restricted to the support of $v$}

By strong convexity of $F_0^*$, we know from \eqref{eq:bound_g} and \eqref{eq:unif_bound_g}  that any maximiser $(\alpha,\xi)$ to $\max_{\al,\xi}\phi(v,\xi,\al)$ satisfy $\norm{\xi}\leq \norm{z}/\gamma$, where $\gamma$ is the strong-convexity constant of $F_0^*$ and $z\in \partial F_0^*(0)$. Also, letting $S=\Supp(v)$, we have $\frac12 \norm{v_S\odot \alpha_S}^2 \leq \norm{z} \norm{\xi} \leq \norm{z}^2/\gamma$. So, there exists a constant $C>0$  and a neighbourhood $V$ around $v$  such that  for all $w\in V$, the maximisers to $\max_{\al,\xi}\phi(w,\xi,\al)$ satisfy
$$
\norm{\alpha_S} \leq C \qandq \norm{\xi}\leq C.
$$
So, for all $w\in V$,
we can  restrict the maximisation of $\phi(w,\cdot,\cdot)$ to the set $$T\eqdef \enscond{(\xi,\alpha)}{\norm{\xi}\leq C, \; \norm{\al_S}\leq C, \; A^\top \xi+L^\top \alpha= 0}$$ and write
$$
f_0(v) = \sup_{(\al,\xi)\in T}\phi(v,\xi,\al).
$$

%Let $(\al_1,\xi_1) \in  \argmax_{\al,\xi\in T} \phi(w,\al,\xi)$ and $(\al_0,\xi_0) \in   \argmax_{\al,\xi\in T} \phi(v,\al,\xi)$ and $w_S =v =  v_S$.
%Note that
%\begin{align*}
%&F_0(w) =-\frac12 \norm{w_{S^c} (\al_1)_{S^c}}^2 + \phi(v,\al_1,\xi_1) = \phi(w,\alpha_1,\xi_1) \geq \phi(w,\al_0,\xi_0)\\
%& = -\frac12 \norm{w_{S^c} (\al_0)_{S^c}}^2 + \phi(v,\al_0,\xi_0) = -\frac12 \norm{w_{S^c} (\al_0)_{S^c}}^2 + F_0(v)\\
%&\geq  -\frac12 \norm{w_{S^c} (\al_0)_{S^c}}^2 + \phi(v,\al_1,\xi_1)
%\end{align*}
%where the first inequality follows as $(\al_1,\xi_1)$ are maximisers of $\phi(w,\cdot,\cdot)$ and the second inerqualtiy follows since $(\al_0,\xi_0)$ are maximisers to $\phi(v,\cdot,\cdot)$.
%So, we have
%$$
%F_0(w) \geq -\frac12 \norm{w_{S^c} (\al_0)_{S^c}}^2 + F_0(v) \qandq   \norm{w_{S^c} (\al_0)_{S^c}}^2 \geq  \norm{w_{S^c} (\al_1)_{S^c}}^2 
%$$
%Moreover, since for all $\al,\xi$,
%$$
%\phi(w,\al,\xi) \leq \phi(v,\al,\xi)
%$$
%we have $f_0(w)\leq F_0(v)$. It follows that
%$$
%\abs{F_0(w) - F_0(v)}\leq \frac12 \norm{\alpha_0}^2 \norm{v-w}_2^2,
%$$
%whenever $w_S = v_S$.

If $w,w'\in V$ and $\Supp(w) = \Supp(w') = S$,  since 
\begin{align*}
\forall (\xi,\al)\in T, \qquad \phi(w,\xi,\alpha) = \phi(w',\xi,\al) +  \frac12 \norm{w'\odot \al}^2 - \frac12 \norm{w\odot \al}^2,
\end{align*}
 we have
$$\abs{\phi(w,\xi,\al)- \phi(w',\xi,\al)} \leq \frac12 C^2 \max_{i\in S} \abs{{w_i'}^2 - {w_i}^2} \leq C^2 B \norm{w'-w},$$
where $B = \sup\enscond{\norm{w}}{w\in V}$.  Hence,
\begin{align*}
\abs{f_0(w)-f_0(w')} & \leq  C^2  B \norm{w'- w}.
\end{align*}

\paragraph{Step 2: Continuity of maximisers}

Let $v\in \RR^p$ and $(\alpha_v,\xi_v)\in \argmax_{\alpha,\xi} \phi(v,\alpha,\xi)$.  We will show that $\alpha_v$ is continuous when restricted to the support of $v$ which we denote by $S\eqdef  \Supp(v)$.  We shall see that this again follows because $\phi(w,\al, g)$ is strongly convex with respect to $\alpha_S$ and $\xi$. There is a neighbourhood  $V$ around $v$, such that for all $w\in V$ with $\mathrm{Supp}(w) = S$,
$$
\phi(w,\alpha_w, \xi_w) - \phi(w,\alpha_v, \xi_v)  \gtrsim  \norm{(\alpha_v - \alpha_w)_S}^2 + \norm{\xi_v-\xi_w}^2.
$$ 
Indeed, since $(\alpha_w,\xi_w)$ maximises $\phi(w,\cdot,\cdot)$, there exists $x$ such that $Ax \in \partial F_0^*(\xi_w)$ and $A x = w^2\alpha_w$ and for any $z_w\in \partial F_0^*(\xi_w)$. Using the fact that $F_0^*$ is $\gamma$-strongly convex,
\begin{align*}
&\phi(w,\al_w,\xi_w) - \phi(w,\al_v,\xi_v) = \frac12 \dotp{w^2}{\al_v^2 - \al_w^2} - F_0^*(\xi_w) + F_0^*(\xi_v)\\
&\geq   \frac12 \dotp{w^2}{\al_v^2 - \al_w^2} + \dotp{\xi_v - \xi_w}{z_w}  + \gamma\norm{\xi_v-\xi_w}^2\\
&= \dotp{w^2 \al_w}{\al_v-\al_w} + \dotp{z_w}{\xi_v-\xi_w} +\frac12 \norm{w\odot(\al_v-\al_w)}^2+ \gamma\norm{\xi_v-\xi_w}^2\\
&= \dotp{x}{L^\top \al_v - L^\top \al_w} + \dotp{x}{A^\top \xi_v - A^\top \xi_w}+\frac12 \norm{w\odot(\alpha_v - \alpha_w)_S}^2+ \gamma\norm{\xi_v-\xi_w}^2\\
&\gtrsim \frac12 \norm{ (\alpha_v - \alpha_w)}^2+ \gamma\norm{\xi_v-\xi_w}^2.
\end{align*}

Moreover, for all $w\in V$ with $\Supp(w) = \Supp(v)$,
$$
\abs{\phi(v,\alpha_v,\xi_v) - \phi(w,\alpha_v, \xi_v)} \leq \sum_i (\alpha_v)_i^2 \abs{v_i^2-w_i^2} \lesssim \norm{v-w}.
$$
It follows that for all $w\in V$ with $\Supp(w) = S$,
\begin{align*}
\norm{(\alpha_v - \alpha_w)_S}^2 &\lesssim \phi(w,\alpha_w, \xi_w) - \phi(w,\alpha_v, \xi_v) \\
&\lesssim  f_0(w) - f_0(v) + \norm{v-w} \lesssim \norm{v-w},
\end{align*}
%Similarly, since $\phi(v,\alpha,\xi)$ is strongly convex with respect to $\xi$, one can also show that 
and also,
$$
\norm{\xi_v- \xi_w} \lesssim \norm{v-w}
$$
for all $w\in V$ with $\Supp(w) = S$.

\paragraph{Step 3: formula on directional derivatives}

Let $v,w\in\RR^p$. Let  $t\in\RR$ and $w_t\in \RR^p$ be such that $w_t\to w$ as $t\to 0$. 
Given
$\al_0,\xi_0 \in \argmax_{\al,\xi} \phi(v, \al,\xi)$,
\begin{align*}
f_0(v) &= \frac12 \norm{\alpha_0 \odot (v+tw_t)}^2 -\frac12 \norm{\alpha_0\odot v}^2+ \phi(v+tw_t, \al_0,\xi_0)\\
&\leq t\dotp{\al_0^2 \odot v}{w_t} +\frac{t^2}{2} \norm{\alpha_0\odot w_t}^2+ f_0(v+tw_t)
\end{align*}
So,
$$
\frac{f_0(v+tw_t)-f_0(v) }{t}+ \dotp{v\odot \alpha_0^2}{w_t} \geq -\frac12 t^2 \norm{\al_0 \odot w_t}^2.
$$
On the other hand,  $S\eqdef \Supp(v)$ and  $(\alpha_t, \xi_t) \in \argmax_{\al,\xi} \phi(v+t(w_t)_S, \alpha,\xi)$, we have
\begin{align*}
f_0(v+tw) \leq f_0(v+t(w_t)_S) &= -\frac12 \norm{\alpha_t \odot (v+t(w_t)_S)}^2 +\frac12 \norm{\alpha_t\odot v}^2 + \phi(v, \alpha_t, \xi_t)\\
&\leq -\dotp{\alpha_t\odot v}{w_t} - \frac{t^2}{2} \norm{\alpha_t\odot w_t}_2^2 +  f_0(v)
\end{align*}
Note that $\alpha_t\to \alpha_0$ as $t\to 0$ due to the continuity of maximisers proved in Step 2.
It follows that
\begin{align*}
\lim_{t\to 0} \frac{f_0(v+tw_t) - f_0(v)}{t} = \dotp{\alpha_0^2 \odot v}{w_0}.
\end{align*}
So, $f_0$ is semi-differentiable and since the directional derivative is  linear with respect to $w$, it follows that $f_0$ is differentiable (see 7.21 and 7.22 of \cite{rockafellar2009variational}).

\paragraph{Additional claim: $f$ is continuous at $v$ (in particular, it is calm)}
Given $v\in\RR^p$, let the neighbourhood $V$ and set $T$ be as in Step 1.
Now, for $w\in V$, 
let $(\al_1,\xi_1) \in  \argmax_{\al,\xi\in T} \phi(w,\al,\xi)$ and  let $(\al_0,\xi_0) \in   \argmax_{\al,\xi\in T} \phi(v,\al,\xi)$.
Note that
\begin{align*}
&f_0(w) =-\frac12 \norm{w_{S^c} \odot (\al_1)_{S^c}}^2 + \phi(w_S,\al_1,\xi_1) = \phi(w,\alpha_1,\xi_1) \geq \phi(w,\al_0,\xi_0)\\
& = -\frac12 \norm{w_{S^c}\odot (\al_0)_{S^c}}^2 + \phi(w_S,\al_0,\xi_0) 
\end{align*}
Note that $\abs{\phi(w_S,\xi_0,\al_0) - \phi(v,\xi_0,\alpha_0) }\lesssim \norm{w_S - v}$, it follows that
\begin{align*}
&-\frac12 \norm{w_{S^c}\odot(\al_0)_{S^c}}^2 + \phi(w_S,\al_0,\xi_0) \\
&\geq  -\frac12 \norm{w_{S^c} \odot (\al_0)_{S^c}}^2 - C  \norm{w_S - v} +\phi(v,\al_0,\xi_0)\\
& =  -\frac12 \norm{w_{S^c}\odot (\al_0)_{S^c}}^2 - C  \norm{w_S - v} +f_0(v) \\
&\geq   -\frac12 \norm{w_{S^c} (\al_0)_{S^c}}^2 - C  \norm{w_S - v}   + \phi(v,\al_1,\xi_1)\\
&=  -\frac12 \norm{w_{S^c}\odot (\al_0)_{S^c}}^2 - C  \norm{w_S - v}  -\frac12\sum_{i\in S} (v_i^2-w_i^2) (\al_1)_i^2  + \phi(w_S,\al_1,\xi_1)\\
&\geq  -\frac12 \norm{w_{S^c}\odot (\al_0)_{S^c}}^2 - 2C  \norm{w_S - v}   + \phi(w_S,\al_1,\xi_1)
\end{align*}
where we made use of optimality of $(\al_0,\xi_0)$ for the second inequality.
It follows that
\begin{align*}
\norm{w_{S^c} \odot(\al_0)_{S^c}}^2 + 2C \norm{w_S - v} \geq  \norm{w_{S^c} (\al_1)_{S^c}}^2 
\end{align*}
and
\begin{align*}
f_0(w) \geq  f_0(v) - C \norm{w_S - v} - \frac12 \norm{w_{S^c}\odot (\alpha_0)_{S^c}}^2
\end{align*}
Finally, since for all $\al,\xi$,
$$
\phi(w,\al,\xi) \leq \phi(w_S,\al,\xi) 
$$
we have $f_0(w) \leq f_0(w_S)$ and hence, $f_0(w)\leq f_0(v) + C \norm{v_S - w_S}$. It follows that
$$
\abs{f_0(w) - f_0(v)}\leq \frac12 \norm{\alpha_0}^2 \norm{w_{S^c}}_2^2 +  C \norm{v_S - w_S}.
$$
\end{proof}

%%%%%%%%%%%%%%%%%%%%%%%%%%%%%%%%%%%%%%%%%%%%%%%%%%%%%%%
\subsection{Differentiability for basis pursuit}\label{sec:diff_basis_pursuit}

Differentiability when $F_0$ is not $\Cc^{1,1}$ is more delicate. We consider the case where $F_0(z) = \iota_{\{y\}}(x)$ here, which corresponds to the so-called basis pursuit problem, where one imposes exact reconstruction $Ax=y$.  In this case,
$$
\phi(v,\xi,\alpha) = -\frac12 \norm{v\odot \alpha}^2 - \dotp{\xi}{y}
$$
By setting $v =0$,
one can see that $\sup_{\alpha,\xi} \phi(0,\xi,\al) = \sup_{L^\top \al = -A^\top \xi} -\dotp{\xi}{y} = +\infty$ and the domain of $f$ cannot be the entire space. 

\begin{thm}
Let $F_0(z) = \iota_{\{y\}}$ and suppose that $y = Ax$ for some $x \in \RR^n$. Suppose that $v\in\RR^p$ satisfies $\Supp(v)\supset \Supp(A x)$. Then, $v\in \dom(f)$, a maximiser  to $\sup_{\al,\xi}\phi(v,\al,\xi)$ exist, and $f$ is differentiable at $v$.
\end{thm}
\begin{proof}

Note that we can write
$$
f_0(v)=\max_{L^\top \alpha\in \Rr(A^\top)} \psi(v,\alpha), \qwhereq \psi(v,\alpha) =  -\frac12 \norm{v\odot \al}^2 +\dotp{ \alpha}{A x}
$$
If $\Supp(A x) \subset\Supp(v)\eqdef S$, 
$$
\psi(v,\alpha) \leq \norm{\alpha_S}\pa{ -\frac12 \min_{i\in S} \abs{v_i}^2 \norm{\alpha_S} +  \norm{A x}} \leq \frac{\norm{A x}^2}{2 \min_{i\in S} \abs{v_i}^2}.
$$
So, $f_0(v)<\infty$.  Let $\alpha_n$ be a maximising sequence, note that $\ens{(\alpha_n)_S}_n$ is uniformly bounded since $\psi(v,\cdot)$ is strongly concave, so, there exists $\alpha_*$ such that  $(\alpha_n)_S \to (\alpha_*)_S$. Since $L^\top (\alpha_n)_S \in \Rr\pa{[L^\top_{S^c}, A^\top]}$ and is convergent, its limit is also in the $\Rr\pa{[L^\top_{S^c}, A^\top]}$. That is, there exists $(\alpha_*)_{S^c}$ and $\xi_*$ with $L^\top (\alpha_*)_S = L^\top (\alpha_*)_{S^c} + A^\top \xi_*$. One can finally conclude from upper semicontinuity of $\psi(v,\cdot)$ that $\alpha_*$ is a maximiser.

It remains to deduce that $f_0$ is differentiable at $v$. The proof is similar to that of Proposition \ref{prop:diff_f0}, we first show that on a neighbourhood $V$ of $v$, the mapping $v'\mapsto\alpha_{v'}$ is continuous when restricted to all $v'\in V$ with $\Supp(v') = \Supp(v)\eqdef S$. Indeed,
$$
\psi(v,\alpha_v) =  -\frac12 \norm{v\odot \al_v}^2 +\dotp{ \alpha_v}{A x} \geq  \psi(v,0) = 0,
$$
which implies that
$$
\frac12 \norm{v\odot \al_v}^2 \leq \norm{(\alpha_v)_{\Supp(A x)}} \norm{A x} \leq  \norm{(\alpha_v)_S}\norm{A x}
$$
since $\Supp(A x)\subseteq S$. It follows that on a neighbourhood around $v$, $v'\mapsto \norm{\alpha_{v'}}$ is uniformly bounded, just like step 1 of the proof of Proposition \ref{prop:diff_f0}. One can then show that $f_0$ is strictly continuous at $v$ when restricted to the support $S$.  Strong concavity with respect to $\norm{v\odot \alpha}^2$ also implies that  $v\mapsto \alpha_v$ is Lipschitz continuous. Using continuity of the maximisers, we can then compute the semiderivative of $f_0$ as in  Step 3 of the proof of Prop \ref{prop:diff_f0} to deduce that $f$ is differentiable. 
\end{proof}

% !TEX root = ../SINUM-VarPro.tex

% lipschitz constant
\newcommand{\Lip}{M}

\section{The fine grids settings}
\label{sec:finegrid-mirror}

Particularly challenging settings correspond to cases where the columns of $A$ are highly correlated. 
This is a typical situation for inverse problems in imaging sciences, and in particular deconvolution-type problems~\cite{candes2014towards,chizat2021convergence}.  
In these settings, $A$ arises from the discretization of some continuous operator, and the dimension $n$ grows as the grid refines. For the sake of concreteness, we consider an ideal low-pass filter in dimension $d$ (for instance $d=2$ for images), which is equivalent to the computation of low Fourier frequencies, up to some cut-off frequency $p$. The rows $A_k \in \RR^n$ of $A$  are indexed by $k=(k_i)_{i=1}^d\in [p]^d \eqdef \ens{0,\ldots, p}^d$, 
\begin{equation}\label{eq:fouriersys}
A_k = \phi\pa{\theta_k}, \qwhereq \theta_k\eqdef \frac{1}{p} (k_i)_{i=1}^d,\; \phi(\theta) \eqdef \frac{1}{m^{d/2}} \pa{e^{2\pi \sqrt{-1} \dotp{\theta}{ \ell}}}_{\ell\in [m/2]^d},
\end{equation}
where $[m/2]\eqdef\ens{-\frac{m}{2},\ldots, \frac{m}{2}}$, and so, $A$ corresponds to the Fourier operator  discretized on a uniform grid on $[0,1]^d$. 
To better cope with the ill-conditioning of the resulting optimizations problem, it is possible to use descent method according to some adapted metric. This can be conveniently achieved using so-called mirror descent scheme/Proximal Bregman Descent scheme, which we review below in Section \ref{sec:mirror_overview}, since this is closely linked to the Hadamard parameterization (as exposed in Section~\ref{sec:hypentropy-mirror}). 

As discussed below, in the mirror descent scheme,
the usual $\ell^2$ proximal gradient descent is retrieved when using a squared Euclidean entropy function. This Euclidean scheme suffers from an exponential dependency on $d$ in the convergence rate. Using non-quadratic entropy functions (such as the so-called hyperbolic entropy), together with a dimension-dependent parameter tuning, leads in sharp contrast to dimension independent rates \cite{chizat2021convergence}. 
After this review of mirror descent, we then analyse in Section~\ref{sec:hadamard-finegrid} the performance of gradient descent on the Hadamard parameterized function $G(u,v)$ on the case of the Lasso. The key observation is that the Lipschitz constant of $G$ is independent of the grid size $n$ and hence, one can derive dimension-free convergence rates on the gradient. Moreover, we draw in Section~\ref{sec:hypentropy-mirror} connections to mirror descent by showing that the continuous time limit (as the gradient descent stepsize tends to 0) corresponds to the mirror descent ODE with a hyperbolic entropy map whose parameter changes with time.

%%%%%%%%%%%%%%%%%%%%%%%%%%%%%%%%%%%%%%%%%%%%%%%%%%%%%%%%%%%%%%%%%%%%%%%%%%%%%%%%
%%%%%%%%%%%%%%%%%%%%%%%%%%%%%%%%%%%%%%%%%%%%%%%%%%%%%%%%%%%%%%%%%%%%%%%%%%%%%%%%
%%%%%%%%%%%%%%%%%%%%%%%%%%%%%%%%%%%%%%%%%%%%%%%%%%%%%%%%%%%%%%%%%%%%%%%%%%%%%%%%
\subsection{Overview of mirror descent}\label{sec:mirror_overview}

We consider a structured optimization problem of the form
\begin{equation}\label{eq-Phi-def}
\min_{x\in\RR^n} \Phi(x)\eqdef R(x) + F(x)
\end{equation}
where $R:\RR^n \to [0,\infty]$ is a (nonsmooth) convex function and  $F:\RR^n \to \RR$ is assumed to be convex and Lipschitz continuous on a closed convex set $\Xx\supset \mathrm{dom}(R)$ with
\begin{equation}\label{eq:mirror_lip_assumption}
\norm{\nabla F(x) - \nabla F(x')}_{\Xx^*} \leq \Lip \norm{x-x'}_{\Xx}
\end{equation}
where $\norm{\cdot}_{\Xx}$ is some norm on $\Xx$ and $\norm{\cdot}_{\Xx^*}$ is the dual norm.
This includes in particular sparsity regularized problems of the form~\eqref{eq:gen-tv}. 
A natural algorithm to consider is the Bregman proximal gradient descent method, of which the celebrated iterative soft thresholding algorithm  is a special case. In this section, we provide a brief overview of this method and the associated convergence results.

Given a strictly convex function (called an entropy function) $\eta:\Ee\to [-\infty,\infty)$ that is differentiable on  an open set $\Ee\supset \mathrm{int}(\Xx)$, its associated Bregman  divergence is defined to be
\begin{equation}
D_{\eta}(a,b)\eqdef \eta(a) - \eta(b) - \dotp{\eta'(b)}{a-b}.
\end{equation}
By possibly rescaling $\eta$, assume that
\begin{equation}\label{eq:pinkser}
D_\eta(a,b) \geq \frac12 \norm{a-b}_{\Xx}^2.
\end{equation}
The Bregman  proximal  gradient descent method  (BPGD) \cite{tseng2010approximation} is
\begin{equation}\label{BPG}
x_{k+1} = \argmin_{x} F(x_k) +\nabla F(x_k)^\top (x - x_k) +R(x) + \frac{\Lip}{2} D_{\eta}(x,x_k), 
\end{equation} 
with corresponds to taking constant stepsize $1/\Lip$.

\begin{rem}
The case of $R = 0$ corresponds to the mirror descent method, this dates back to  \cite{blair1985problem,alber1993metric} and has in more recent years been revitalised by \cite{beck2003mirror}. The proximal version given here is due to \cite{tseng2010approximation}. 
%The  analysis of this method in the fine grids/continuous measures is due to    \cite{chizat2021convergence}, up until the recent work of \cite{chizat2021convergence}, convergence results for mirror descent are only ``almost dimension independent".
\end{rem}
It is shown in \cite{tseng2010approximation} that this is a descent method with $\Phi(x_{k+1}) \leq\Phi(x_k)$ and for any $x\in \mathrm{dom}(R)$, 
\begin{equation}\label{eq:mirror_tseng}
\Phi(x_k) - \Phi(x) \leq \frac{1}{k} \Lip D_\eta(x,x_0).
\end{equation}

%%%%%%%%%%%%%%%%%%%%%%%%%%%%%%%%%%%%%%%%%%%%%%%%%%%%%%%%%%%%%%%%%%%%%%%%%%%%%%%%
%%%%%%%%%%%%%%%%%%%%%%%%%%%%%%%%%%%%%%%%%%%%%%%%%%%%%%%%%%%%%%%%%%%%%%%%%%%%%%%%
%%%%%%%%%%%%%%%%%%%%%%%%%%%%%%%%%%%%%%%%%%%%%%%%%%%%%%%%%%%%%%%%%%%%%%%%%%%%%%%%
\subsection{The Lasso ($\ell_1$) special case}
\label{sec:l1setting}

The BPGD algorithm~\eqref{BPG} is mainly interesting when the updated  (the so-called proximal operator associated to $R$) step can be computed in closed form. This is not the case for an arbitrary operator $L$, and we thus focus on the setting $L=\Id$.  For the sake of simplicity, we also consider the case where there is no group structure (Lasso), so that $R(x) = \norm{x}_1$. The most natural norm to perform the convergence analysis is $\norm{\cdot}_\Xx = \norm{\cdot}_1$, so that 
$\norm{\nabla F(x) - \nabla F(x')}_\infty \leq \Lip_1 \norm{x-x'}_1$. 
%\begin{rem}[A fine grids example]
%If $F(x) = \frac12 \norm{A x - y}^2$ and the matrix $A$ has normalized columns, then $\Lip_1=1$.
% To give a concrete example, one can consider for $k=(k_i)_{i=1}^d\in [p]^d \eqdef \ens{0,\ldots, p}^d$,
%\begin{equation}\label{eq:fouriersys0}
%A_k = \phi\pa{\theta_k}, \qwhereq \theta_k\eqdef \frac{1}{p} (k_i)_{i=1}^d,\; \phi(\theta) \eqdef \frac{1}{m^{d/2}} \pa{e^{2\pi \sqrt{-1} \dotp{\theta}{ \ell}}}_{\ell\in [m/2]^d},
%\end{equation}
%where $[m/2]\eqdef\ens{-\frac{m}{2},\ldots, \frac{m}{2}}$, and so, $A \in \RR^{\times n}$ with $n = p^d$ corresponds to the Fourier operator  discretized on a uniform grid on $[0,1]^d$.
%\end{rem}
%
 For this choice of $R$, \eqref{BPG} can be rewritten as
\begin{equation}\label{eq:BPG_l1}
\nabla \eta(x_{k+1}) =\Tt_{\lambda\tau}\pa{ \nabla \eta(x_k) - \frac{\tau}{n} \nabla F(x_k)}, 
\end{equation}
where $\Tt_\tau(z) = \max(\abs{z}-\tau,0)\odot \sign(z)$ is the soft thresholding operator.
Let us now single out  notable choices of entropy functions, in order to particularize the convergence bound~\eqref{eq:mirror_tseng}:
\begin{itemize}
	\item \textit{The quadratic entropy:} observe that $n\norm{x-x'}^2 \geq \norm{x-x'}_1^2$, so \eqref{eq:pinkser}  holds by choosing $\eta(x) = \frac{n}{2}\norm{x}^2$. For $\norm{x}_1\leq 1$ and $ \norm{x_0}_1\leq 1$, we have $D(x,x_0) \leq n \norm{x}^2 +n \norm{x_0}^2 \leq 2n$. The error bound  \eqref{eq:mirror_tseng}  is therefore $\Oo(\frac{n\Lip_1}{k})$.

%For $\eta(x) = \frac12 \norm{x}^2$. In this case, $\nabla \eta(x) = x$. Moreover, since $D_\eta(x,x') = \frac12 \norm{x-x'}^2$, it is natural to choose $\norm{\cdot}_\Xx = \norm{\cdot}_{\Xx^*}$ as the Euclidean norm. For \eqref{eq:mirror_lip_assumption}, since $\nabla F$ is $\Lip_1$-Lipschitz with respect to the $\ell_\infty/\ell_1$ norms,  wrt Euclidean norm we have
%$
%\norm{\nabla F(x) - \nabla F(x')}\leq  n\Lip_1 \norm{x-x'}.
%$
%So we should take $L = n\Lip_1$. Assuming that $\norm{x}_1, \norm{x_0}_1\leq 1$, we have $D(x,x_0) \leq \norm{x}^2 + \norm{x_0}^2 \leq 2$. The error bound  \eqref{eq:mirror_tseng}  is therefore $\Oo(\frac{n\Lip_1}{k})$.
% \begin{rem}
% This bound on $L$ is sharp, indeed of the Fourier matrix given in \eqref{eq:fouriersys0}, one has $\norm{\nabla^2 F} = \norm{A}^2 \geq n/m^d$.
% \end{rem}
 
 	\item \textit{The hyperbolic entropy:} introduced in\cite{ghai2020exponentiated}, it is defined for $c>0$ by
\begin{equation}\label{eq:hyp_ent_fn}
\eta_c(s) = s\cdot  \mathrm{arcsinh}(s/c) - \sqrt{s^2 + c^2} + c,
\end{equation}
so that $
 \eta_c'(s) = \mathrm{arcsinh}(s/c) \qandq  \eta_c''(s) = \frac{1}{\sqrt{s^2 +c^2}}.
  $
  It is shown in \cite{ghai2020exponentiated} that $\eta_c$ satisfies
  $
  D_\eta(x,x') \geq \frac{1+cn}{2} \norm{x-x'}_1^2
  $.
  In particular, \eqref{eq:pinkser} holds by choosing $c = \frac{1}{n}$ and rescaling $\eta$ by $\frac12$. On the other hand, for $\norm{x}_1,\norm{x'}_1\leq 1$,
$D_\eta(x,x')  = \Oo(\log(n)).
$
The error bound \eqref{eq:mirror_tseng}  is therefore $\Oo(\frac{\log(n) \Lip_1}{k})$.
\end{itemize}

\paragraph{Grid-free convergence rates}
The above results show that the error bound \eqref{eq:mirror_tseng} in general has a dependency on $n$, either through $\Lip$ or through $D_\eta(x,x_0)$.
This is thus unable to cope with very fine grids, and the analysis breaks in the ``continuous'' (often called off-the-grid) setting where discrete vectors with bounded $\ell^1$ are replaced by measures with bounded total variation~\cite{bredies2013inverse,candes2014towards}.
To address this issue, a more refined analysis  of BPGD is carried out in\cite{chizat2021convergence} and this lead to the first grid-free convergence rates for BPGD.  In particular,  it is shown that the objective for quadratic entropy converges at rate $\Oo(k^{-2/(d+2)})$, independent of grid size $n$ but dependent on the underlying  dimension $d$. 
 In contrast, BPGD with hyperbolic entropy satisfies $\Phi(x_k) - \min_x \Phi(x)  = \Oo(d\log(k)/k)$. % \todo{In the paper of Lenaic, it's written that the hyperbolic entropy rate is independent of $d$, but I don't know how he gets this. In Thm 4.1 of his paper, he needs to bound 
% $$
% \psi(\alpha)\lesssim \inf_{\epsilon>0} \epsilon^q + \alpha \epsilon^d \eta(\epsilon^{-d})
% $$
% If I plug in the hyperbolic entropy and $q=2$, this is
%  $$
% \psi(\alpha)\lesssim \inf_{\epsilon>0} \epsilon^2 + \alpha \mathrm{arcsinh}(\epsilon^{-d}/c) - \alpha  \sqrt{1+c^2 \epsilon^{2d}} + c\alpha \epsilon^d 
% $$
% If  I ignore the last 2 terms (they will be bounded by $\alpha$ if I assume $\epsilon<1$), and use the asymptotic that  $\mathrm{arcsinh}(\epsilon^{-d}/c) \sim \ln(2\epsilon^{-d}/c)$ for small $\epsilon$, then
% $$
% \inf_{\epsilon>0} \epsilon^2 + \alpha \ln(2\epsilon^{-d}/c) \sim \alpha d\ln(1/(\alpha d))
% $$
% which yields $\Oo(d\ln(k)/k)$ for $\alpha = 1/k$. Was I too crude in my estimations? Even before dropping the terms, it seems that $d$ appears in the bound and I don't why  it should cancel out.
% }\todo{I agree with your analysis, I think Lenaic simply does not care about the constant depending on $d$, only the rate?}
% !TEX root = ../SINUM-VarPro.tex
\newcommand{\xp}{x}

\subsection{The Hadamard parametrization: grid-free convergence analysis}
\label{sec:hadamard-finegrid}

In this  section, we show that  gradient descent with fixed timestep on the Hadamard parameterization  also leads to grid and dimension free convergence guarantees. The caveat is that due to the nonconvex nature of our problem,  our convergence results are only for the gradient norm and thus  weaker than the objective convergence results of \cite{chizat2021convergence}.
The Hadamard parameterization of \eqref{eq-Phi-def} in the group Lasso case (i.e. $L=\Id$ in~\eqref{eq:gen-tv}) is
$$
\min_{u,v} G(u,v)\qwhereq G(u,v) \eqdef \frac12 \norm{u}_2^2 + \frac12 \norm{v}_2^2  + F(u\odot v)
$$
where $F$ is differentiable with  $\Lip_F>0$  such that
\begin{equation}\label{eq:lipschitz}
\norm{\nabla F(\xp)-\nabla F(\xp')}_{\infty,2} \leq \Lip_F \norm{\xp-\xp'}_{1,2},
\end{equation}
where $\norm{z}_{\infty,2}\eqdef \max_{g\in\Gg} \norm{z_g}$.
For a stepsize $\tau>0$, the gradient descent iterations are
\begin{equation}\label{eq:gd_uv}
\begin{split}
u_{k+1} &= u_k - \tau g_{u_k}, \qwhereq g_{u_k} = u_k + v_k\odot \nabla F(u_k\odot v_k),\\
v_{k+1} &= v_k - \tau g_{v_k}, \qwhereq  g_{v_k} = v_k + \pa{(u_k)_g^\top \nabla F(u_k\odot v_k)_g}_{g\in\Gg}.
\end{split}
\end{equation}
\begin{rem}
In the Lasso setting where $u_k$ and $v_k$ have the same dimensions,  if $u_0 = v_0$, then $u_k=v_k$ for all $k$, while if $u_0 = -v_0$, then $u_k = -v_k$ for all $k$. One should therefore initialise with $\abs{u_0} \neq \abs{v_0}$. In practice, we find that random initialisation of $u_0$ and $v_0$ works well.
\end{rem}

%\paragraph{Remarks in initialisation}
%First observe that if $u_0 = v_0$, then, it is straightforward to see by induction that $u_k = v_k$ for all $k$ and
%\begin{align*}
%g_{v_k} = g_{u_k} = u_k + u_k f'(u_k^2)
%\end{align*}
%while if $u_0 = -v_0$, then  $u_k = -v_k$ for all $k$ and
%\begin{align*}
%&g_{u_k} =  u_k -  u_k f'(-u_k^2) \qandq  g_{v_k} = - u_k +  u_k f'(-u_k^2)\\
%&\implies -g_{u_k}^2 = -u_k^2 + 2 u_k^2 f'(-u_k^2) - u_k^2 f'(-u_k^2)^2
%\end{align*}

Since gradient descent is a descent method, one can assume that all iterates lie inside some ball, that is, all iterates satisfy
\begin{equation}\label{eq:ball_constr}
\frac12\pa{\norm{u}^2+\norm{v}^2}\leq G(u_0,v_0) \eqdef \frac12 B^2.
\end{equation}
Suppose that 
\begin{equation}\label{eq:grad_constr}
\sup_{\norm{x}_{1,2} \leq B^2/2} \norm{\nabla F(x)}_{\infty,2}\leq K.
\end{equation}
Note that this implies $\sup_{{\norm{u}^2+\norm{v}^2} \leq B^2} \norm{\nabla F(u\odot v)}_{\infty,2}\leq K$.
Under these assumptions, the following Proposition shows that  $\nabla G$ is Lipschitz with respect to the Euclidean norm, with a Lipschitz constant that depends only on $\Lip_F,B,K$.  This in turn ensure convergence rates for~\eqref{eq:gd_uv} which are dimension-free.
Note that using this Hadamard parameterization, one considers descent on $u$ and $v$ with respect to the standard Euclidean metric. This  convergence statement is thus a direct consequence of the standard descent lemma for gradient descent (Lemma \ref{lem:grad_bound_desc_lemma}).

\begin{prop}\label{prop:lip_g}
Assume that $F\in C^{1,1}$ with Lipschitz gradient satisfying \eqref{eq:lipschitz} and uniformly bounded gradient  \eqref{eq:grad_constr}. Then, given $u_1,v_1,u_2,v_2$ satisfying \eqref{eq:ball_constr}, we have the following Lipschitz bound on $\nabla G$,
$$
\norm{\nabla G(u_1,v_1)-\nabla G(u_2,v_2)} \leq \Lip_G \norm{(u_1,v_1)-(u_2,v_2)},
$$ 
where $\Lip_G \eqdef  2 (K+\Lip_F B^2)$. For stepsize $\tau = 1/\Lip_G$, we have
\begin{equation}
\min_{k\leq T} \norm{\nabla G(u_k,v_k)}^2 \leq \frac{2\Lip_G}{T} (G(u_0,v_0)-G(u_{T+1},v_{T+1}))
\end{equation}
\begin{equation}\label{upper_bound_grad_sum}
\qandq \sum_{j=k}^\infty \norm{\nabla G(u_j,v_j)}^2 \leq 2\Lip_G (G(u_j ,v_j) - \lim_{j\to\infty} G(u_j ,v_j)) .
\end{equation}

\end{prop}
\begin{proof}
Note that $\nabla_u G(u,v) =u- v\odot \nabla F(u\odot v)$ and $\nabla_v G(u,v) = v- (u_g^\top  \nabla F(u\odot v)_g)_g$, so that
\begin{align*}
&\norm{\nabla_u G(u_1,v_1) - \nabla_u G(u_2,v_2)}\\
&\leq \norm{u_1-u_2}+ \norm{v_1 \odot  \nabla F(u_1\odot v_1) - v_2 \odot  \nabla F(u_2\odot v_2)}\\
&\leq \norm{u_1-u_2}+   \norm{(v_1-v_2)\odot  \nabla F(u_1\odot v_1)} + \norm{v_2\odot  (\nabla F(u_1\odot v_1)  - \nabla F(u_2v_2))}\\
&\leq\norm{u_1-u_2}+ \norm{v_1-v_2} \norm{\nabla F(u_1\odot v_1)}_{\infty,2} + \norm{v_2} \Lip_F \norm{u_1\odot v_1 -u_2\odot v_2}_{1,2}\\
&\leq \norm{u_1-u_2}+K \norm{v_1-v_2} + \Lip_F \norm{v_2} \pa{\norm{v_1}\norm{u_1-u_2} + \norm{u_2}\norm{v_1-v_2}}\\
&\leq (1+\Lip_F B^2)\norm{u_1-u_2} + (K+\Lip_F B^2) \norm{v_1-v_2}.
\end{align*}
The term $\norm{\nabla_u G(u_1,v_1) - \nabla_u G(u_2,v_2)}$ can be bounded in a similar way and the result follows.
The final gradient bound is then a direct consequence of Lemma \ref{lem:grad_bound_desc_lemma}.
\end{proof}

The crucial point which makes this Hadamard parameterization attractive is that in the fine grids setting, $\Lip_F$ and $K$ typically have no dependence on the grid discretization or the underlying dimension. This implies that the Hadamard parametrization leads to grid-free and dimension-free convergences rate on the gradient. Consider the case of trivial groups $\Gg = \{\{j\}\}_{j=1}^n$, $F(\xp) = F_0(A\xp)$,  $\nabla F(\xp) = A^\top \nabla F_0(A\xp)$ and
$$
	\Lip_F \leq \norm{A}_{1\to 2}^2 \Lip_{F_0}
$$
where $\Lip_{F_0}$ is the Lipschitz constant of $F_0$ with respect to the Euclidean norm and $\norm{A}_{1\to 2} \leq 1$ if the columns of $A$ are normalised. For the Fourier example mentioned in \eqref{eq:fouriersys}, $F_0$ is the quadratic function $\norm{\cdot - y}_2^2$  and we can take $\Lip_F = 1$.

\subsection{The Hadamard flow: connection with mirror descent}
\label{sec:hypentropy-mirror}

In this section, we consider the case of $\ell_1$ regularization~\eqref{sec:l1setting} (trivial group structure). 
The goal of this section is to highlight the connection to mirror descent (Proposition  \ref{prop:flow}). Based on this connection, we show in Proposition \ref{prop:rate} that convergence of the objective is controlled by the convergence of $\nabla G$.
This analysis does not carry over the group Lasso case, because the evolution of the flow on $(u,v)$ cannot be mapped back to a differential equation on the initial variable $x$.
Specialized to the case of $L=\Id$, the Hadamard parametrized function  is 
$$
	G(u,v) = \min_{u,v\in\RR^n} F(u\cdot v)+ \frac{\lambda}{2}\norm{u}_2^2 + \frac{\lambda}{2} \norm{v}_2^2.
$$
Note that letting $\tau \to 0$, the continuous flow equations of \eqref{eq:gd_uv} are
\begin{align}
\dot u(t) &= - \lambda u(t) - v(t) \cdot  \nabla F( u\cdot v), \label{eq:u_cts_flow}
\\
\dot v(t) &= - \lambda v(t) - u(t) \cdot  \nabla F( u\cdot v) . \label{eq:v_cts_flow}
\end{align}
%\todo{I'm a little unsure about a few points
%\begin{itemize}
%\item  it seems $G$ is not differentiable since $\alpha$ is not unique. So, the above is just the subgradient.
%\item Even if $G$ is not differentiable, the flow equation on $\beta$ in \eqref{eq:beta_flow_tv} is still well defined if $F$ is injective, since then, $\beta$ is well-defined (and  hence, $D^\top \alpha$ is also well-defined).
%\item I think there is actually nothing specific about the quadratic loss here, and we could consider a general convex differentiable function $F$. The only reason why I wrote the quadratic is so I can write the linear system of equations \eqref{eq:beta_t}.
%\end{itemize}
%}

The following propositions show that the $L_2$ flow on $G$ corresponds to ``mirror descent'' with the generalized hyperbolic entropy function $\eta^D_{\gamma(t)}$.

\begin{prop}\label{prop:flow}
Let $x(t)\eqdef u(t) \cdot v(t)$ where $u,v$ satisfy \eqref{eq:u_cts_flow} and \eqref{eq:v_cts_flow}. For $\gamma>0$, let $\eta_\gamma$ denote the hyperbolic entropy function  defined in \eqref{eq:hyp_ent_fn}. The following holds
 \begin{equation}\label{eq:beta_flow_tv}
\frac{\mathrm{d}}{\mathrm{d}t} \nabla \eta_{\gamma(t)}(x(t) ) = -2 \nabla F(x(t))
\end{equation}
where $\gamma(t) = \frac12 \abs{u(0)^2 - v(0)^2} \exp(-2\lambda t)$.
\end{prop}

\begin{proof}
From the flow equations \eqref{eq:u_cts_flow} and \eqref{eq:v_cts_flow},
\begin{equation}\label{eq:z_eqn}
\dot x(t)= \dot u(t) \cdot v(t) + \dot v(t)\cdot u(t)
= -( u(t)^2 +v(t)^2)  \cdot \nabla F(x(t)) -2  x(t),
\end{equation}
Note that $(u(t)^2 + v(t)^2)^2 - 4 u(t)^2\cdot  v(t)^2 = (u(t)^2 - v(t)^2)^2$ and
\begin{align*}
&\frac{\mathrm{d}}{\mathrm{d}t}[(u(t)^2 - v(t)^2)^2] = -  4\lambda(u(t)^2   - v(t)^2 )^2
\end{align*}
and so, $(u(t)^2 - v(t)^2)^2 = \exp(-4\lambda t)  (u(0)^2 - v(0)^2)^2$ which implies that
$$
(u(t)^2 + v(t)^2)^2  = 4 x(t)^2 +  c^2\exp(-4\lambda t)  \qwhereq c\eqdef  \abs{u(0)^2 - v(0)^2}.
$$
By denoting $\gamma(t) \eqdef \frac12 c \exp(-2\lambda t)$, the equation \eqref{eq:z_eqn} can be rewritten as
 \begin{equation}\label{eq:zflow}
\frac{\pa{\dot x(t) + 2\lambda x(t)}}{\sqrt{ x(t)^2 +  \gamma(t)^2}}  =-2\nabla F(x(t)).
\end{equation}  
Finally, 
\begin{align*}
\frac{\mathrm{d}}{\mathrm{d}t} \nabla \eta_{\gamma(t)}(x(t) ) &=
\frac{\mathrm{d}}{\mathrm{d}t} \nabla \eta_c(x(t) \exp(2\lambda t)) \\
&= \left(\eta_c''\pa{x(t) \exp(2\lambda t)} \exp(2\lambda t) \pa{\dot x(t) + 2 \lambda  x(t) } \right)\\
&=  \frac{\pa{\dot x(t) + 2 \lambda  x(t) } }{ \sqrt{x(t)^2 + \gamma(t)^2}} =  -2 \nabla F(x(t)).
\end{align*}  
\end{proof}

\begin{rem}[{Algorithmic regularisation properties}] \label{alg_reg}
Following \cite{azulay2021implicit}, we make some informal comments on the significance of Proposition \ref{prop:flow} when
 $\lambda = 0$, as inducing an ``implicit bias'' selecting a particular solution to the linear system $Ax=y$. Here,  $\gamma(t) = c\eqdef \frac12 \abs{u(0)^2 - v(0)^2}$ is constant for all $t$ and 
$$
\nabla \eta_{c} (x(t)) - \nabla \eta_{c}(x(0)) =-2  A^\top r(t) \qwhereq r(t) \eqdef \int_0^t A  x(s) \mathrm{d}s - y, 
$$
which, assuming that $x(t)$ converges to $x_*$ such that $Ax_*  = y$, is the optimality condition for
$$
\min_{x} D_{\eta_c}(x ,x(0)) \quad \text{s.t.} \quad Ax = y, 
$$
where $D_{\eta_c}(x,x')$ is the Bregman divergence associated to  $\eta_c$. So, even without explicit regularisation,  the flow $x(t)$ defined by \eqref{eq:u_cts_flow}, \eqref{eq:v_cts_flow} is regularised by $\eta_c$.
\end{rem}

\begin{rem}[Difference of squares parameterization]
Another parameterization for the lasso is the \textit{squared-parametrization}\cite{chizat2021convergence}: let $x = r\cdot r - s\cdot s$ and perform $L_2$ gradient flow on
\begin{equation}\label{eq:squared}
 F(r\cdot r - s\cdot s) + \lambda \norm{r}^2 + \lambda \norm{s}^2.
\end{equation}
By writing $r =\frac12( u+v)$ and  $s =\frac12( u-v)$, we have
$
r\cdot r - s\cdot s = u\cdot v  
$ and one can observe that this is equivalent to the Hadamard parameterization. Note however that this equivalence is only in the case of the Lasso and the Hadamard parametrization can be used to handle more complex regularisers such as the group  $\ell_1$ norm.

% and $ \norm{r}^2 + \norm{s}^2= \frac12 \norm{u}^2+\frac12\norm{v}^2 $.
%
%\todo{Gab: wouldn't it be clearer to write $x(t) = r(t) \cdot r(t) - s(t)\cdot s(t)$ and state that (to check)  the flow on $(r,s)$ and $(u,v)$ are related by $(r(t),s(t))=(u(t)-v(t),u(t)+v(t))/2$?}
%Note that \eqref{eq:u_cts_flow} and \eqref{eq:v_cts_flow} is equivalent to the squared-parametrization \cite{chizat2021convergence} in the trivial group setting $\Gg = \ens{\{j\}}_{j=1}^p$: let $x = u\cdot u - v\cdot v$ and perform $L_2$ gradient flow on
%\begin{equation}\label{eq:squared}
% F(u\cdot u - v\cdot v) + \lambda \norm{u}^2 + \lambda \norm{v}^2.
%\end{equation}
%Observe that  \eqref{eq:squared} can be written as
%$$
%F((u+v) \cdot (u-v)) + \frac12 \lambda \norm{(u+v)}^2 +\frac12 \lambda \norm{u-v}^2.
%$$
%The $L_2$ gradient flow on \eqref{eq:squared} is
%\begin{align*}
%\dot u &= -2 u\odot \nabla F(u\cdot u - v\cdot v) - 2\lambda u\\
%\dot v &= 2v\odot \nabla F(u\cdot v - v\cdot v) - 2 \lambda v.
%\end{align*}
%By adding and subtracting these two equations, the flow on $u-v$ and $u+v$ obtained from the squared parameterization is, up to  a factor of 2, equivalent to the $(u,v)$ flow of the Hadamard parameterization. However, one difference is that the Hadamard parametrization can be used to handle more complex regularisers such as the group  $\ell_1$ norm.
\end{rem}

We saw in Proposition \ref{prop:flow} that in the continuous time limit, gradient descent on the Hadamard parametrization can be interpreted as mirror descent with a varying entropy function. For fixed timestep, one can write for $\xp_k\eqdef u_k \cdot v_k$
$$
	\xp_{k+1} = \xp_k - \tau H_k^{-1} \nabla F(u_k \cdot v_k)  + \tau^2 g_{u_k} \cdot g_{v_k} \qwhereq H_k = \diag( 1/(u_k^2+v_k^2) ).
$$
Ignoring the  $\tau^2$ term, one can view this as variable metric descent on $x_k$. By making use of this link, we can relate the convergence of the objective on $x_k$ to the convergence of the gradient  of the overparametrized function $\nabla G$ as follows.

\begin{prop}\label{prop:rate}
Suppose that $\Gg = \ens{\{j\}}_{j=1}^p$.
Assume that $F$ satisfies \eqref{eq:lipschitz} and \eqref{eq:grad_constr}.
Let $\tau = 1/\pa{\kappa \Lip_G}$ where $\Lip_G$ upper bounds the Lipschitz constant of $\nabla G$ and $\kappa = \max(1,(1+K^2)/\Lip_G) $ (note  that $\Lip_G =  \Oo(K+\Lip_F B^2) $ by Proposition \ref{prop:lip_g}). Then,
$$
\Phi(\xp_k) - \lim_{k\to\infty } \Phi(\xp_k) \leq C \sum_{j=k}^{\infty} \norm{\nabla G(u_j, v_j)}^2 + C\rho^k.
$$
with $\rho \eqdef 1-\frac{1}{\kappa M_G}$.
In particular, if $\norm{\nabla G(u_k, v_k)}=\Oo(1/k)$, then the objective converges at rate $\Oo(1/k)$.
\end{prop}

The proof of this proposition can be found in Section~\ref{sec:proof-proprate}. 
This proposition along with \eqref{upper_bound_grad_sum} shows that the  objective convergence rate is equivalent to the convergence of  the tail sum of the gradients, but, at present, we do not have sufficiently strong convergence results on this gradient sum to obtain convergence rates. 

Figure \ref{fig:ISTA} provides some empirical finding suggesting that indeed, $\norm{\nabla G(u_k, v_k)}$ is of the order $\Oo(1/k)$.
The  problem considered is the Fourier system \eqref{eq:fouriersys} with $n=300$, $m=2$, $\lambda = \lambda_{\max}/10$ and the underlying signal to recover being 1-sparse (a single Dirac mass). 
For reference, the dashed lines show the $1/k$ and $1/k^{2/3}$ convergence lines. Moreover, ``Hadamard Grad'' on the left figure shows how $\norm{\nabla G(u_k,v_k)}$ converges -- it matches the $1/k$ line and converges at the same rate as the objective function. Note that both Hyperbolic entropy and the Hadamard flow exhibit $\Oo(1/k)$ convergence, however, one practical advantage of Hadamard and Noncvx-pro is that since these methods are based on Euclidean geometry, one can apply standard tools for acceleration, such as Barzilai-Borwein (BB) stepsize~\cite{barzilai1988two}. Finally, observe that ISTA converges at rate $\Oo(1/k^{2/3})$ as proved in \cite{chizat2021convergence}, and as can be seen on the right figure, the use of BB stepsize also accelerates ISTA (although there is no theoretical proof of this).

\begin{figure}
\begin{center}
\begin{tabular}{cc}
\includegraphics[width=0.45\linewidth]{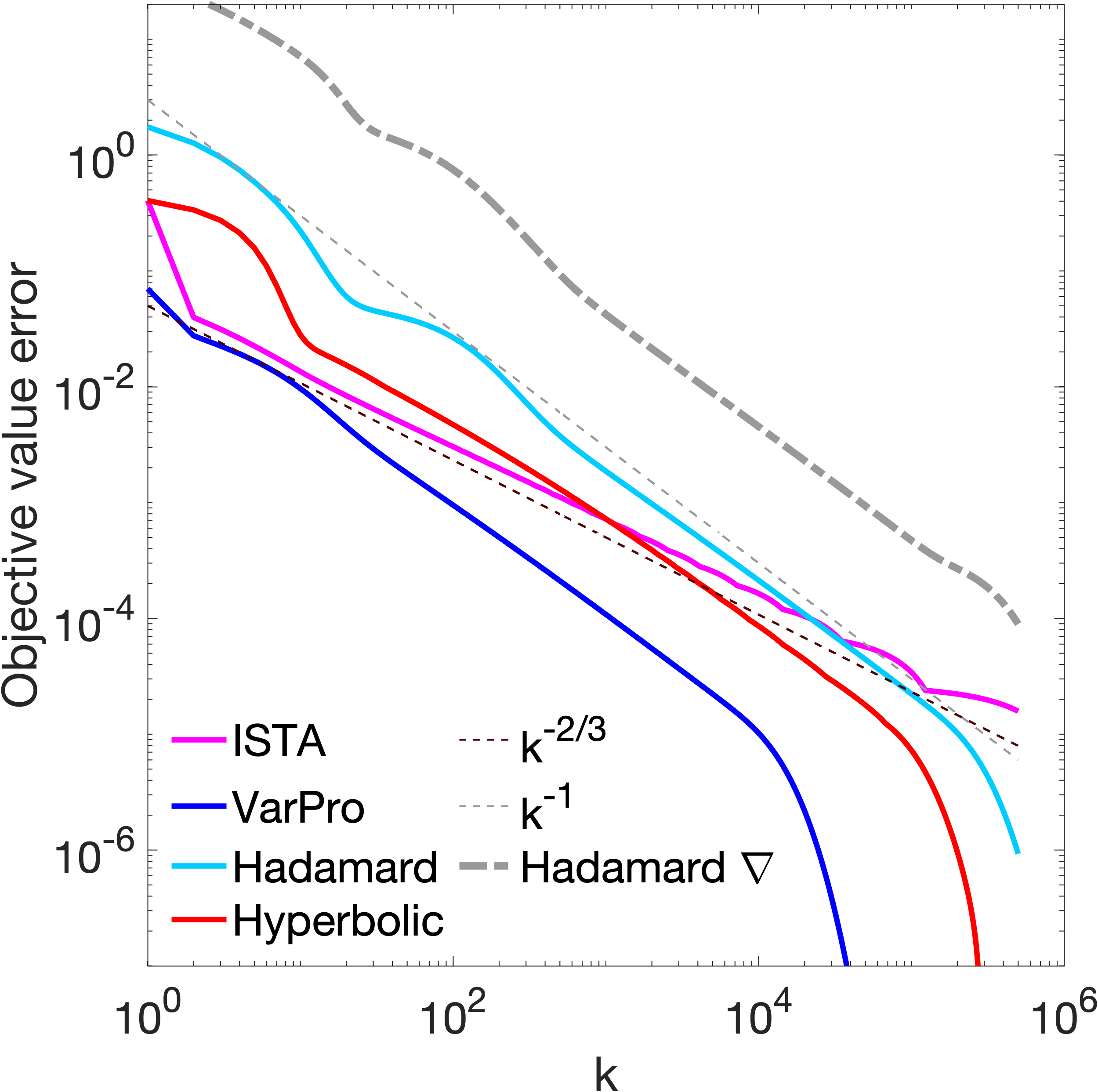}&
\includegraphics[width=0.45\linewidth]{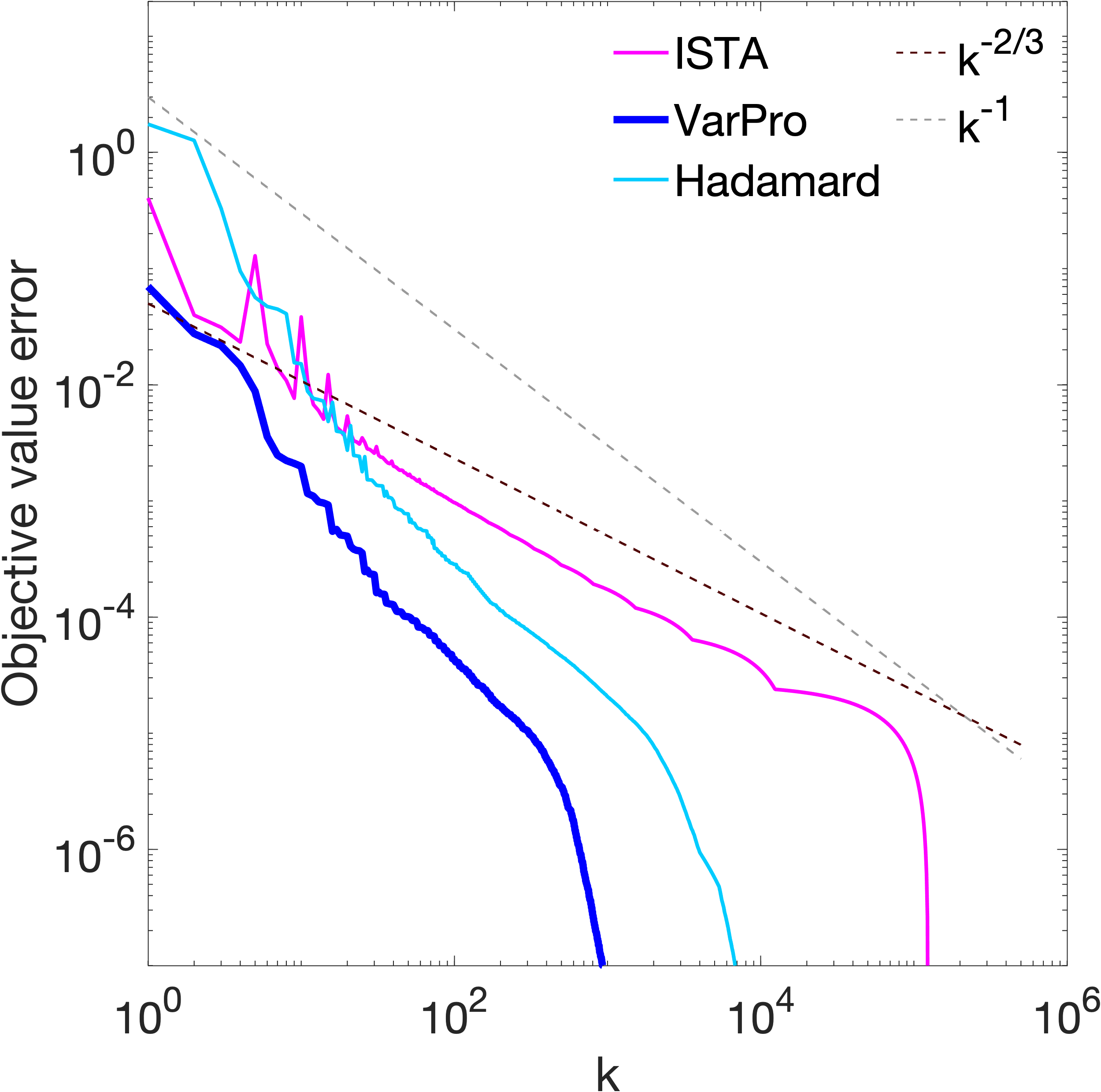}\\
Fixed step sizes & BB step sizes
\end{tabular}
\end{center}
\caption{Comparison of: ISTA, Gradient descent for Noncvx-pro, gradient descent on the Hadamard parameterization,  Bregman projected gradient with the Hyperbolic entropy. 
The left figure shows the objective convergence against iterations $k$ with \textit{fixed} stepsize, and the right figure corresponds to  Barzilai-Borwein (BB) stepsize.
 \label{fig:ISTA}}
\end{figure}

%%%%%%%%%%%%%%%%%%%%%%%%%%%%%%%%%%%%%%%%%%%%%%%%%%%%%%%%%%%%%%%%%%%%%%%%%%%
\subsection{Proof of Proposition \ref{prop:rate}}
\label{sec:proof-proprate}

Proposition \ref{prop:rate} is a direct consequence of the following stronger result.

\begin{prop}\label{prop:rate2} Suppose that we have a trivial group structure $\Gg = \ens{\{j\}}_{j=1}^p$ and
assume that $F$ satisfies \eqref{eq:lipschitz} and \eqref{eq:grad_constr}.
Let $\tau = 1/\Lip_G$ where $\Lip_G$ upper bounds the Lipschitz constant of $\nabla G$, note that $\Lip_G =  \Oo(K+\Lip_F B^2) $ by Proposition \ref{prop:lip_g}.
Define $H_k = \diag(1/(u_k^2 + v_k^2))$  and $\norm{x}_{H_k}\eqdef \sqrt{\dotp{H_k x}{x}}$.
 Define $\Phi_k \eqdef  F(\xp_k) + 2\norm{\xp_k}_{H_k}^2$ and $\Phi(\xp) \eqdef F(\xp) +\norm{\xp}_1$. To simplify the expression below, assume that $K,\Lip_G,B\geq 1$. For any $\bar \xp$, we have
 \begin{align*}
    (\Phi_k - \Phi(\bar \xp) ) = \Oo\pa{ C \norm{\nabla G(u_k,v_k)}^2   +\sqrt{C} \norm{\xp_k - \bar \xp}_{H_k} \norm{\nabla G(u_k,v_k)}     },
 \end{align*}
  where $C=  K \Lip_G (G(u_0,v_0) - G(u_*,v_*)) $.

\end{prop}

Note that $\Phi_k$ approximates $\Phi(\xp_k)$, indeed,  $\Phi_k=2\norm{\xp_k}_{H_k} = \sum_{i}2 (u_k\odot  v_k)_i^2/(u_k^2+ v_k^2)_i$. Since  $a,b>0$ and $a\geq b$ implies that
\begin{align*}
& ab -  \frac{2a^2 b^2 }{a^2+b^2}   = \frac{ab}{a^2+b^2}(a-b)^2 \leq a^2 - b^2,
\end{align*}
we have $\Phi(\xp_k) - \Phi_k \leq \abs{\norm{u_k}^2 - \norm{v_k}^2} = \Oo(\rho^k). $ Finally, by plugging in $\bar \xp = \xp_{k+1}$ to the above proposition  and summing over $k, k+1,k+2,\ldots$ yields Proposition \ref{prop:rate}. The rest of this section is devoted to proving Proposition \ref{prop:rate2}.

We begin with two lemmas, the first Lemma will be used to show that $\Phi_k-\Phi(\xp_k)$ defined in Proposition \ref{prop:rate2} converges to  0 linearly, while the second lemma  provides several useful bounds in terms of the gradient of $G$. 
\begin{lem}\label{lem:shrinkage_ukvk}
Assume that  $F$ satisfies \eqref{eq:lipschitz}  and \eqref{eq:grad_constr} holds for all $k$.
Let $\kappa>0$ be such that $\kappa \geq \frac{1}{ M_G}(1+K^2 )$ and  let $\tau = \frac{1}{\kappa \Lip_G}$. Then,
$$
\abs{\norm{u_k}^2 - \norm{v_k}^2} = \Oo(\rho^k),
 $$ 
 where $\rho \eqdef 1-\frac{1}{\kappa M_G}$.
\end{lem}

\begin{proof}
Notice that 
\begin{align*}
\norm{u_{k+1}}^2  - \norm{v_{k+1}}^2 &= (1-2\tau) (\norm{u_k}^2 -\norm{v_k}^2) + \tau^2 (\norm{g_{u_k}}^2 - \norm{g_{v_k}}^2  )\\
& \leq  \abs{\norm{u_k}^2 -\norm{ v_k}^2 }(1-2\tau+ \tau^2(1+\norm{\nabla F'(\xp_k)}_{\infty}^2  ))
\end{align*}
By choosing $\tau = 1/(\kappa M_G)$,
$$
(1-2\tau+ \tau^2(1+\norm{\nabla F'(\xp_k)}_{\infty}^2 ) ) \leq 1-\frac{1}{\kappa M_G} \eqdef \rho
$$
%if 
%$$
%\frac{1}{\kappa M_G} \pa{  2- \frac{1}{\kappa M_G}(1+K^2 )}\geq \rho = \frac{1}{\kappa M_G}
%$$
%if 
%$$
%\frac{1}{ M_G}(1+K^2 )\leq \kappa
%$$
and so, $\abs{\norm{u_{k+1}}^2  -\norm{ v_{k+1}}^2} \leq \rho \abs{\norm{u_{k}}^2  -\norm{ v_{k}}^2} = \Oo(\rho^{k+1})$. 
\end{proof}

\begin{lem}\label{lem:grad_bounds}
Let  $H_k = \diag(1/(u_k^2+v_k^2))$. We have the following bounds
\begin{itemize}
\item $\norm{u_k \cdot  g_{v_k}}_{H_k}^2 \leq \norm{g_{v_k}}^2$ and  $\norm{v_k\cdot g_{u_k}}_{H_k}^2 \leq \norm{g_{u_k}}^2$
\item $
\norm{ g_{u_k} \cdot g_{v_k} }_{H_k}^2    \leq   C_0 \norm{g_{v_k}}^2
$, where $C_0\eqdef 2\max(1, K^2)$.
\item $\norm{ \xp_k -\xp_{k+1} }_{H_k}^2 \leq \tau^2  C_1 \norm{\nabla G(u_k,v_k)}^2$ where $C_1 \eqdef  (2+\tau  +2\tau(2+\tau) \max(1,K)) $.
\end{itemize}
%A consequence of the last inequality is that
%$$
%\norm{\xp_k -\xp_{k+1}}_{H_k}^2 \leq \tau^2 C_1 \norm{\nabla G(u_k, v_k)}^2. 
%$$
\end{lem}

\begin{proof}
Clearly, $\norm{u_k \cdot g_{v_k}}_{H_k}^2 \leq \norm{g_{v_k}}_2^2$ and $$
\norm{ g_{u_k} \cdot g_{v_k} }_{H_k}^2  = \sum_i \pa{\frac{g_{u_k}^2 }{u_k^2+ v_k^2}}_i \pa{g_{v_k}}_i^2   \leq 2\max(1,K^2) \norm{g_{v_k}}^2
$$
\begin{align*}
\text{since} \quad
g_{u_k}^2 &= u_k^2 + v_k^2\cdot \nabla F(\xp_k)^2 + 2\xp_k \cdot \nabla F(\xp_k) \leq 2(  u_k^2 + v_k^2 \cdot \nabla F(\xp_k)^2 ) \\
&\leq 2 \max(1, \nabla F (\xp_k)^2) \cdot (u_k^2+v_k^2).
\end{align*}
Finally, $\frac{1}{\tau}(\xp_{k+1} - \xp_k )=  u_k\cdot  g_{v_k} +  v_k \cdot  g_{u_k}- \tau g_{u_k} \cdot   g_{v_k}$, so that 
\begin{align*}
&\norm{ u_k \cdot   g_{v_k} +  v_k \cdot   g_{u_k}- \tau g_{u_k} \cdot   g_{v_k} }_{H_k}^2 = \norm{u_k \cdot   g_{u_k}}_{H_k}^2 + \norm{v_k\cdot   g_{v_k}}_{H_k}^2  
\\
&\qquad + 2\dotp{u_k\cdot  v_k}{g_{u_k}\cdot  g_{v_k}}_{H_k}^2 + \tau^2\norm{ g_{u_k} \cdot  g_{v_k} }_{H_k}^2 -2\tau \dotp{u_k \cdot  g_{v_k} +  v_k \cdot  g_{u_k}}{ g_{u_k}\cdot  g_{v_k}}_{H_k}^2\\
&\leq (2+\tau)\norm{u_k \cdot  g_{u_k}}_{H_k}^2 + (2+\tau)\norm{v_k \cdot  g_{v_k}}_{H_k}^2+ ( \tau^2+2\tau)\norm{ g_{u_k}\cdot   g_{v_k} }_{H_k}^2 ,
\end{align*}
and the result follows by the preceding bounds.
\end{proof}

\begin{proof}[Proof of Proposition \ref{prop:rate2}]
By multiplying together the two equations in \eqref{eq:gd_uv}, we first interpret  \eqref{eq:gd_uv} as  variable metric descent on $\xp_k = u_k \cdot  v_k$:
$$
\xp_{k+1} = \xp_k - \tau u_k \cdot  g_{v_k} - \tau v_k \cdot  g_{u_k} + \tau^2 g_{u_k} \cdot  g_{v_k}.
$$
Then,
\begin{align*}
\norm{\xp_{k+1} - \bar \xp}_{H_k}^2 &=  \norm{\xp_k - \bar \xp - \tau u_k \cdot  g_{v_k} - \tau v_k\cdot   g_{u_k} + \tau^2 g_{u_k} \cdot  g_{v_k} }_{H_k}^2\\
&= \norm{\xp_k - \bar \xp }_{H_k}^2 +T_1+T_2+T_3
\end{align*}
\begin{align*}
\qwhereq T_1 &\eqdef \norm{\tau u_k\cdot   g_{v_k} + \tau v_k \cdot  g_{u_k}- \tau^2 g_{u_k}\cdot   g_{v_k} }_{H_k}^2 \\
T_2&\eqdef  2\tau^2 \dotp{  g_{u_k} \cdot  g_{v_k} }{H_k(\xp_k - \bar \xp)} \\
T_3&\eqdef - 2\tau \dotp{  u_k\cdot   g_{v_k} +  v_k \cdot  g_{u_k} }{H_k(\xp_k - \bar \xp)}.
\end{align*}
The theorem is simply a consequence of   bounding $T_1,T_2,T_3$ using Lemma \ref{lem:grad_bounds}  and summing the inequality over $k=0,\ldots, T$.  Let $C_0, C_1$ be as in Lemma \ref{lem:grad_bounds}.
Indeed, by Lemma \ref{lem:grad_bounds}, $T_1 \leq \tau^2 C_1 \norm{\nabla G(u_k,v_k)}^2$.
To bound $T_2$, observe that
\begin{align*}
\dotp{g_{u_k} \cdot  g_{v_k}}{ H_k(\xp_k - \bar \xp} &\leq \norm{\xp_k - \bar \xp}_{H_k} \norm{g_{u_k}\cdot g_{v_k}}_{H_k}
=\norm{\xp_k - \bar \xp}_{H_k} \sqrt{C_0 } \norm{g_{v_k}}
\end{align*}
So,
$$
T_2 \leq 2\tau^2 \sqrt{C_0} \norm{\xp_k - \bar \xp}_{H_k} \norm{\nabla G(u_k,v_k)}.
$$
%Or, $T_2/(2\tau^2) \lesssim \norm{H_k(\xp_k - \bar \xp)}_\infty \norm{\nabla G(u_k,v_k)}^2$.
Consider the final term $T_3$: observe that 
$$
	u_k\cdot   g_{v_k} + v_k \cdot  g_{u_k} =  2\xp_k + (u_k^2 + v_k^2) \cdot  \nabla F(\xp_k)
$$
and $\norm{\abs{\xp_k}/(u_k^2 + v_k^2)}_\infty \leq 1/2$.
It follows by convexity of $F$ that
\begin{align*}
&\dotp{  u_k \cdot  g_{v_k} +  v_k \cdot  g_{u_k} }{H_k(\xp_k - \bar \xp)} \\
&= \dotp{ \frac{2\xp_k}{u_k^2+v_k^2}}{\xp_k - \bar \xp} + \dotp{\nabla F(\xp_k)}{\xp_k - \bar \xp}\\
&\geq (2\norm{\xp_k}_{H_k}^2 - \norm{\bar \xp}_1) +  ( F(\xp_k) - F(\bar \xp)) = \Phi_k - \Phi_*.
\end{align*}
So, 
$$
T_3 \leq -2\tau (\Phi_k - \Phi_*).
$$
It follows that
 \begin{align*}
 2\tau(\Phi_k - \Phi_*) \leq& \pa{\norm{\xp_k - \bar \xp}_{H_k}^2 -  \norm{\xp_{k+1} - \bar \xp}_{H_k}^2} \\
 &+ \tau C_1 \norm{\nabla G(u_k,v_k)}^2 + 2\tau^2 \sqrt{C_0} \norm{\xp_k - \bar \xp}_{H_k} \norm{\nabla G(u_k,v_k)}.
 \end{align*}
 Finally,
 \begin{align*}
 \norm{\xp_k - \bar \xp}_{H_k}^2 &-  \norm{\xp_{k+1} - \bar \xp}_{H_k}^2
 = \dotp{\xp_k -\xp_{k+1}}{\xp_k -\bar \xp}_{H_k} +\dotp{\xp_{k}-\xp_{k+1} }{\xp_{k+1} -\bar \xp}_{H_k}\\
 &\leq  \tau \sqrt{C_1}( \norm{\xp_k -\bar \xp}_{H_k} +\norm{\xp_{k+1} -\bar \xp}_{H_k} )\norm{\nabla G(u_k,v_k)}, 
 \end{align*}
and
\begin{equation*}
 	\norm{\xp_{k+1}-\bar \xp}_{H_k}\leq  \norm{\xp_{k+1}-\xp_k}_{H_k} + \norm{\xp_k - \bar \xp}_{H_k} \leq \tau \sqrt{C_1}\norm{\nabla G(u_k,v_k)}+ \norm{\xp_k - \bar \xp}_{H_k}.
\end{equation*}
\end{proof}

% !TEX root = ../SINUM-VarPro.tex

\section{Nonsmooth Robust losses}
\label{sec:nonsmooth-robust}

In this section, we describe some generalizations of our method to cope with non-smooth robust losses in Section \ref{sec:nonsmoothloss} and non-convex regularization functionals in Section \ref{sec:nonconvex}.  
These generalizations leverage so-called quadratic variational forms, recalled in Section~\ref{sec:quadvar}, which generalizes the overparameterization formula~\eqref{eq:overparam} beyond the $\ell^1-\ell^2$ norm.

\subsection{Quadratic Variational Forms}\label{sec:quadvar}

It is well known that quadratic variational forms exist for many nonsmooth regularisers, including nuclear norm, $\ell_q$ and also other nonconvex regularisers \cite{geman1992constrained,black1996unification}. In general, for a function
 $R:\RR^n\to \RR$ (see \cite{poon2021smooth} for a proof), one has the equivalence between: 
\begin{itemize}
 \item[i)] ) $R(\xp) = \phi(\xp\odot \xp)$ where $\phi$ is proper, concave and upper semi-continuous, with domain $\RR_+^d$. 
\item[ii)] There exists a convex function $\psi$ for which  $R(\xp) = \inf_{z\in\RR^n_+} \frac12 \sum_{i=1}^n z_i \xp_i^2 + \psi(z)$. 
\end{itemize}
Furthermore,  %$h(\eta) = 2 (-\phi)^*(-\frac{1}{2\eta})$,
 $\psi(z) = (-\phi)^*(-z/2)$ is defined via the convex conjugate $(-\phi)^*$ of  $-\phi$.
%  leading to~\eqref{eq:etatrick0} using the change of variable $\eta \leftarrow 1/z$ and $h(\eta) = 2\psi(1/\eta)$.

One particularly interesting class of functions which fit into the quadratic variational framework are (group) $\ell_q$ semi-norms for $q\in (0,2)$.
\begin{lem}\label{lem:lq}
Let $\beta>0$ and $q=2\beta/(1+\beta)$. Then,
$$
\frac{1}{q}\sum_{g\in\Gg} \norm{x_g}^q= \min_{\eta\in\RR^{\abs{\Gg}}_+} \frac12 \sum_{g\in\Gg} \frac{\norm{x_g}^2}{\eta_g} +\frac{1}{2\beta} \sum_{g\in\Gg} \eta_g^\beta
= \min_{x = u\odot v} \frac12 \norm{u}^2 + \frac{1}{2\beta} \sum_i \abs{v_i}^{2\beta}.
$$
%and for $q = \beta/(1+\beta)$,
%\begin{align*}
%\frac{1}{q} \sum_i \abs{x_i}^q = \min_{\eta\in\RR_+^d}   \sum_i \frac{\abs{x_i}}{\eta_i} +\frac{1}{\al}\sum_i \eta_i^\alpha
%\end{align*}
\end{lem}

In the remaining part of this section, we discuss two extensions of our VarPro approach: the first is where both the loss function and regulariser have quadratic variational forms, and the second is the use of non-convex functionals.

\subsection{Nonsmooth loss functions}\label{sec:nonsmoothloss}

Consider for $y\in\RR^m$, $L\in\RR^{p\times n}$ and $A\in \RR^{m\times n}$,
$$
\min_{\xp\in\RR^n} \Phi(\xp) = R_1(L\xp)+ \frac{1}{\lambda} R_2(A\xp-y)
$$
where the $R_i$ functionals (for $i=1,2$) both have quadratic variational forms
\begin{equation}
R_i(z) =\min_{\eta\in\RR^{n_i}_+} \frac12 \sum_{g\in \Gg_i}\frac{\norm{z_g}^2}{\eta_g} + \sum_g h_i(\eta_g), 
\end{equation}
where we assume that $h_1$ and $h_2$ are both differentiable functions, we have the partitions $\cup_{g\in \Gg_1} g = \ens{1,\ldots,p}$,  $\cup_{g\in \Gg_2} g = \ens{1,\ldots,m}$ and $d_1 = \abs{\Gg_1}$, $d_2 = \abs{\Gg_2}$.

\begin{prop}\label{prop:robust_loss}
We have
$$\min_\xp\Phi(\xp) = \min_{v\in\RR^{d_1},w\in\RR^{d_2}} f(v,w)\qwhereq f(v,w) \eqdef  h_1(v)+\frac{1}{\lambda} h_2(w) + \phi(v,w), $$
with 
$$
\phi(v,w) = \max_{\al\in\RR^p,\xi\in\RR^m}  \enscond{-\frac12 \norm{v\odot \alpha}^2 - \frac{\lambda}{2} \norm{w\odot \xi}^2 +\dotp{\xi}{y} }{L^\top \alpha+A^\top \xi = 0}.
$$
The optimal solutions satisfy$$
Lx = -v^2 \alpha \qandq Ax= y- \lambda w^2 \xi.
$$
Moreover, the maximiser $\al,\xi$ to the inner problem $\phi$ satisfy for some $x\in\RR^n$
\begin{equation}\label{eq:inner_robust}
\begin{pmatrix}
\diag(\bar v^2) & 0 &L\\
0 &\lambda \diag(\bar w^2) & A\\
L^\top &A^\top & 0
\end{pmatrix} \begin{pmatrix}
\alpha\\\xi \\\xp
\end{pmatrix} = \begin{pmatrix}
0\\y
\\0
\end{pmatrix},
\end{equation}
where as before, $\bar v$ and $\bar w$ are the extensions of $v$ and $w$ so that $\bar v\cdot \alpha = v\odot \alpha$ and $\bar w\cdot \xi - w\odot \xi$.
\end{prop}

\begin{proof}
We can write $\min_{\xp\in\RR^n} \Phi(\xp)$ as
\begin{align*}
\min_{\xp,u,v,z,w} \enscond{\frac12 \norm{u}^2 + h_1(v)+\frac{1}{2\lambda} \norm{z}^2 + \frac{1}{\lambda} h_2(w) }{u\odot v = L\xp, z\odot w = A\xp-y}
\end{align*}
where the minimisation is over the variables $\xp\in\RR^m$, $v\in\RR^{d_1}$, $u\in\RR^p$, $w\in\RR^{d_2}$ and $z\in\RR^m$.
This is a convex problem over the variables $u,z,\xp$, so
by convex duality,
\begin{align*}
&\min_{v,w}\min_{\xp,u,v,z,w} \frac12 \norm{u}^2 + h_1(v)+\frac{\norm{z}^2}{2\lambda}  +\frac{h_2(w)}{\lambda}  + \dotp{\alpha}{u\odot v - L\xp}  + \dotp{\xi}{z\odot w - A\xp+y}\\
&=
\min_{v,w}\max_{\al,\xi}  \enscond{-\frac{\norm{v\odot \alpha}^2}{2} + h_1(v) -\frac{\lambda}{2} \norm{w\odot \xi}^2 + \frac{h_2(w)}{\lambda} +\dotp{\xi}{y} }{L^\top \alpha+A^\top \xi = 0}
\end{align*}
where the dual variable are $\al\in\RR^p,\xi\in\RR^m$ and
where the optimal solutions satisfy
$u+v\odot \alpha = 0$, $z+ \lambda w\odot \xi=0$.
\end{proof}

\begin{rem}
Note that if $A=\Id$, then we can write \eqref{eq:inner_robust} as
$$
x = y - \lambda w^2\odot \xi \qandq \xi = -L^\top \alpha
$$
$$
\qandq \pa{\diag(\bar v^2)   + \lambda L\diag(\bar w^2)  L^\top }\alpha  = - Ly 
$$
\end{rem}

We now exemplify this general formulation on three illustrative scenario: TV-L1, square root Lasso and matrix recovery problems. 
 
%\todo{Gab: maybe would be good to discuss informally differentiability issues. Shall we give the generic formula to compute the gradient or is it too complicated ?}

\paragraph{Differentiability of $f$}
Formally, the gradient of $f$ is given as
\begin{equation}\label{eq:grad_nonsmooth_loss}
\partial_v f(v,w) = \nabla h_1(v) - v\odot \alpha^2 \qandq \partial_w f(v,w) =\frac{1}{\lambda} \nabla h_2(w) - \lambda w\odot \xi^2
\end{equation}
where $(\alpha,\xi) \in \argmax \phi(v,w)$ solve the inner problem given in  Proposition \ref{prop:robust_loss}. Note that these formulas are well-defined  for $v$ and $w$ such that the maximisers to the inner problem exist: in this case, $v\odot \alpha$ and $w\odot \xi$ are unique thanks to the quadratic terms $\norm{v\odot \alpha}^2$ and $\norm{w\odot \xi}^2$ in  $\phi$. It is straightforward to conclude using Theorem \ref{thm:rockafellar} that $f$ is differentiable whenever $v$ and $w$ have all nonzero entries. To see differentiability in general, one should assume that there exists $x_0$, $u_0$, $z_0$  such that
\begin{equation}\label{eq:should_assume}
Ax_0 -y = z_0\odot w \qandq u\odot v = Lx_0.
\end{equation}
Indeed, to establish the existence of maximisers, following \eqref{sec:diff_basis_pursuit}, one first observes that $(0,0)$ is a feasible point to the inner maximisation problem in $\phi(v,w)$, so
$
\phi(v,w) \geq 0
$
and one can restrict the maximisation problem to $\alpha,\xi$ such that $L^\top \alpha + A^\top \xi = 0$ and 
$$
\dotp{\xi}{y} \geq \frac12\norm{v\odot \alpha}^2 +\frac{\la}{2}\norm{w\odot \xi}^2.
$$
From this, 
in general, it is not clear that one can extract uniformly bounded (and hence convergent up to a subsequence) maximising sequences $\alpha_n$ and $\xi_n$. However, if  \eqref{eq:should_assume} holds, then
$$
\norm{v\odot \alpha}\norm{u_0} + \norm{w\odot \xi}\norm{z_0}  \geq  \dotp{\xi}{ Ax_0 -z_0 \odot w } \geq \frac12\norm{v\odot \alpha}^2 +\frac{\la}{2}\norm{w\odot \xi}^2
$$
from which it is clear  that $\norm{v\odot \alpha}$ and $\norm{w\odot \xi}$ are uniformly bounded. One can then proceed as in Section \ref{sec:diff_basis_pursuit} to extract maximising sequences to deduce the existence of maximisers and hence differentiability of $f$.
\subsubsection{TV-L1} 

We consider the case where $R_1$ and $R_2$ are both $\ell_1$-norms, so that $h_1 = h_2 = \frac12 \norm{\cdot}^2$.
The use of the $\ell^1$ norm as a loss function is popular to cope with outliers and impulse noise, and a typical example is when $L$ is a finite difference approximation of the gradient operator (as defined in Section~\ref{sec:TV}), corresponding to the so-called TV-L1 method~\cite{nikolova2004variational}.
Given solutions $(\al,\xi,\xp)$ to the linear system  \eqref{eq:inner_robust} (note that $\al$ and $\xi$ are unique defined on the support  of $v$ and $w$ respectively), the  gradient of $f$  is 
$$
\partial_v f = v - v\odot\alpha^2 \qandq \partial_w f = w-w \odot\xi^2.
$$
In figure \ref{fig:tvl1}, we show the results of denoising where $A=\Id$,  $L$  and  $R_1$  correspond to group-TV  as described in Section \ref{sec:TV} and  $R_2$ is a  group-$\ell_1$ norm where each group is of size 3 corresponding to the 3 colour channels. For each image, we corrupt 25\% of the pixels with  salt and pepper noise and show the objective convergence error for different regularisation strengths $\lambda$. We compare against Primal-Dual (whose implementation is as described in the appendix).

\begin{figure}
\begin{tabular}{c@{\hspace{3pt}}c@{\hspace{3pt}}c@{\hspace{3pt}}c}
$\lambda = 0.6$&$\lambda = 1.0$&$\lambda = 2.0$\\
\includegraphics[width=0.32\linewidth]{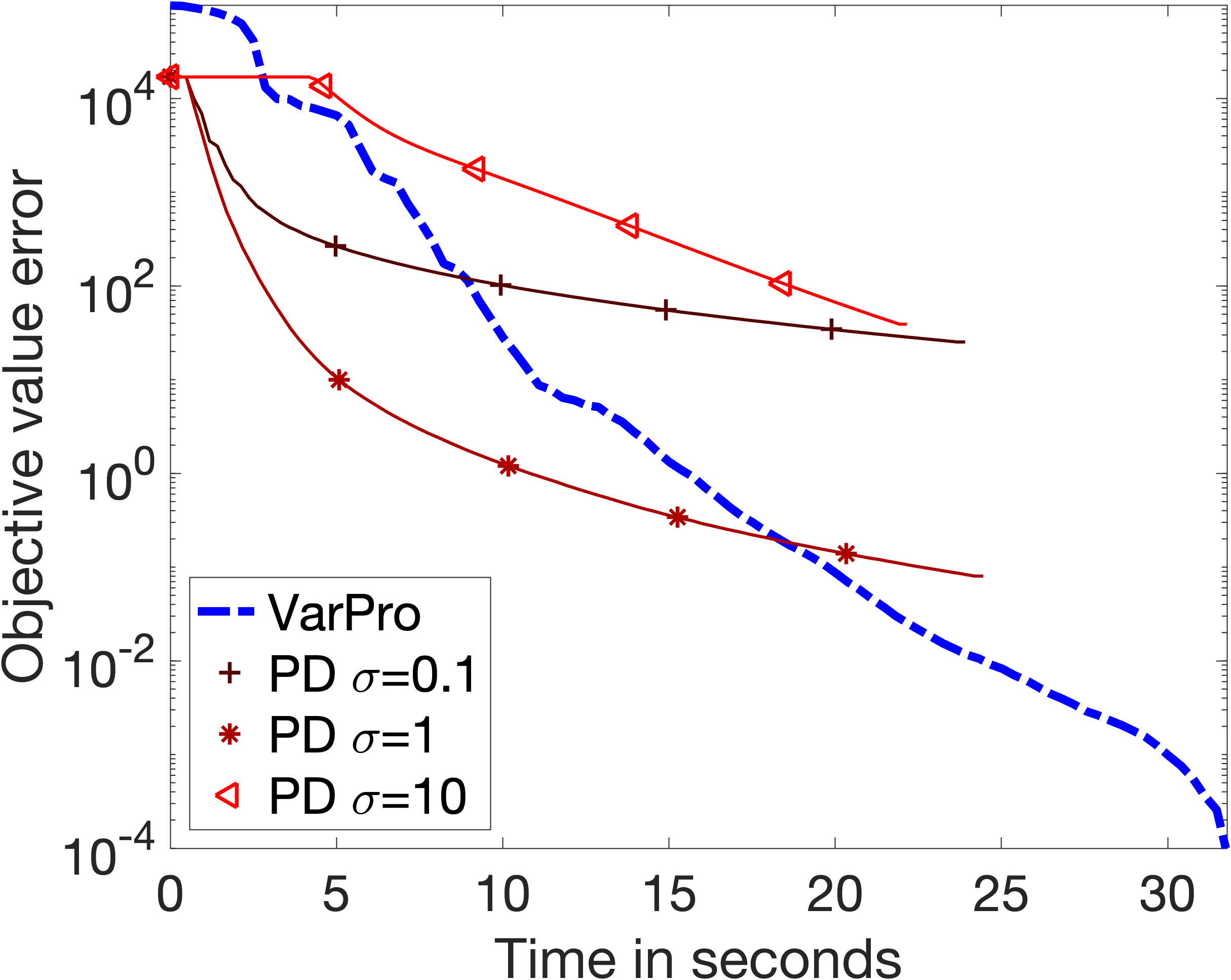}
&\includegraphics[width=0.32\linewidth]{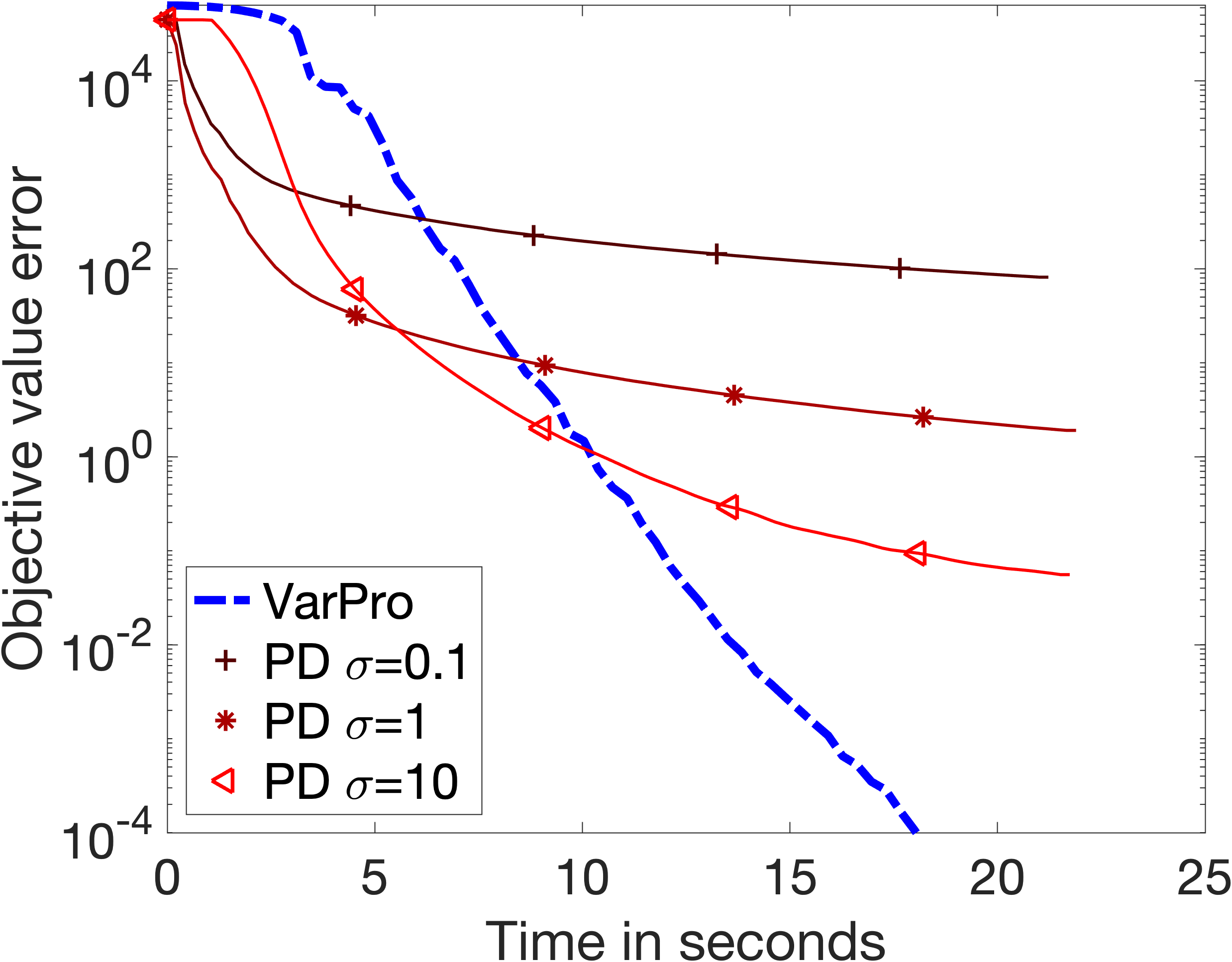}
&\includegraphics[width=0.32\linewidth]{{figures/TVl1/hestain/tvl1proxcalc_227_303_2.00}}
\\
%\includegraphics[width=0.29\linewidth]{{figures/TVl1/pears/tvl1proxcalc_486_732_1.00}}
%&\includegraphics[width=0.29\linewidth]{{figures/TVl1/pears/tvl1proxcalc_486_732_1.00}}
%&\includegraphics[width=0.29\linewidth]{{figures/TVl1/pears/tvl1proxcalc_486_732_2.00}}
%\\
%\includegraphics[width=0.28\linewidth]{{figures/TVl1/pears/tvl1Reconstr_proxcalc_486_732_1.00}}
%&\includegraphics[width=0.28\linewidth]{{figures/TVl1/pears/tvl1Reconstr_proxcalc_486_732_1.00}}
%&\includegraphics[width=0.28\linewidth]{{figures/TVl1/pears/tvl1Reconstr_proxcalc_486_732_2.00}}
%\\
\includegraphics[width=0.32\linewidth]{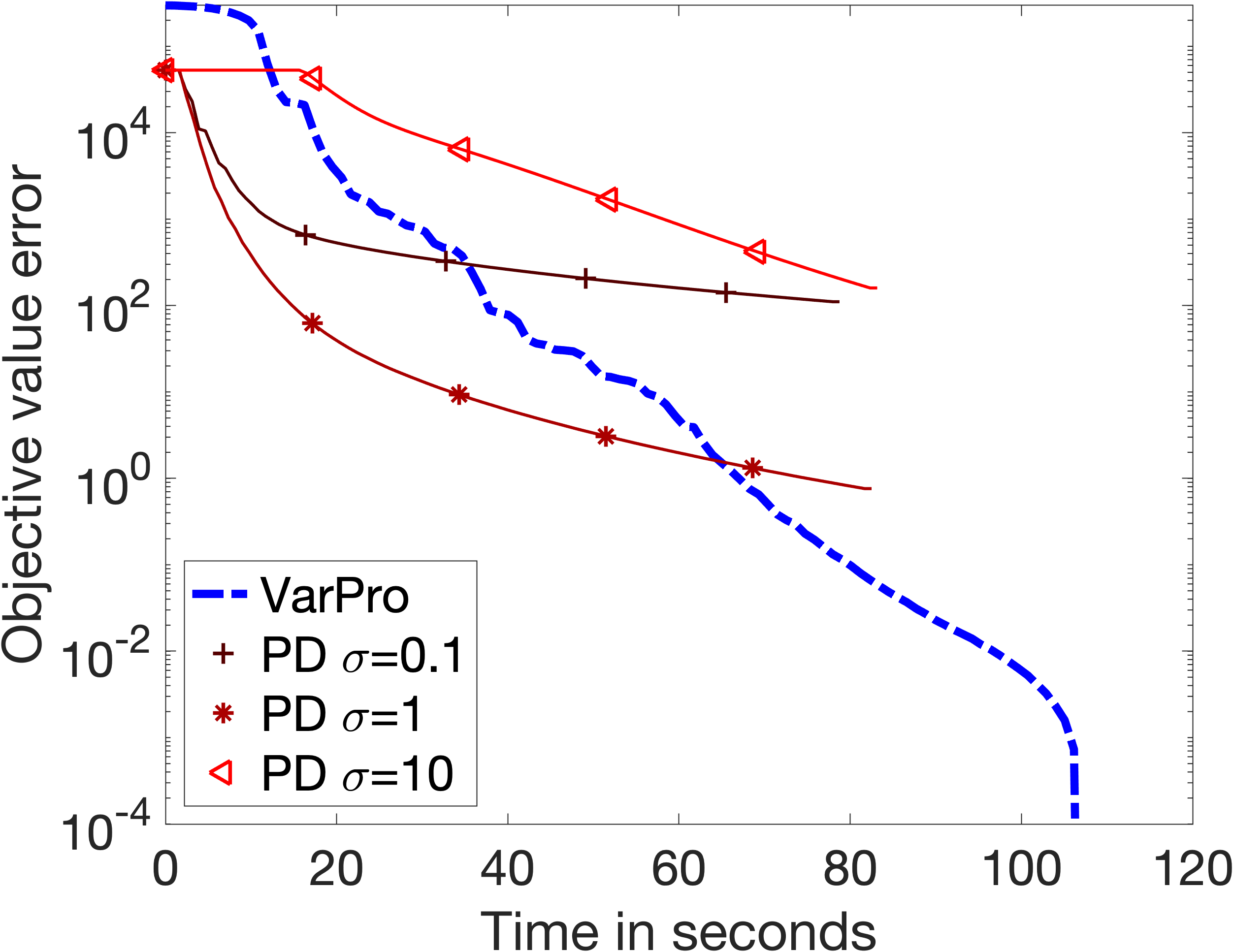}
&\includegraphics[width=0.32\linewidth]{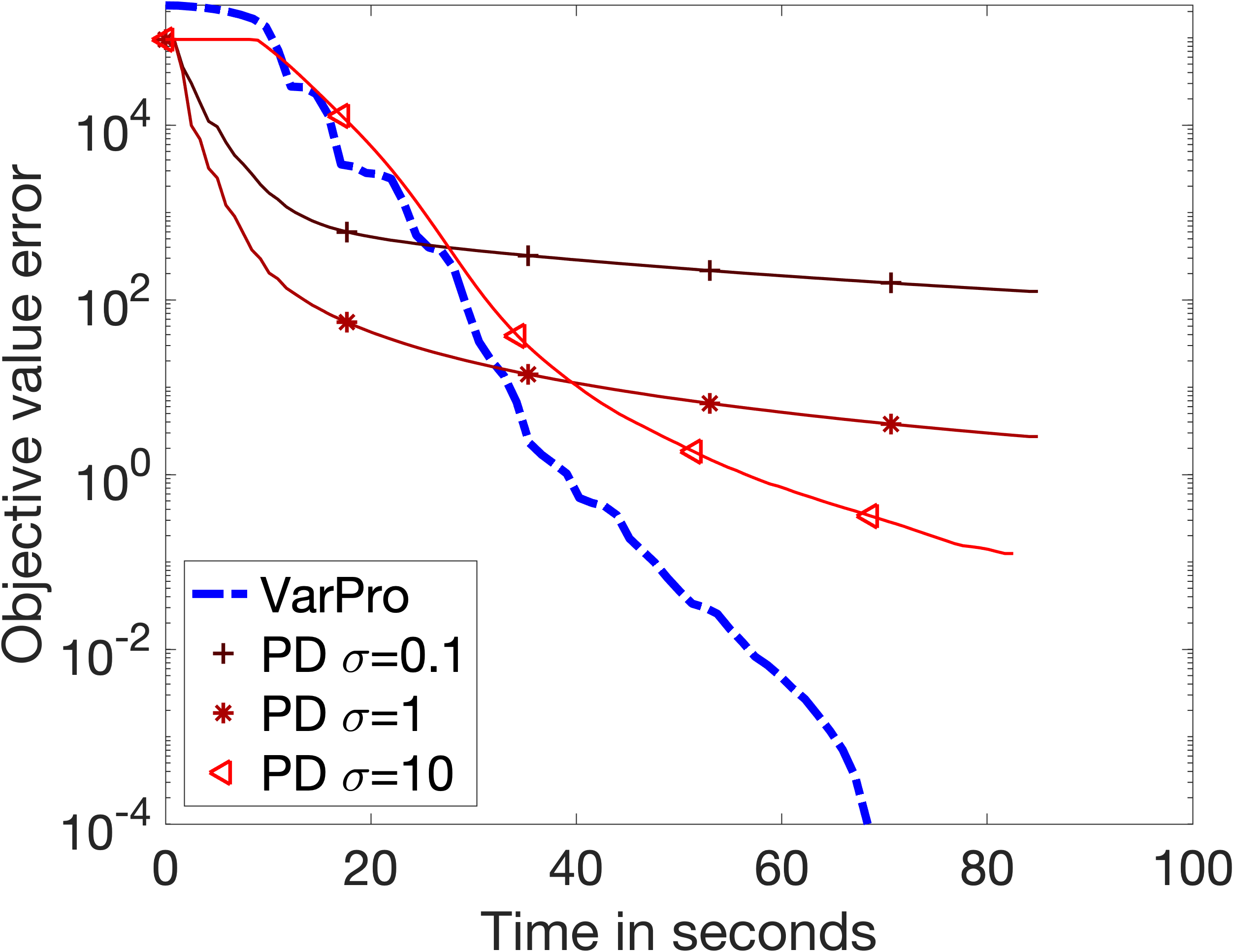}
&\includegraphics[width=0.32\linewidth]{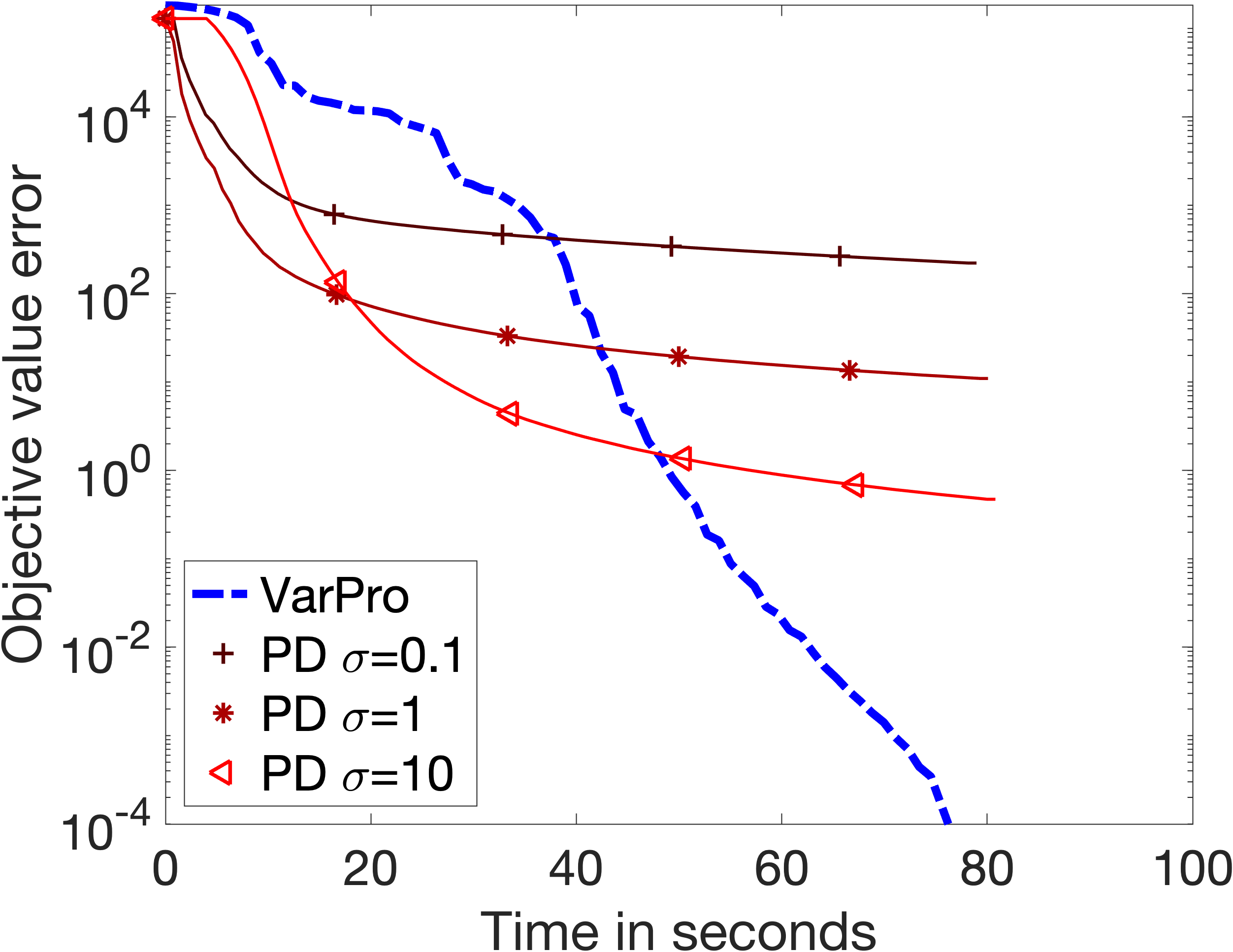}
\end{tabular}
\caption{Comparison of Primal-Dual and VarPro for TV-$L_1$ denoising with salt and pepper noise. The top row corresponds to the image ``hestain'' of size $227\times 303$ and the bottom row corresponds to the image ``peppers" of size $384\times 512$. The figures show the objective error against computational time. \label{fig:tvl1} }
\end{figure}

\begin{figure}
\begin{tabular}{c@{\hspace{1pt}}c@{\hspace{1pt}}c@{\hspace{1pt}}c}
$\lambda = 0.6$&$\lambda = 1.0$&$\lambda = 2.0$&Input\\
\includegraphics[width=0.23\linewidth]{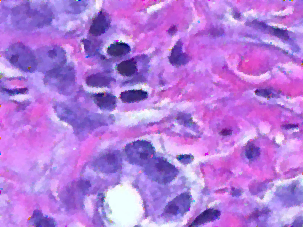}
&\includegraphics[width=0.23\linewidth]{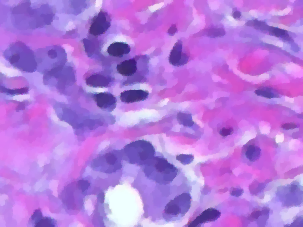}
&\includegraphics[width=0.23\linewidth]{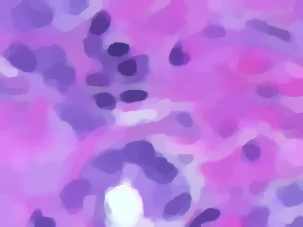}&\includegraphics[width=0.23\linewidth]{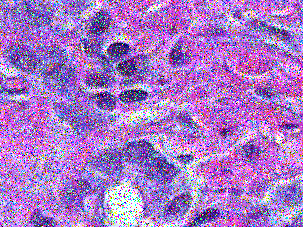}
\\
\includegraphics[width=0.23\linewidth]{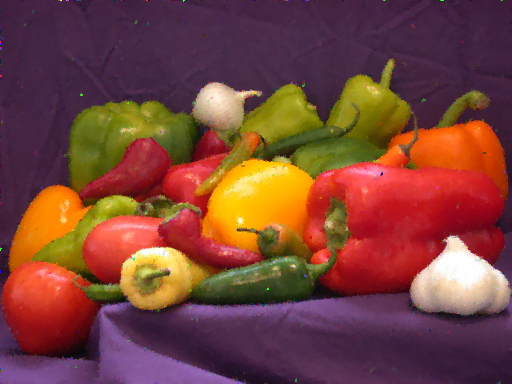}
&\includegraphics[width=0.23\linewidth]{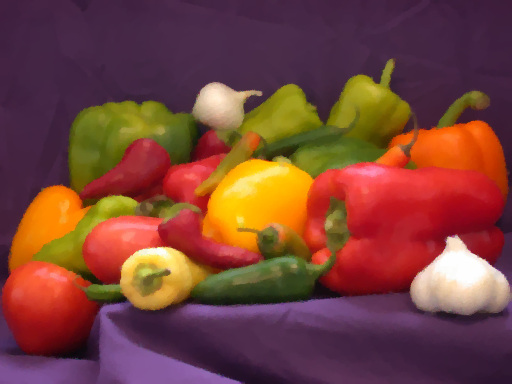}
&\includegraphics[width=0.23\linewidth]{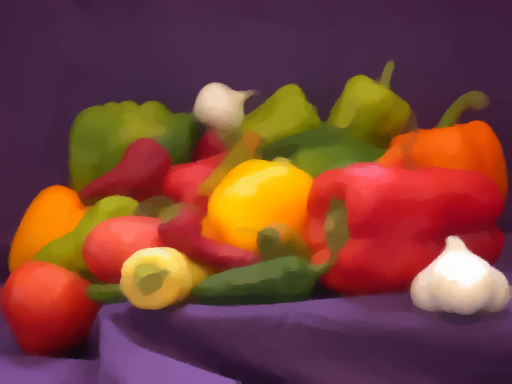}
&\includegraphics[width=0.23\linewidth]{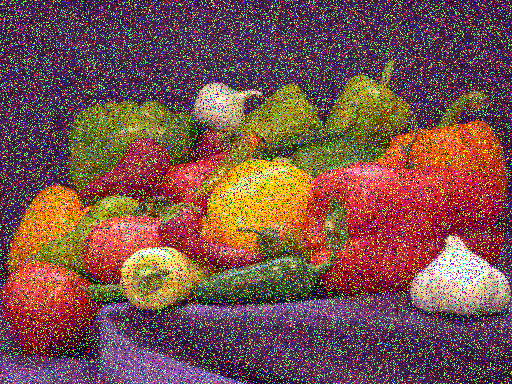}
\end{tabular}
\caption{Comparison of Primal-Dual and VarPro for TV-$L_1$ denoising with salt and pepper noise. The far right images are the noisy input images. The other images, hestain in the first row and peppers in the second row, are the outputs of VarPro.\label{fig:tvl1_images} }
\end{figure}

\subsubsection{The square root Lasso}

We now consider the case of $L=\Id$, $A\in \RR^{m\times n}$ and $y\in\RR^m$, and $R_i$ are the group-$\ell_1$ norms. This corresponds to the square root group lasso
\begin{align}\label{eq:sqrtlasso}
\min_\xp  \norm{\xp}_1 +\frac{1}{\lambda \sqrt{m}} \norm{A\xp- y}_{2}.
\end{align}
This optimization problem is equivalent to the original group Lasso problem but this equivalence requires to tune the multipliers $\la$ involved in both problem, which depends on the input data $y$ and the design matrix $A$~\cite{belloni2011square}. An advantage of the square root Lasso formulation is that, thanks to the 1-homogeneity of the functionals, the parameter $\la$ requires less tuning and is approximately invariant under modification of the noise level and number of observations. Another setting is used in multitask learning is  where the Loss function is the nuclear norm \cite{geer2016chi} (see Appendix \ref{app:matrices} for remarks on the use of quadratic variational forms for this setting).
The equivalent VarPro formulation is an optimisation problem over $n+1$ variables
 \begin{equation*}
	\min_{w\in\RR,v\in\RR^n}  f(v,w) , 
\end{equation*}
\begin{equation*}
\text{where }
f(v,w) =\frac12 \norm{v}^2 +\frac{1}{2\lambda \sqrt{m}} w^2+\max_{\xi} -\frac{1}{2} \norm{v\odot( A^\top \xi)}^2 -\frac{\lambda\sqrt{m}}{2} w^2 \norm{\xi}^2  + \dotp{  y }{\xi}
\end{equation*}
with $x=v^2 \odot A^\top \xi$.

We consider two examples: (i) $A$ is a random Gaussian matrix with $m=300$ and $n=2000$ and $y=Ax_0+w$ where $x_0$ is $s=40$-sparse and the entries of $w$ are iid Gaussian with variance 0.01. (i) $A$ is the MNIST dataset \footnote{\url{https://www.csie.ntu.edu.tw/~cjlin/libsvmtools/datasets/}}, with $m=60,000$ and $n=683$. Note that the first order optimality condition of~\eqref{eq:sqrtlasso} is
$$
0\in  \partial \norm{x}_1 + \frac{1}{\lambda \sqrt{m}}A^\top \partial \norm{ A\xp-y }= 0
$$
and $0$ is a solution if $\lambda \norm{y}\sqrt{m} \geq \norm{A^\top y}_\infty $. We therefore define $\lambda_{\max} \eqdef \frac{1}{\norm{y}\sqrt{m}}\norm{A^\top y}_\infty$ and consider $\lambda = \frac{1}{p}\lambda_{\max}$ for different $p>1$.

Two popular approaches to solve this optimisation problem in the literature are alternating minimisation~\cite{giraud2021introduction} and coordinate descent \cite{belloni2011square,ndiaye2017efficient}. The alternating approach iterates between the following steps \cite{giraud2021introduction}
$$
	\eta_k \eqdef \norm{A\xp_k - y}
	\qandq
	x_{k+1} \in \argmin_{\xp} \lambda\norm{\xp}_1 + \frac{1}{2\eta_k} \norm{A\xp - y}^2.
$$
This is  referred to as the scaled Lasso algorithm and requires solving the Lasso at each step. In Figure \ref{fig:sqrtlasso}, we compare against these two approaches -- the scaled Lasso algorithm uses Varpro for the Lasso as the inner solver and the coordinate descent code can be  downloaded online \cite{belloni2011square} \footnote{\url{https://faculty.fuqua.duke.edu/~abn5/belloni-software.html}}. One can observe that for coordinate descent, although it is very effective for large regularisation parameters, its performance deteriorates for small $\lambda$. In contrast, our proposed method is robust to different regularisation strengths.

\begin{figure}
\begin{tabular}{ccc}
$\lambda =\lambda_{\max}/ 2$&$\lambda =\lambda_{\max}/ 4$& $\lambda =\lambda_{\max}/ 20$\\
\includegraphics[width=0.3\linewidth]{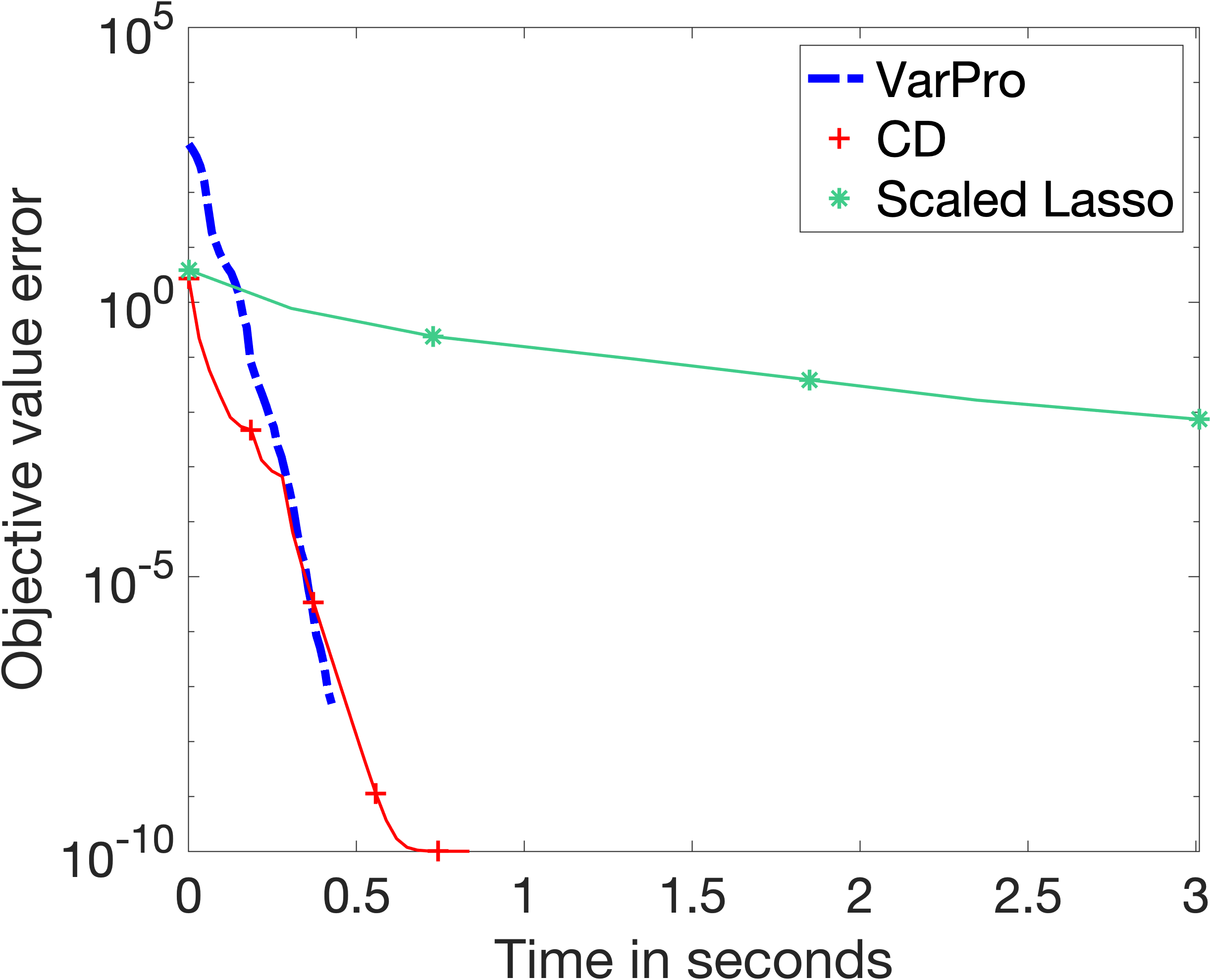} &
\includegraphics[width=0.3\linewidth]{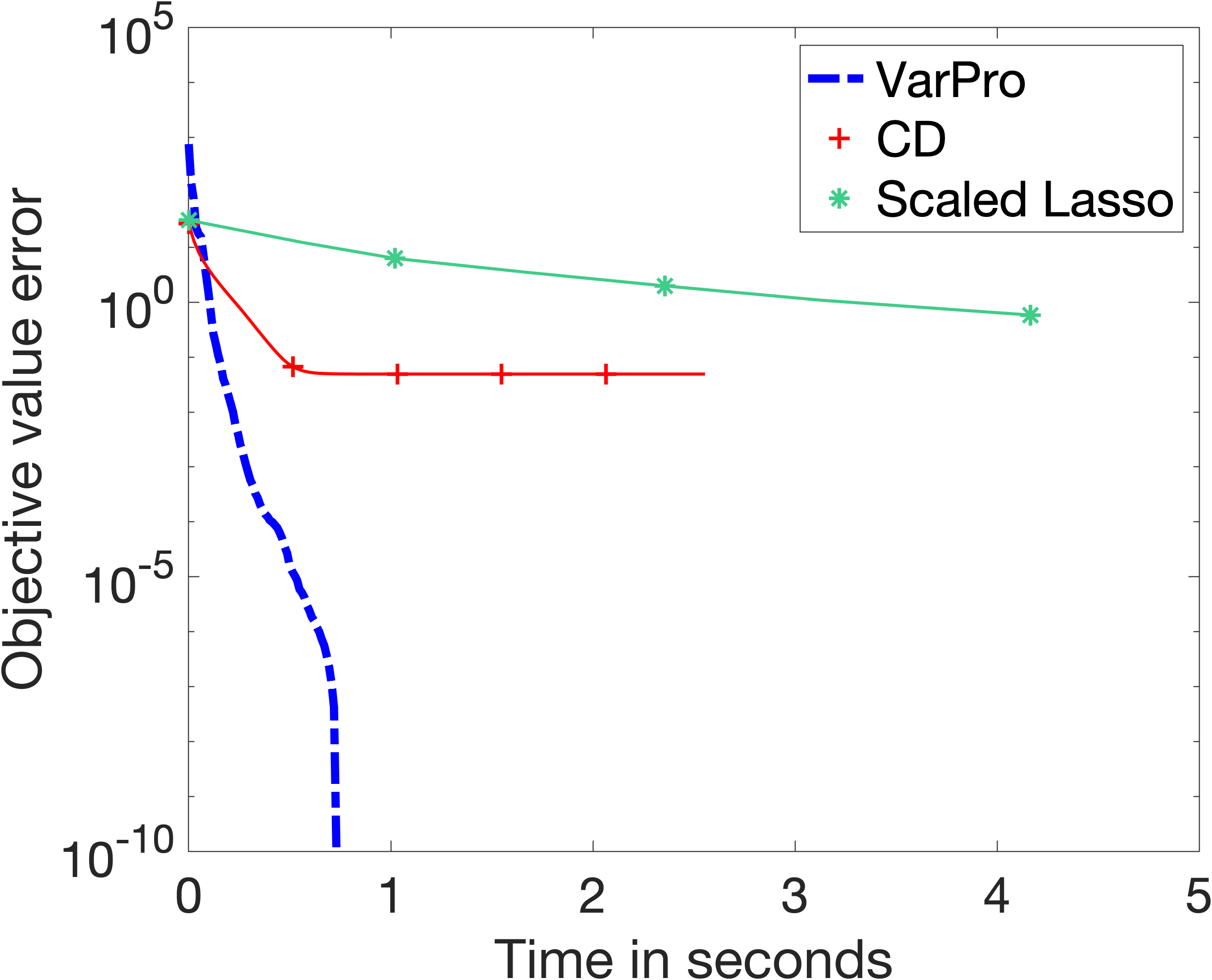} &
\includegraphics[width=0.3\linewidth]{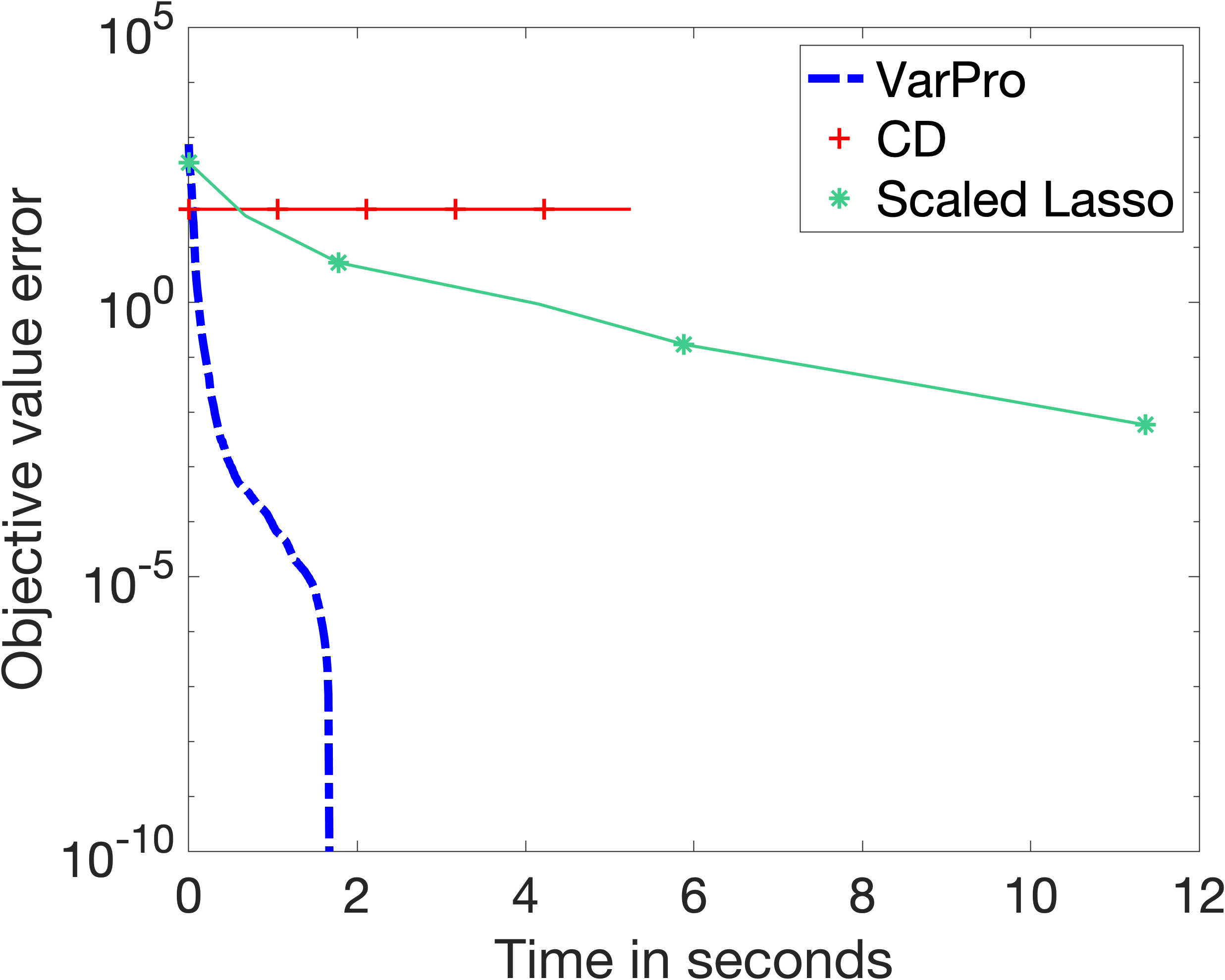} \\
\includegraphics[width=0.3\linewidth]{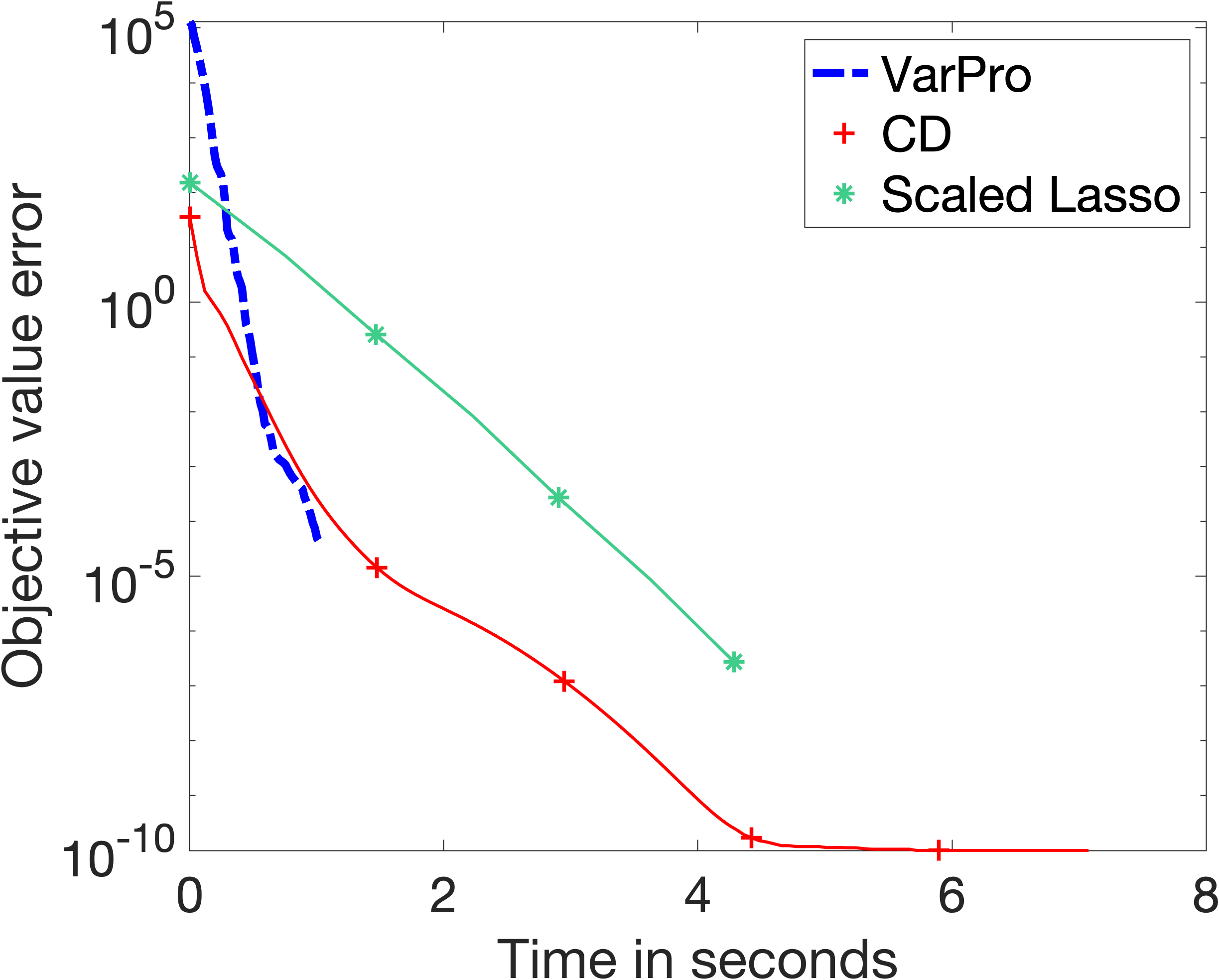} &
\includegraphics[width=0.3\linewidth]{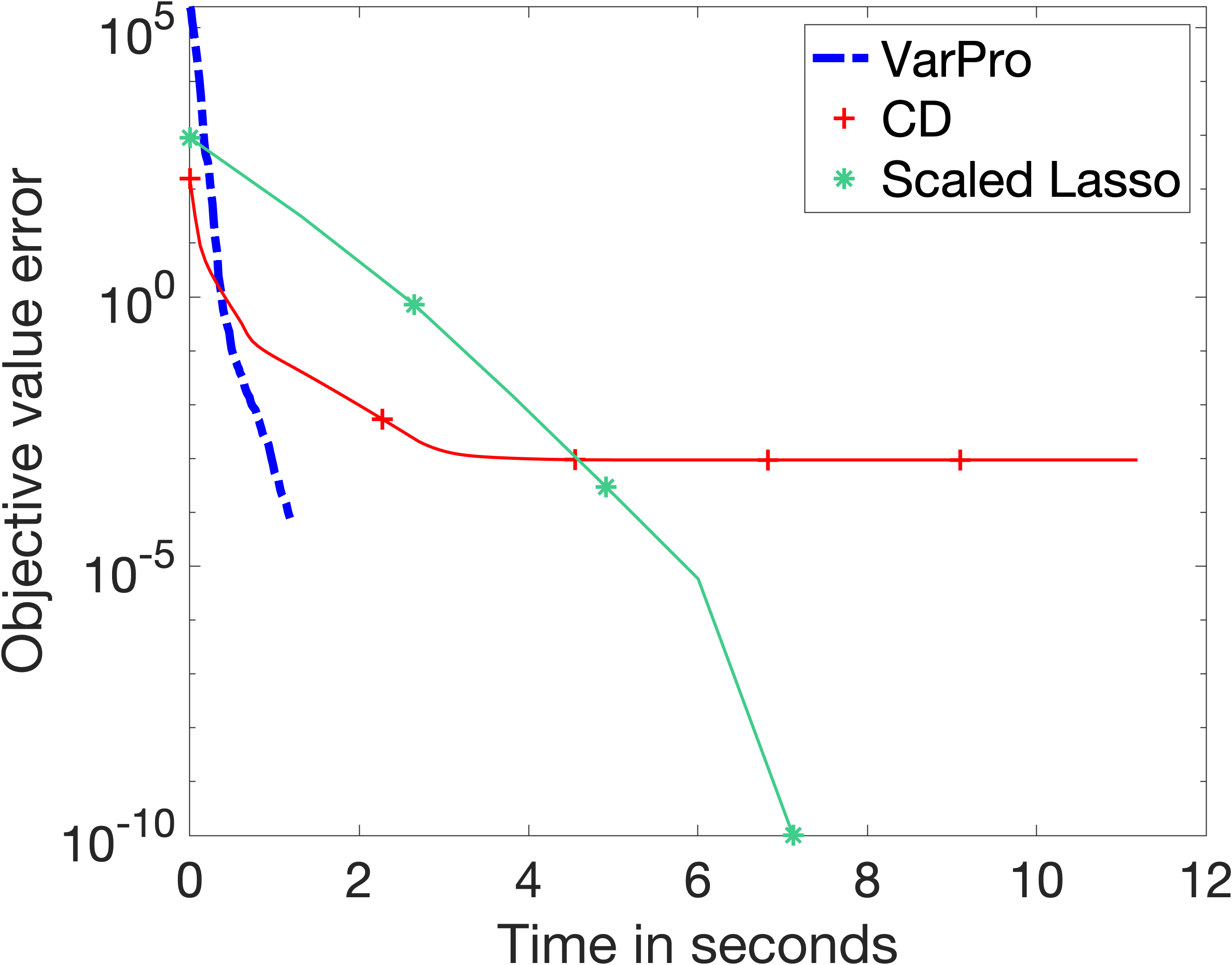} &
\includegraphics[width=0.3\linewidth]{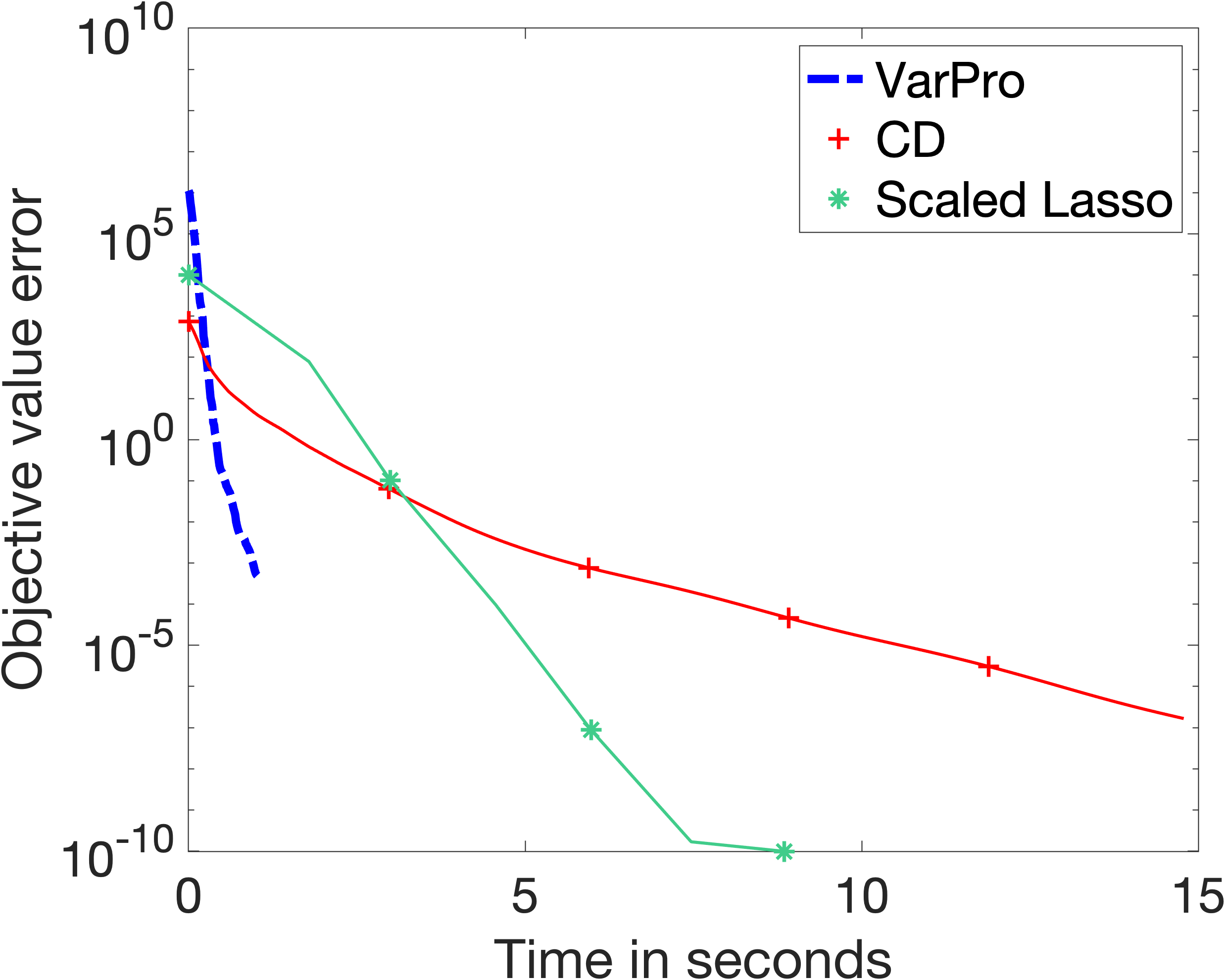} 
\end{tabular}
\caption{Comparison of Coordinate Descent (CD), Scaled Lasso and VarPro for the square root Lasso. Plot of objective error against computational time. Top row: Random Gaussian matrix with $m=300$, $n=2000$ and $y = A x_0 + w$ where $x_0$ is $s=40$ sparse and the entries of $w$ are iid Gaussian with variance $0.01$. Bottom row: MNIST dataset where $m=60000$ and $n=683$. \label{fig:sqrtlasso}}
\end{figure}

\subsection{Nonconvex regularizations}
\label{sec:nonconvex}

Let us consider the case where  the regulariser is a group $\ell_q$ semi-norm, denote $\norm{x}_{q,2} = \pa{ \sum_{g\in\Gg} \norm{x_g}^q}^{1/q}$ for $q\in (0,1)$. From \eqref{lem:lq}, we observe that if $q>2/3$, then the Hadamard parameterization gives 
$$
	\frac{1}{q}\sum_{g\in\Gg} \norm{x_g}^q= \min_{x = u \odot v} \frac12 \norm{u}^2 + \frac{1}{2\al} \sum_{g\in\Gg} \abs{v_g}^{2\beta}, 
$$
with $2\beta>0$, so that the VarPro method corresponds to a differentiable optimisation problem, as exposed in \cite{poon2021smooth}. 
To go below $p=2/3$ and induce a stronger sparsity regularization while maintaining differentiability, we propose to introduce more over-parameterization. We expose the setting of a three-factors parametrization, but this is easily generalizable to more factors to further reduce the value of $q$. By doing so, this also raises the question of exploring multiple-levels optimizations (beyond simply bilevel programming).

%\todo{Gab: to be consistent with above, maybe rename $(p,q)$ as $(\alpha,p)$?}

\begin{prop}
Let $q,\beta>0$ be such that  $\beta= \frac{q}{ (2 -2q) }$.
Then,
\begin{equation}
\frac{1}{q}\sum_{g\in\Gg} \abs{x_g}^q = \min\enscond{ \frac12 \norm{u}_2^2 +\frac12 \norm{v}_2^2 + \frac{1}{2 \beta}\sum_{g} \abs{w_g}^{2\beta}}{x =  u\odot (v \cdot w)},
\end{equation}
where the minimisation is over $u\in\RR^p$ and $v,w\in\RR^{\abs{\Gg}}$.
\end{prop}

\begin{proof}
By applying Lemma \ref{lem:lq} twice,
\begin{align*}
& \min_{x =  u\odot( v\cdot w)} \frac12 \norm{u}_2^2 +\frac12 \norm{v}_2^2 + \frac{1}{2\beta}\sum_{g\in\Gg } \abs{w_g}^{2\beta}\\
 &=  \min_{x = u\odot z} \frac12 \norm{u}_2^2 + \min_{z = v\cdot w}\pa{\frac12 \norm{v}_2^2 + \frac{1}{2\beta}\sum_{g\in\Gg} \abs{w_g}^{2\beta}}\\
 &=  \min_{x =  u\cdot z} \frac12 \norm{u}_2^2 + \frac{1}{2r}\sum_{g\in\Gg} \abs{z_g}^{2r} \qwhereq  r = \beta/(1+\beta)\\
 &=\frac{1}{q} \sum_{g\in\Gg} \norm{x_g}^q \qwhereq  q = \frac{2r}{(1+r)} =  \frac{2\beta}{(1+2\beta)}  ,
\end{align*}
and equivalently,  $\beta= \frac{q}{ (2 -2q) }$.
\end{proof}

The VarPro problem is differentiable if and only if $2\beta>1$, which equivalently $q>1/2$.
We now consider the case $q=2/3$, so that given a convex loss function $F$, we re-write 
$$
\min_\xp \frac{3}{2}\norm{\xp}_{2/3}+ F_0(A\xp)
$$
equivalently as 
\begin{equation}\label{eq:3split}
\min_{u,v,w } \frac{1}{2} \pa{\norm{u}_2^2+\norm{v}_2^2+\norm{w}_2^2}+ F_0(A( u\odot (v\cdot w))).
\end{equation}
One can deal with this optimization problem in three ways: as an optimisation over 3 variables ; formulate a VarPro bilevel problem with two variables $(v,w)$ on the outer problem; formulate as a bilevel problem with one variable $v$ on the outer problem. We first discuss differentiability issues of each of these three cases.

\paragraph{Option 1: Optimisation over 3 variables}

One can directly optimise \eqref{eq:3split}  over the  three variables $u,v,w$ and this  is differentiable when  $F_0$ is differentiable.

\paragraph{Option 2: Two variables on outer problem} 
One rewrite \eqref{eq:3split} as $\min_{v,w} f(v,w)$ where
\begin{equation}\label{eq:varpro_1inside}
\begin{split}
f(v,w) &= \min_{u} \frac{1}{2} \pa{\norm{u}^2+\norm{v}^2+\norm{w}^2}+ F_0(A( u\odot( v\cdot w)))\\
&=\frac{\norm{v}^2+\norm{w}^2}{2}+ \max_\alpha -\frac{1}{2} {\norm{(v\cdot w)\odot A^\top \alpha}^2}  - F_0^*(\alpha),
\end{split}
\end{equation}
where we observe that the minimisation problem over $u$ is convex and the second line is the result of convex duality.
Assuming that $F_0$ is convex,  the inner problem  over $u$ is strongly convex and $f$ is differentiable whenever $F_0$ is differentiable. Note however that even when $F_0$ is not differentiable, the solution to the inner problem is unique and one can  write
$$
\partial_v f = v - v \cdot \pa{ w_g^2 \norm{A_g^\top \alpha}^2}_{g\in\Gg} \qandq \partial_w f = w  - w\cdot \pa{ v_g \norm{A_g^\top \alpha}^2}_{g\in\Gg} ,
$$
where $\alpha$ is a dual solution to the inner problem. These formulas are well defined since $(v\cdot w) \odot A^\top \alpha$ is unique.

\paragraph{Option 3: One variable on the outer problem}
One can consider $\min_v f(v)$ where
\begin{equation}\label{eq:2inside}
\begin{split}
f(v) &= \min_{u,w} \frac{1}{2} \pa{\norm{u}^2+\norm{v}^2+\norm{w}^2}+ F_0(A( u\odot (v\cdot w)))\\
&=\min_{z} \frac{1}{2} \norm{v}^2+\norm{z}_{1,2}+ F_0(A(v\odot z))\\
&= \min_{z} \max_\alpha \frac{1}{2} \norm{v}^2+\norm{z}_{1,2}+ \dotp{\alpha}{A(v\odot z)} -F_0^*(\alpha).
\end{split}
\end{equation}
Note that the inner minimisation problem is convex and provided that the inner problem has a unique solution $z$ and $F_0$ is differentiable,  the function $f$ is differentiable with $$
\nabla f(v) = v +( \dotp{z_g}{ A_g^\top \alpha})_{g\in\Gg}
$$
Indeed, since $v\odot A^\top \alpha \in\partial \norm{z}_{1,2}$, $v\odot A^\top \alpha$ is unique on the support of $z$ and since the group support of $z$ is contained in the support of $v$, $z\odot A^\top \alpha$ is unique. One condition to ensure that $z$ is unique is if $F_0$ is strongly convex and $A$ is injective.

\paragraph{Remarks on conditioning}
This analysis raises the question of the which option to favor among the three. 
It is clear that Option 3 is computationally the most expensive, since one needs to solve an $\ell_1$ minimisation problem to compute the gradient, while the gradient in Option 2 can be computed in closed form by inverting a linear system. 
For Option 3, the resolution of this inner $\ell_1$ problem can leverage any existing solvers, and it is possible to re-use another VarPro method, which corresponds to doing a three-level programming. We explore this option in the numerical examples below.  
However, in terms of conditioning of the Hessian, Option 3 is the most desirable as we now explain.

Given a bloc-matrix $H = [A,B; B^\top,D]$, its Schur complement with respect to $D$ is denoted $H/D \eqdef A - B^\top D^{-1} B$.
Given $(u,v,w)\mapsto f(u,v,w)$ and denoting its Hessian by $H$, the Hessian of $u\mapsto \min_{v,w} f(u,v,w)$ is the Schur complement of $H$ with respect to $\partial_{vw}^2 f$ and the Hessian of  $(u,v)\mapsto \min_{w} f(u,v,w)$ is the Schur complement of $H$ with respect to $\nabla^2_{u} f(u,v,w)$, assuming that these Hessians exist.
In general, the condition number of the Schur complement of a matrix is no larger than the condition number of the original matrix when it is symmetric positive (or negative) semi-definite
due to the interlacing property of eigenvalues \cite{smith1992some}, that is for an $r\times r$ submatrix of $H$, we have
$$
\lambda_i(H) \leq \lambda_i(H/A) \leq \lambda_{i+n-r} (H).
$$
We also have the following interlacing property  \cite{fan2002schur} for schur complements: if $H$ is symmetric semi-definite, $\alpha'\subseteq \alpha \subseteq [n]$ and $H[\alpha]$ is the submatrix indexed by $\alpha$,
$$
\lambda_i(H/H[\alpha']) \leq \lambda_i(H/H[\alpha]) \leq \lambda_{i+\abs{\alpha} - \abs{\alpha'}}(H/H[\alpha'])
$$
so the condition number of $H/H[\alpha]$ is no worse than the condition number of $H/H[\alpha']$. Putting aside the difficulty that $v\mapsto \min_{u,w} f(u,v,w)$ may not be differentiable, these interlacing properties of Schur complements suggest that this formulation  is better conditioned than the alternative formulation of $(v,w)\mapsto \min_{u} f(u,v,w)$.

\paragraph{Numerical illustrations}

The methods that we compare to are:
\begin{itemize}
\item VarPro (1 inside). This is minimising the function $f$ defined in \eqref{eq:varpro_1inside} where the inner problem is the solution to a linear system.
\item VarPro (2 inside). This is minimising the function$f$ defined in \eqref{eq:2inside} where the inner problem is the solution to an $\ell_1$ problem.
\item Iterative reweighted least squares (we followed the implementation as described in  \cite{chartrand2008iteratively}).
\item  Reweighted $\ell_1$ \cite{chen2010convergence}.
\end{itemize}
Since the $\ell^q$ regularizer we consider is non-convex, we analyze the performances of these algorithms according both to its ability to select ``good '' minimizers and its speed of convergence. 

In order to asses the quality of the computed solutions, we consider a noiseless recovery problem from observations $Y = A\xp_*$ where $\xp_*$ has $s$ nonzero rows, and check wether the considered algorithms  are able to recover $x^*$ as a function of $s$.
We consider, for $A\in\RR^{m\times n}$, $Y\in\RR^{m\times T}$ and $q = \frac23$, the constrained problem
$$
\min_{\xp\in\RR^{n\times T}} \sum_{i=1}^n \norm{\xp^{(i)}}_2^q \quad \text{s.t.}\quad A\xp = Y.
$$
Here, $\xp^{(i)}$ denotes the $i$th row of $\xp$. 
We consider the case of $n=256$, $s=40$ and $T\in\ens{1,50,100}$ and several values of $m$. For each $T$, we generate 100 random instances of $A\in\RR^{m\times n}$, $\xp_*\in\RR^{s\times T}$ which is $s$-row sparse and let $y = A\xp_*$.  Then, for each value of $m$, we take the first $m$ rows of $A$ and first $m$ entries of $y$, run our methods for this data, then count the number of times for which one has successful recovery (here, successful means that the relative error in $2$-norm is less than 0.01). The results for $T=1,50,100$ are shown in Figure \ref{fig:noncvx}.
In terms of recovering the phase transition, the observation is that when $T$ is large, VarPro obtains on-par or slightly better than IRLS and reweighted-$\ell_1$. One of the advantages of VarPro with only a linear solve on the inner problem is that it is substantially faster than the other methods (see the next paragraph). When $T=1$, the performance is more varied, with IRLS performs the best, and Varpro  with $\ell_1$ on the inside and reweighted $\ell_1$ both out perform Varpro with a linear solver on the inside. It should be noted that since since IRLS is gradually decreasing a regularisation parameter, it is a form of graduated non-convexity method, and this is likely to lead to better local minimums than directly solving the $\ell_{2/3}$ optimisation problem.

To illustrate the time-performance of the different algorithms, we consider
$$
	\min_{\xp\in\RR^{n\times T}}  \sum_{i=1}^n \norm{\xp^{(i)}}_2^q +\frac{1}{2\lambda}\norm{ A\xp - Y}_F^2
$$
for $\lambda = 0$, and where $A\in\RR^{m\times n}$ is a random Gaussian matrix with $n=256$ and $\xp_*$ is row-sparse with $40$ nonzero rows. For $T=100$ and $T=50$, we set $m=45$ and for $T=1$, we set $m=85$. In each case, we generate 20 random instances, and given each problem instance, we apply each method and record the running time against the objective error (this is the error to the best objective value found by the 4 methods). The results are show in Figure \ref{fig:ncvx_time}.

\begin{figure}
\begin{center}
\begin{tabular}{c@{\hspace{.5em}}c@{\hspace{.5em}}c}
$T=1$&$T=50$&$T=100$\\
\includegraphics[width = 0.31\linewidth]{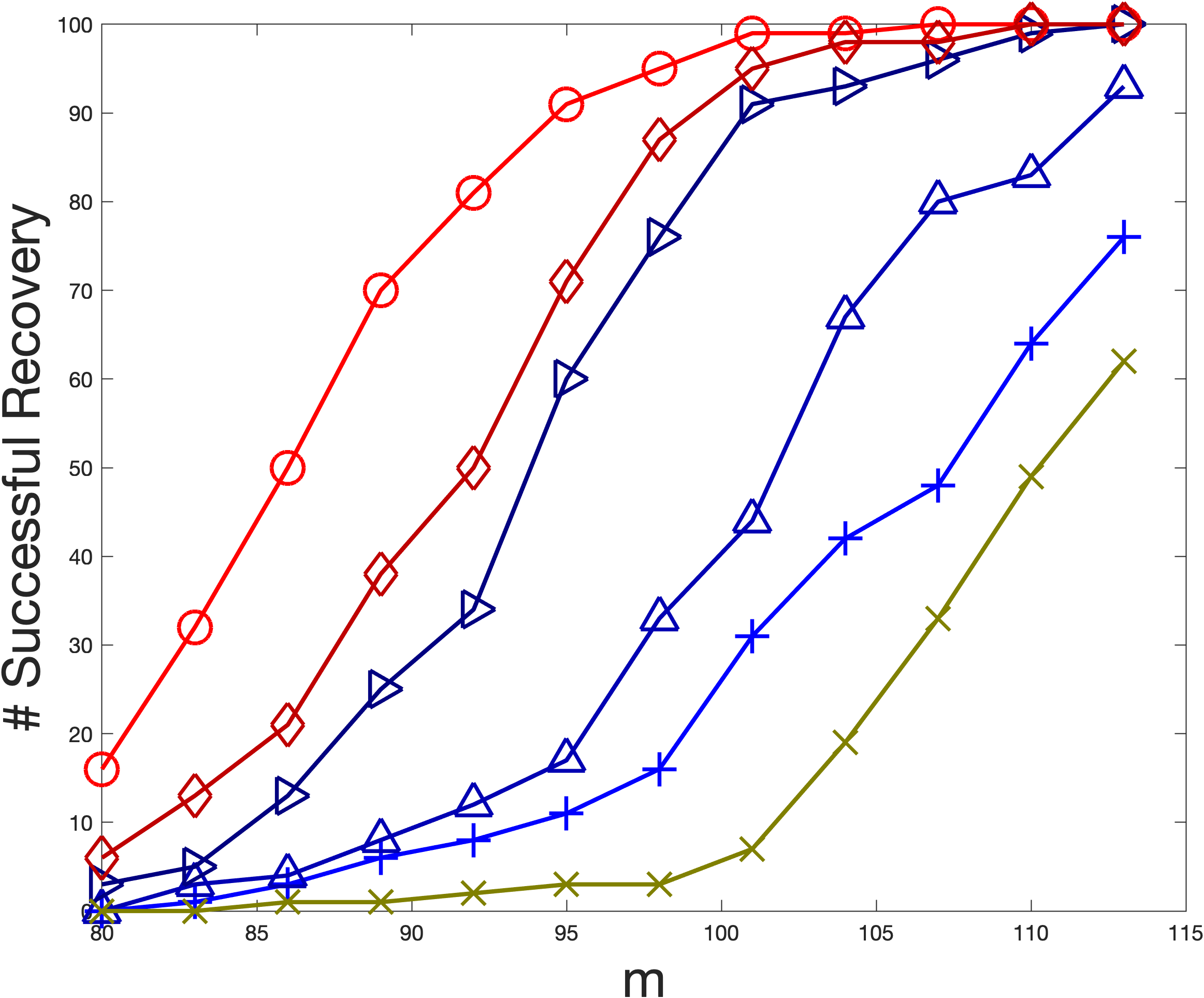}&
\includegraphics[width = 0.31\linewidth]{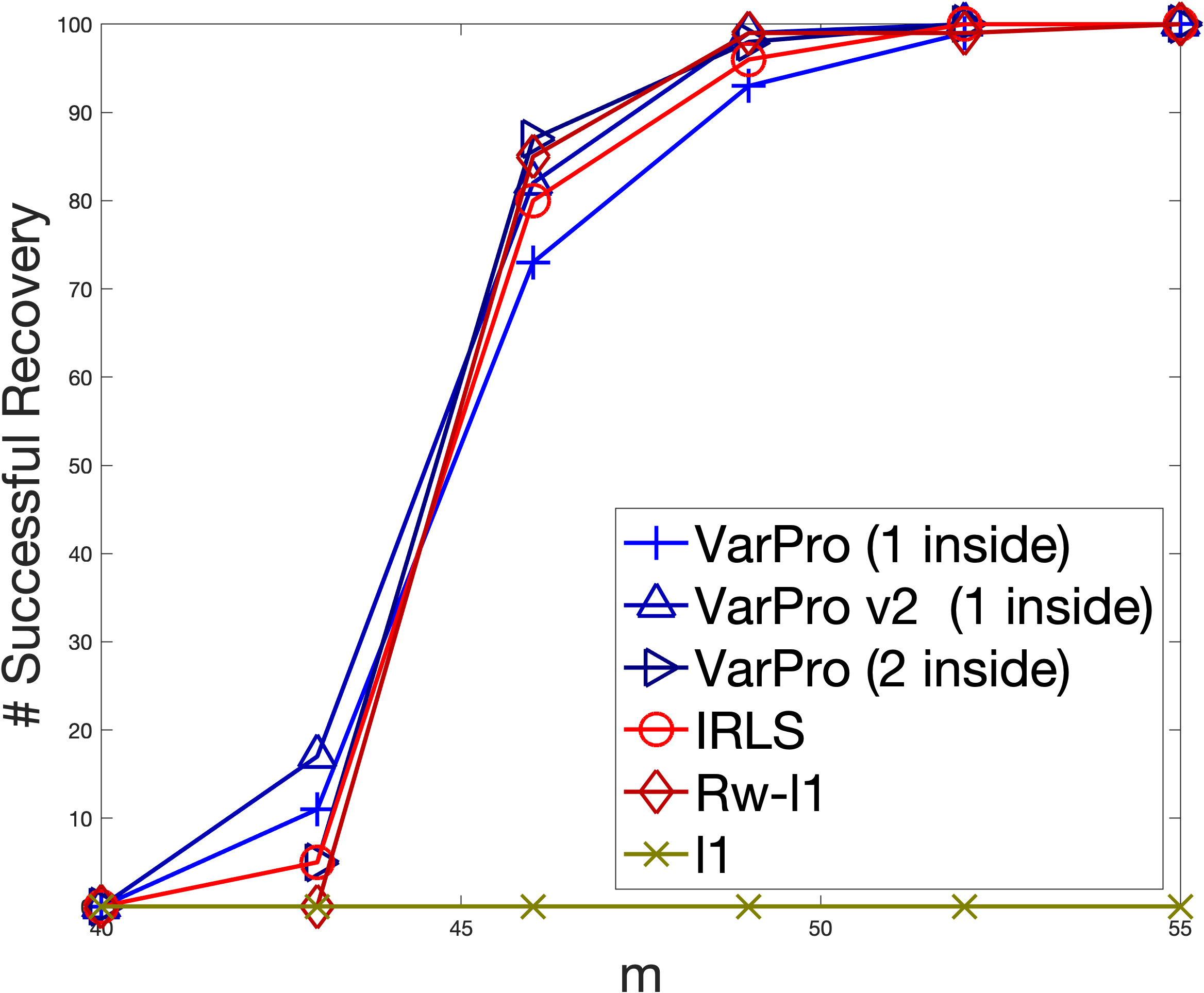}&
\includegraphics[width = 0.31\linewidth]{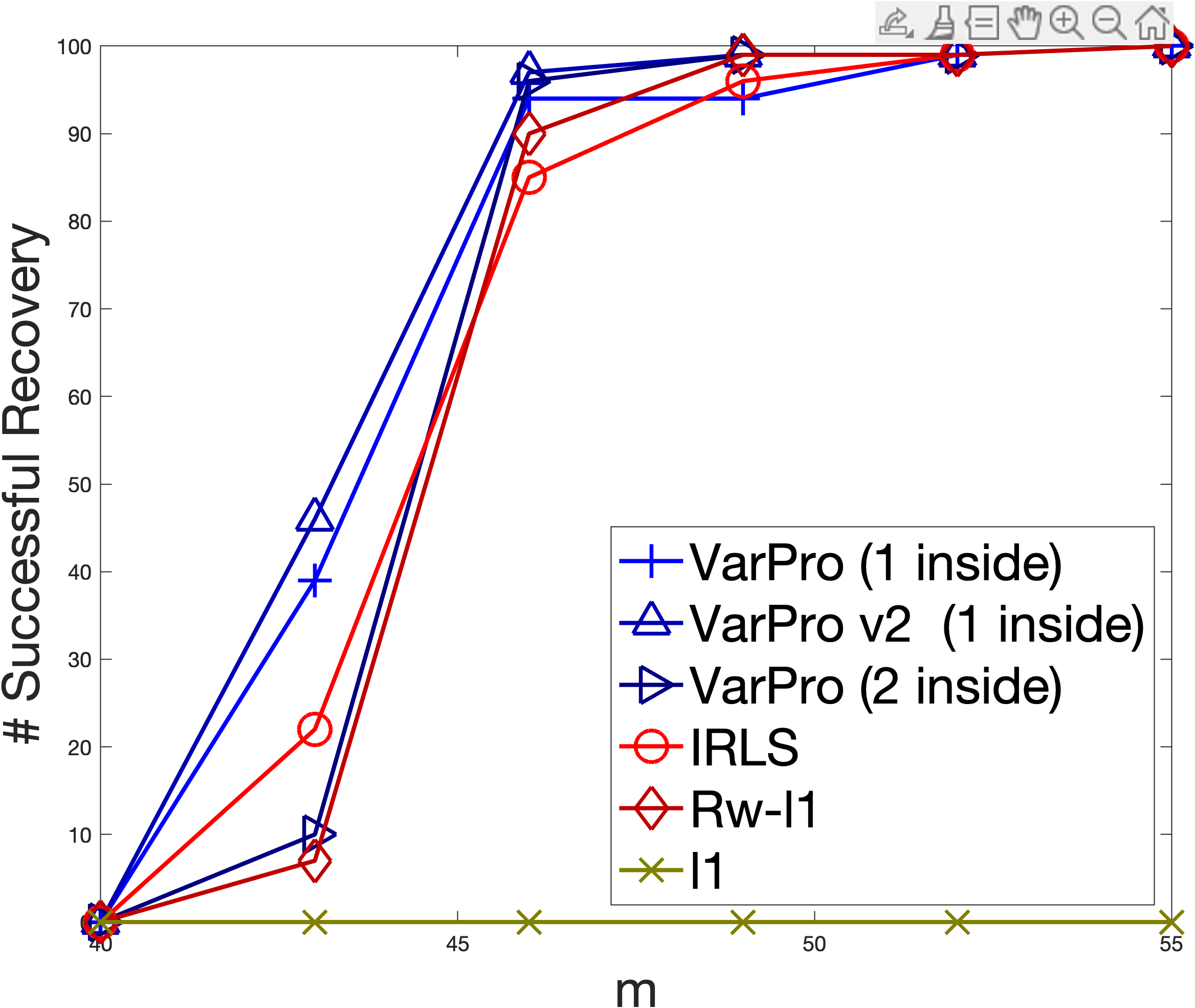}
\end{tabular}
\end{center}
\caption{Group-$\ell_{2/3}$ Basis Pursuit over 100 random instances. The underlying vector $\xp$ a matrix of dimension $n=256$ by $T$, with $ s=40$ nonzero rows. The data matrix is a random Gaussian matrix of dimension $m$ by $n$.  For each of the plots, we consider different values of $m$, and the plot shows for each method, how many of the 100 random problem instances are exactly reconstructed. To show the impact of an improved starting point, VarPro v2 (1 inside) re-runs VarPro (1 inside) an additional 2 times with random initial values and takes the best result.   \label{fig:noncvx}}
\end{figure}

\begin{figure}
\begin{tabular}{c@{\hspace{.5em}}c@{\hspace{.5em}}c}
$T=1$&$T=50$&$T=100$\\
\includegraphics[width = 0.31\linewidth]{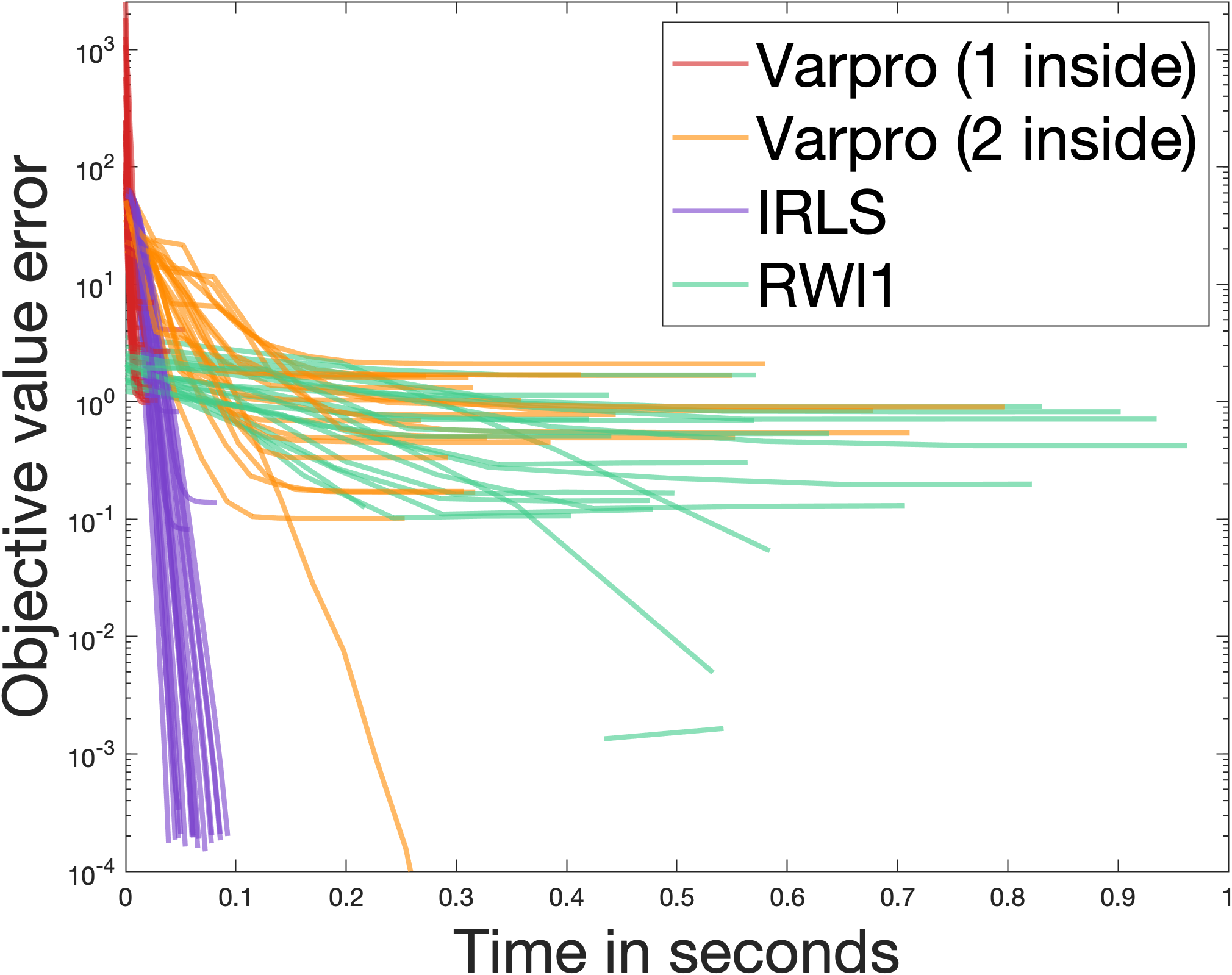}&
\includegraphics[width = 0.31\linewidth]{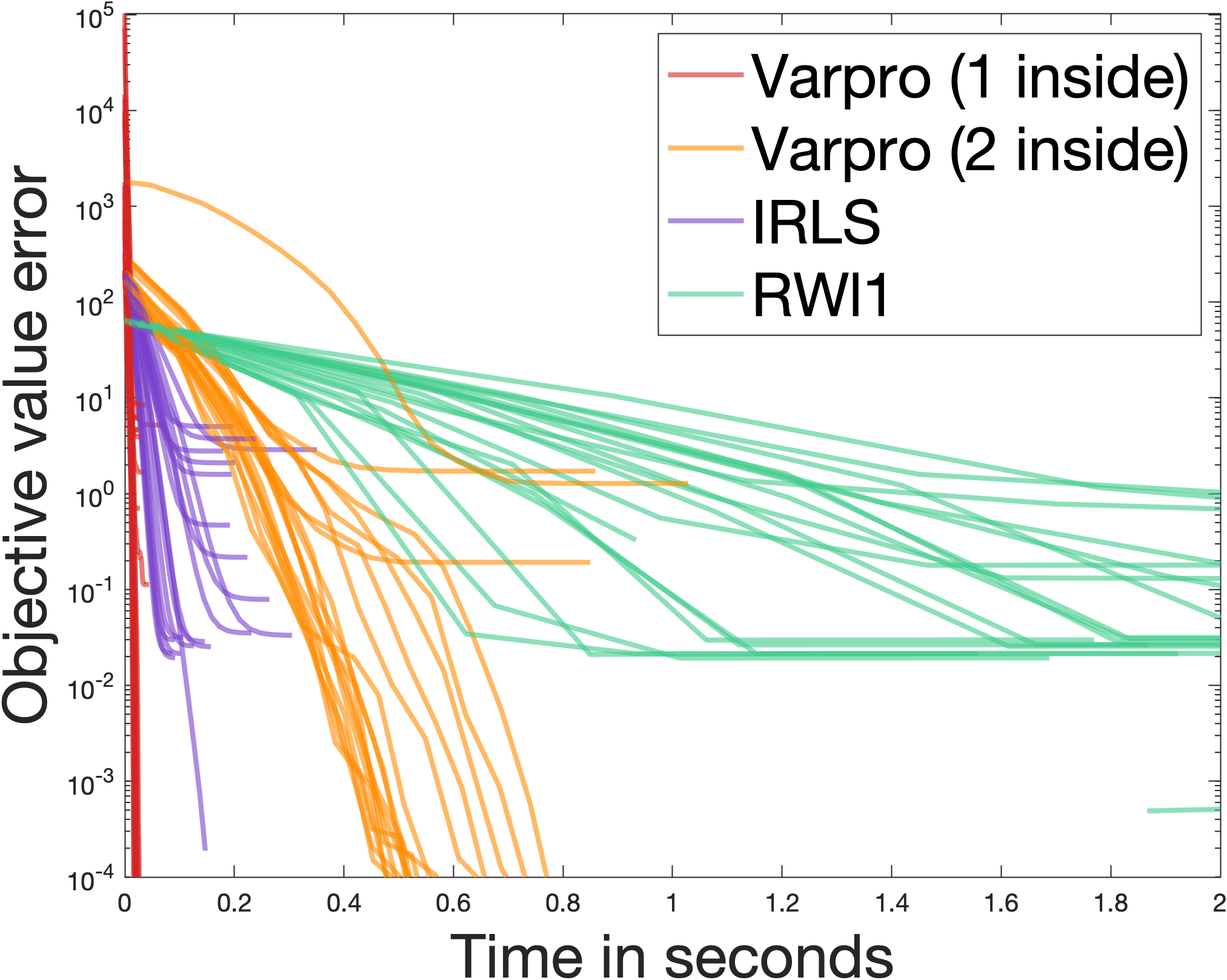}&
\includegraphics[width = 0.31\linewidth]{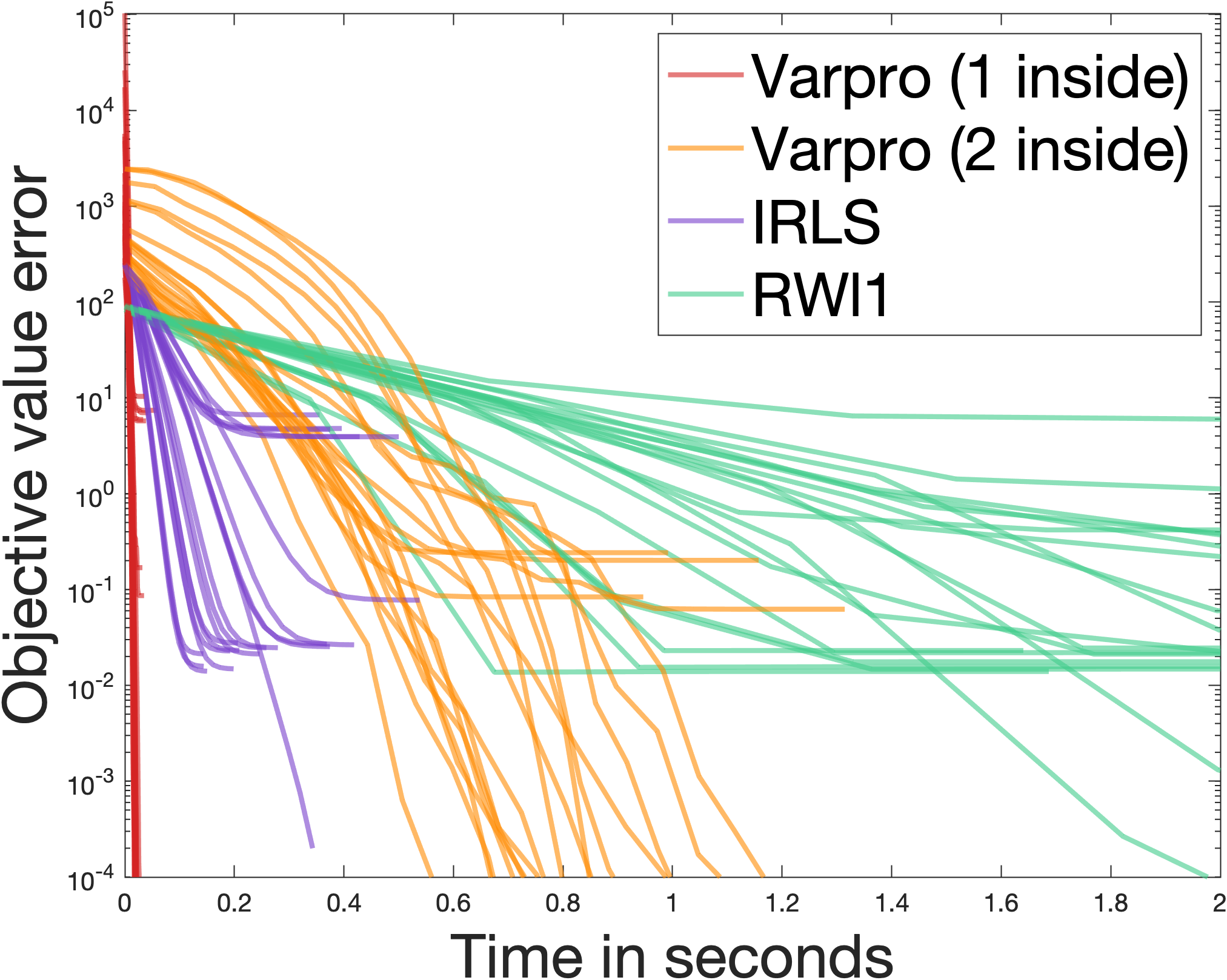}
\end{tabular}
\caption{Group-$\ell_{2/3}$ Lasso over 20 random instances. The regularisation parameter is 0.1.  The graphs show a plot of objective error against computational time.\label{fig:ncvx_time} }
\end{figure}

% !TEX root = ../SINUM-VarPro.tex

\section*{Conclusion}

We have presented a generic and versatile class of optimization methods, which can cope with a wide range of non-smooth losses and regularization functionals. 
An appealing feature of these approaches is that they rely on smooth optimization technics and can thus leverage standard efficient solvers such as quasi-Newton. 
On the theoretical side, we highlighted that handling generalized sparse regularizers such as total variation is more intricate than the Lasso case, and in particular differentiability requires a greater care. We also draw connexions with mirror-descent methods, which leads to constant independent of the grid-size. Unfortunately, although we were able to partly lift difficulties due to non-convexity, a full convergence analysis is still beyond reach with our proof technics.

\section*{Acknowledgments}

The work of G. Peyr\'e was supported by the French government under management of Agence Nationale de la Recherche as part of the ``Investissements d'avenir'' program, reference ANR19-P3IA-0001 (PRAIRIE 3IA Institute) and by the European Research Council (ERC project NORIA). 
 % !TEX root = ../SINUM-VarPro.tex

 \appendix

 \section{Envelope theorem}

 \begin{defn}
 We say that $f:\RR^n\to \bar \RR$ is strictly continuous at $\bar x$ if there is a neighbourhood $\Nn_{\bar x}$  around $\bar x$ such that 
 $$
 \abs{f(x)-f(x')} \leq \norm{x-x'}, \qquad \forall x,x' \in \Nn_{\bar x}
 $$
 We say that $f$ is strictly differentiable  at $\bar x$ if $f(\bar x)$ is finite and there is a vector $v$, which is the gradient $\nabla f(\bar x)$, such that  
 $$
 \lim_{x,x' \to x} \frac{f(x')-f(x) - \dotp{v}{x'-x}}{\norm{x-x'}} = 0, \quad \text{with} \quad x'\neq x.
 $$
 \end{defn}
 Note that strict differentiability is stronger that simply differentiability and if there is an open set $\Oo$ on which $f$ is finite, then $f$ is strictly differentiable on $\Oo$ is equivalent to $f$ is $\Cc^1$ on $\Oo$ \cite[Cor 9.19]{rockafellar2009variational}
 
\begin{defn}
Let $\Oo\subset \RR^n$ be an open set. A function $f:\Oo\to \RR$ is said to be  lower-$\Cc^1$ on $\Oo$, if for all $x\in \Oo$, there exists a neighbourhood $\Nn$ of $x$ with the representation $$ f(x) = \max_{t\in T} \phi(t,x)$$
in which the functions $\phi(t,\cdot)$ are $\Cc^1$, the set $T$ is compact, and $\phi(t,x)$ and $\nabla_x \phi(t,x)$ depend continuously on $(t,x)\in T\times V$.
\end{defn}

\begin{thm}\cite[Theorem 10.31]{rockafellar2009variational} \label{thm:rockafellar}
Suppose $f:\Oo \to \RR$ is lower-$\Cc^1$ on an open set $\Oo$, then $f$ is strictly differentiable on a set $D$ with $\Oo\setminus D$ negligible. Moreover, $f$ is differentiable at $x$ if
$
\enscond{\nabla_x \phi(t,x)}{t\in T(x)} 
$ is single-valued.
\end{thm}
\begin{proof}
This is a direct consequence of  \cite[Thm 10.31]{rockafellar2009variational} where it is shown that if $f$ is a lower-$\Cc^1$ function, then:
\begin{itemize}
\item It is strictly continuous and regular on $\Oo$ and is semidifferentiable. 
\item $f$ is strictly differentiable on a set $D$ with $\Oo\setminus D$ negligible. 
\item $\partial f(x) = \mathrm{con}\enscond{\nabla f_t(x)}{t\in T(x)}$
\end{itemize}
Finally, by  \cite[Theorem 9.18]{rockafellar2009variational},  strict differentiability at $x$ is equivalent to $f$ strictly continuous and regular with $\partial f(x)$ single-valued.
%
% It is strictly differentiable where $\partial f(\bar x)$ is a singleton. [consequence of Thm 9.18 and Thm 10.31. By Thm 9.18, strict differentiability at $x$ is equivalent to $f$ strictly continuous and regular with$\partial f(x)$ single-valued. Thm 10.31 tells us that $\partial f(x) = \mathrm{con}\enscond{\nabla f_t(x)}{t\in T(x)}$ where $T(x) = argmax_{t\in T} f_t(x)$].

\end{proof}

We recall the notion of semiderivatives (note that this is different from directional derivatives, since we take the limit along all $w'$ converging to $w$.
\begin{defn} \cite[Def 7.20]{rockafellar2009variational}
Let $f:\RR^n\to \bar \RR$ and suppose thst $\bar x\in \mathrm{dom}(f)$. If
$$
\lim_{\substack{\tau \to 0\\ w'\to w}} \frac{f(\bar x+ \tau w')-f(\bar x)}{\tau}
$$
exists,  it is is the semiderivative  of $f$ at $\bar x$ for $w$, and $f$ is semidifferentiable at $\bar x$ for $w$. If this holds for every $w$, $f$ is semidifferentiable at $\bar x$
\end{defn}
In  \cite[Cor 7.22]{rockafellar2009variational}, we have the following result:
\begin{prop}\label{prop:semiderivative}
A function $f:\RR^n\to \bar \RR$ is differentiable at $\bar x$, a point with $f(\bar x)$ finite, if and only if $f$ is semidifferentiable at $\bar x$ and the semiderivative for $w$ depends linearly on $w$.
\end{prop}
In certain cases, we can show that the inner problem of our VarPro function can be restricted to a compact set and directly apply Theorem \ref{thm:rockafellar} to deduce differentiability. However, in cases where this is not possible, we will directly compute the semiderivative and apply Proposition \ref{prop:semiderivative}.
 \section{Standard Gradient descent results}

\begin{lem}[Gradient bound]\label{lem:grad_bound_desc_lemma}
Suppose that $f\in C^{1,1}$ and $\nabla f$ being $L_2$-Lipschitz with respect to the Euclidean norm.
Let $$w^{k+1} = w^k -\frac{1}{L_2} \nabla f(w^k).$$
For all $T\in\NN$ and $j\in [1,2]$,  then for $C\eqdef  2(f(w^0) - f(w^*)) L_2$, 
$$
\sum_{k\leq T} \norm{\nabla f(\xp^k)}^2 \leq   C
\qandq
\min_{k\leq T}\norm{\nabla f(\xp^k)}^2 \leq \frac{C  }{T}.
$$
\end{lem}
 \begin{proof}[Proof of Lemma \ref{lem:grad_bound_desc_lemma}]

By Taylor expansion, for all $v,w$,
$$
f(v) \leq f(w) + \nabla f(w)^\top (v-w) + \frac{L_2}{2}\norm{v-w}^2.
$$ 
So,\begin{equation}\label{eq:guarantee_descent}
f(w^{k+1})\leq f(w^k) - \frac{1}{2L_2} \norm{\nabla f(w^k)}^2,
\end{equation}
Note also that
$$
\norm{\nabla f(w^k)}^j \leq  \pa{(2L_2)(f(w^k) - f(w^{k+1})) }^{j/2}
$$
So, for $j\in [1,2]$,
\begin{align*}
\sum_{k\leq T} \norm{\nabla f(w^k)}^j &\leq (2L_2)^{\frac{j}{2}} \sum_{k\leq T} \pa{f(w^k) - f(w^{k+1}) }^{j/2} \\
&\leq  (2L_2)^{\frac{j}{2}} \pa{ \sum_{k\leq T} \pa{f(w^k) - f(w^{k+1}) } }^{j/2} T^{(2-j)/2}\\
&\leq (2L_2)^{\frac{j}{2}}  \pa{f(w^0) - f(w^*) }^{j/2} T^{(2-j)/2}
\end{align*}
and 
$$
\min_{k\leq T}\norm{\nabla f(w^k)}^2 \leq \frac{2L_2}{T} \pa{ f(w^0) - f(w^*)}.
$$

\end{proof}

\section{ADMM and Primal-Dual}\label{sec:admm}

 In the following, we compare against ADMM 
 \cite{boyd2011distributed} and Primal-Dual splitting \cite{Chambolle2011}.
\paragraph{ADMM}
%For
%\begin{align*}
%\min_{\xp} \frac{1}{2\lambda} \norm{A\xp - y}_2^2 + \norm{z}_1 \qwhereq z - L \xp = 0.
%\end{align*} 
ADMM \cite{boyd2011distributed} seeks to minimise for $\tau>0$,
\begin{align*}
\min_{\xp, z} \max_{\psi} \frac{1}{2\lambda} \norm{A\xp - y}_2^2 + \norm{z}_1 +\dotp{\psi}{ z - L \xp } + \frac{\tau}{2} \norm{z  - D\xp}_2^2
\end{align*} 
by the iterations
\begin{align*}
&\xp_{k+1} = \argmin_\xp  \frac{1}{2\lambda} \norm{A\xp - y}_2^2  - \dotp{\psi_k}{  L \xp } + \frac{\tau}{2} \norm{z_k  - D\xp}_2^2\\
& z_{k+1} = \argmin_z  \norm{z}_1 +\dotp{\psi}{ z } + \frac{\tau}{2} \norm{z  - D\xp}_2^2\\
& \psi_{k+1} = \psi_k + \tau (z_k - L \xp_k )
\end{align*}
The update on $\xp_k$ requires solving the linear system
$$
(A^\top A  +\lambda \tau D^\top D) \xp  = A^\top y +\lambda D^\top \psi_k +\lambda \tau D^\top z_k
$$
For this step, we carry out a reordering of the columns and a cholesky factorisation which is reused throughout the iterations when carrying out the matrix inversion.

\paragraph{Primal-Dual} 
 Primal-Dual splitting \cite{Chambolle2011} solves
% $$
% \min_\xp \sup_z \dotp{D\xp}{z} - F^*(z) + \frac{1}{2\lambda}\norm{A\xp - y}_2^2
% $$
 $$
 \min_\xp \sup_z \dotp{K\xp}{z} - F^*(z) + G(\xp)
 $$
The iterations are, for $\sigma,\tau>0$ such that $\sigma \tau \leq 1/\norm{K}^2$,
\begin{align*}
z^{n+1} = \Prox_{\sigma F^*} (z^n + \sigma K \bar x^n)\\
x^{n+1}=\Prox_{\tau G}(x^n - \tau K^* z^{n+1})\\
\bar x^{n+1} = x^{n+1} + \theta (x^{n+1} - x^n)
\end{align*}

For $\min_x  \norm{Dx}_1 + \frac{1}{2\lambda}\norm{Ax - y}_2^2$, we can take $K=D$, $F(z) = \norm{z}_1$ and $G(x) = \frac{1}{2\lambda}\norm{Ax - y}_2^2$. In this case,
$$
\Prox_{\tau G}(z) = \pa{\Id + \frac{\tau}{\lambda} A^\top A}^{-1}(\frac{\tau}{\lambda} A^\top y + z).
$$
For the case where $X$ is a masking operation, the matrix inversion in the update of $\xp_k$ is a simple rescaling operation. Where the matrix $X$ does not admit an efficient inversion formula (e.g. random Gaussian matrices), we carry out once a cholesky factorisation which is used throughout the iterations.

For $\min_x \norm{Dx}_1 + \frac{1}{\lambda}\norm{Ax- y}_1$ where $A\in \RR^{m\times n}$, $D\in \RR^{p\times n}$ and $y\in \RR^m$, we write this as
\begin{align*}
&\min_{x\in\RR^n,z\in \RR^{p+m}} \norm{z}_1 \quad \text{s.t.}\quad \begin{pmatrix}
Dx\\ Ax-y
\end{pmatrix} = z\\
&\min_{x\in\RR^n,z\in \RR^{p+m}} \max_{\xi \in \RR^{p+m}} \norm{z}_1 +\dotp{K \begin{pmatrix}
x\\z
\end{pmatrix} }{\xi} - \dotp{\begin{pmatrix}
0\\ y
\end{pmatrix}}{\xi}
\end{align*}
where 
$$
K = \begin{pmatrix}
D & -\Id_p & 0\\
A & 0 &-\Id_m
\end{pmatrix}.
$$
We let $G((x,z)) = \norm{x}_1 + \norm{z}_1$ and $F^*(\xi) =  \dotp{\begin{pmatrix}
0\\ y
\end{pmatrix}}{\xi}$.

\section{Quadratic variational form for matrices}\label{app:matrices}

We consider the following multitask problem \cite{geer2016chi}, for $A\in \RR^{m\times n}$, $Y\in \RR^{m\times T}$, and $X\in\RR^{n\times T}$
\begin{equation}\label{eq:multitask}
\min_{X} \frac{1}{\lambda} \norm{AX - Y}_*  + \norm{X}_{1,2}, 
\end{equation}
where $\norm{X}_{1,2} = \sum_{i} \norm{X_i}_2$ with $X_i$ denoting the $i$th row of $X$. The loss function is the nuclear norm (sum of the singular values of  a matrix) and this is known to  have  a variational form for $X\in \RR^{n_1\times n_1}$
\begin{align*}
	\norm{X}_* &= \min_{\Sigma\in \SS_+^{n_1\times n_1}} \frac12\dotp{\Sigma X}{X} + \frac12\tr(\Sigma)\\
& =  \min_{U\in \RR^{n_1\times n_1}, V\in \RR^{n_1\times n_2}}\enscond{\frac12\norm{U}_F^2+\frac12\norm{V}_F^2}{X = UV}.
\end{align*}
By making use of the quadratic variational forms for the group $\ell_1$ and nuclear norm, we can rewrite
\eqref{eq:multitask} as
\begin{align*}
&\min_{v\in \RR^n,U\in \RR^{n\times T}} \min_{\substack{W\in \RR^{m\times m},\\Z\in \RR^{m\times T}}} \enscond{
		\frac{\norm{v}^2}{2} \!+\! \frac{\norm{U}_F^2}{2} \!+\! \frac{\la}{2} \norm{W}_F^2 \!+\! \frac{\la}{2} \norm{Z}_F^2 }{A(\diag(v)U)-Y = WZ
	}\\
&=\min_{v,W} f(v,W)\eqdef \max_{\alpha\in\RR^{m\times T}}\frac{1}{2} \norm{v}^2 - \frac{1}{2} \norm{v\odot A^\top \alpha}_F^2 +\frac{\la}{2} \norm{W}_F^2 - \frac{1}{2\la} \norm{W^\top \alpha}_F^2-\dotp{\al}{Y}.
\end{align*}
Note that the inner problem is a least squares problem with solution $\alpha$ satisfying
$$
(A\diag(v^2)A^\top +\frac{1}{\la}WW^\top)\alpha = -Y.
$$
The outer problem has gradient
$$
\partial_v f(v,W) = v - \pa{v_i \norm{A_i^\top \alpha}^2}_{i=1}^n\qandq \partial_W f(v,W) = \lambda W -\frac{1}{\la}\alpha \alpha^\top W.
$$

\bibliographystyle{plain} 
\bibliography{mybiblio}

\end{document}